\documentclass{amsart}
\usepackage[margin=1in]{geometry}
\usepackage{amsmath}
\usepackage{amssymb}
\usepackage{amsthm}
\usepackage{amscd}
\usepackage{enumerate}

\allowdisplaybreaks

\usepackage{graphicx}
\usepackage{amsmath,amsthm,amsfonts,amscd,amssymb,comment,eucal,latexsym,mathrsfs}
\usepackage{stmaryrd}
\usepackage[all]{xy}

\usepackage{epsfig}

\usepackage[all]{xy}
\xyoption{poly}
\usepackage{fancyhdr}
\usepackage{wrapfig}
\usepackage{epsfig}

\theoremstyle{plain}
\newtheorem{thm}{Theorem}[section]
\newtheorem{prop}[thm]{Proposition}
\newtheorem{lem}[thm]{Lemma}
\newtheorem{cor}[thm]{Corollary}

\theoremstyle{definition}
\newtheorem{defn}[thm]{Definition}
\theoremstyle{remark}

\newtheorem{example}{Example}

 \def\C{{\mathbb{C}}} \def\F{{\mathbb{F}}}   \def\N{{\mathbb{N}}}    \def\R{{\mathbb{R}}}  \def\Z{{\mathbb{Z}}}

\newcommand\vre{\varepsilon}
\newcommand\varpesilon{\vre}
\newcommand\varespilon{\vre}

\newcommand\Ad{\operatorname{Ad}}

\newcommand\Ball{{\operatorname{Ball}}}

\newcommand\Fix{{\operatorname{Fix}}}
\newcommand\Hamm{{\operatorname{Hamm}}}

\newcommand\Hom{\operatorname{Hom}}

\newcommand\id{\operatorname{id}}

\newcommand\Map{{\operatorname{Map}}}
\newcommand\Meas{{\operatorname{Meas}}}

\newcommand\Proj{\operatorname{Proj}}
\newcommand\Prob{\operatorname{Prob}}

\renewcommand\Re{\operatorname{Re}}
\newcommand\rea{\operatorname{Re}}

\newcommand\supp{\operatorname{supp}}
\newcommand\Sym{\operatorname{Sym}}

\newcommand\Sub{\operatorname{Sub}}
\newcommand\Span{\operatorname{span}}

\newcommand\topo{\operatorname{top}}

\newcommand\tr{\operatorname{tr}}
\newcommand\Tr{\operatorname{Tr}}

\newcommand{\actson}{\curvearrowright}
\newcommand{\actons}{\curvearrowright}
\newcommand{\acston}{\actson}

\newcommand{\ip}[1]{\langle #1 \rangle}

\begin{document}
\title{Local weak$^{*}$-Convergence, algebraic actions, and a max-min principle}      % Enter your title between curly braces
\author{Ben Hayes}\thanks{The author acknowledges support from NSF grants DMS-1600802 and DMS-1827376.}
\address{University of Virginia\\
          Charlottesville, VA 22904}
\email{brh5c@virginia.edu}
\date{\today}
\maketitle

\begin{abstract}
We continue our study of when topological and measure-theoretic entropy agree for algebraic action of sofic groups. Specifically, we provide a new abstract method to prove that an algebraic action is strongly sofic. The method is based on passing to a ``limiting object" for sequences of measures which are asymptotically supported on topological microstates. This limiting object has a natural poset structure that we are able to exploit to prove a max-min principle: if the sofic approximation has ergodic centralizer, then the largest subgroup on which the action is a local weak$^{*}$-limit of measures supported on topological microstates is equal to the smallest subgroup which absorbs all topological microstates. We are able to provide a version for the case when the centralizer is not ergodic. We give many applications, including show that for residually finite groups completely positive (lower) topological entropy (in the presence) is equivalent to completely positive (lower) measure entropy in the presence.

\end{abstract}

\tableofcontents

\section{Introduction}

This paper furthers the study of the entropy theory on algebraic actions of sofic groups along the lines we previously developed in \cite{Me12, Me13}. Given a countable, discrete, group $G,$ an \emph{algebraic action} of $G$ is an action $G\actons  X$ by continuous automorphisms of a compact, metrizable group $X.$ We may view the study of these actions a part of topological dynamics, by viewing this as an action of $G$ by homeomorphisms of a compact, metrizable space. We may also view these actions as part of the study of ergodic theory, by giving $X$ its Haar measure, denoted $m_{X},$ and thinking of the action $G\actson X$ as a probability measure-preserving action. A major theme of the study  of algebraic actions is that there is a heavy interaction between the topological dynamics of the action and the measure-theoretic dynamics of the action.

For this paper, we are primarily interested in connecting the topological entropy to the measure entropy of algebraic actions. In the case that  $G$ is amenable, we have a entirely satisfactory understanding of the connections between the entropy theory of the algebraic action as a topological dynamical system and the algebraic action as a measure-theoretic dynamical system. For example, when $G$ is amenable and $G\actson X$ is an algebraic action we know:
\begin{itemize}
\item the topological and measure entropy (with respect to the Haar measure) of $G\actson X$ agree \cite{Den}, \label{I:equality of entropy}
\item the maximal measure-theoretic factor action with zero measure entropy (i.e. the Pinsker factor) is of the form $G\actson X/Y$ for some closed, normal, $G$-invariant subgroup $Y$ of $X$, and this is also the topological Pinsker factor \cite{Berg, ChungLi},
\item  $G\actson X$ has completely positive topological entropy if and only if $G\actson (X,m_{X})$ has completely positive measure entropy \cite{Berg, ChungLi}.
\end{itemize}

Thus the classical objects of entropy theory (measure entropy, topological entropy, topological/measure-theoretic Pinsker factors) are the same as viewed from either the topological or measure-theoretic lens, and in the case of the Pinsker factor, we know that the Pinsker factor retains the algebraic structure of $G\actson X$ (i.e. it is another algebraic action).

Given the recent extension of entropy theory to sofic groups by Bowen, Kerr-Li \cite{Bow, KLi} it is natural to try and generalize this to the case when $G$ is sofic. Because of our recent work, we know of general conditions which guarantee that there are strong connections between the entropy theory of $G\actson X$ as a topological dynamical system and the entropy theory of $G\actson X$ as a probability measure-preserving action. This general condition is that of strong soficity (see the discussion before Corollary \ref{C: equivalence cpe ultrafilter} for the definition of strong soficity), and under this condition we have the following theorem which is a combination of work in \cite{Me12, Me13}.

\begin{thm}[Theorem 1.1 in \cite{Me12}, Corollary 1.4 in \cite{Me13}]\label{T:known results intro}
Suppose $G$ is a countable, discrete, sofic group with sofic approximation $\sigma_{k}\colon G\to S_{d_{k}}.$ Let $G\actson X$ be an algebraic action, and suppose that $G\actson (X,m_{X})$ is strongly sofic. Then:
\begin{itemize}
\item The topological and measure entropy of $G\actson X$ with respect to $(\sigma_{k})_{k}$ agree, i.e. $h_{(\sigma_{k})_{k},\topo}(G\actson X)=h_{(\sigma_{k})_{k}}(G\actson (X,m_{X})),$
\item The outer Pinsker factor of $G\actson (X,m_{X})$ is algebraic, i.e. of the form $G\actson (X/Y,m_{X/Y})$ where $Y$ is a closed, normal, $G$-invariant subgroup,
\item The action $G\actson (X,m_{X})$ has completely positive measure entropy in the presence if and only if $G\acston X$ has completely positive topological entropy in the presence.
\end{itemize}

\end{thm}

The reader should not worry too much about the difference between ``outer Pinsker factor" and ``Pinsker factor", or between ``completely positive measure/topological entropy in the presence" and ``completely positive measure/topological entropy" in the passage from amenable groups to sofic groups. It suffices to say simply that ``Outer Pinsker factor" is the correct analogue of ``Pinsker factor" in the nonamenable case, and ``completely positive measure/topological entropy in the presence" is the correct analogue of ``completely positive measure/topological entropy" in the nonamenable case, as the naive generalizations of  `Pinsker factor" and ``completely positive measure/topological entropy" to the nonamenable setting do not   have the right permanence properties (see, e.g., \cite{Me9,Me13}).

 We mention that strong soficity of $G\actson (X,m_{X})$ has been verified in every case where the measure-theoretic entropy has been computed and is not $-\infty$ (see e.g. \cite{BowenEntropy, BowenLi, Me5, GabSew}). However one major issue with Theorem \ref{T:known results intro} is that the verification that an action is strongly sofic can be rather involved. For instance, the verification of strong soficity for balanced algebraic actions (i.e. an algebraic action which is the dual of the natural $G$ action on $\Z(G)^{\oplus n}/\Z(G)^{\oplus n}A$ for some $A\in M_{n}(\Z(G))$)  carried out in Section 5 of \cite{Me5} is the most difficult and technical portion of an already difficult and technical paper.

Part of the reason for these technical proofs it that by their very nature they  involve rather explicit computations and often messy analytic arguments. These concrete methods give you lots of insight into the results and these are good, if involved, proofs. However, a major drawback of them is that the actions to which these methods apply have to take a fairly specific form.
We would like to simply put general conditions  (e.g. mixing, ergodicity, complete positive entropy etc.) on an action that do not force the action to be of a very specific type (e.g. a balanced algebraic action), and one can simply verify that an algebraic action is strongly sofic by just checking these general conditions. A major goal of this paper is to address these issues, and we do this primarily by introducing some general ``machinery" that one is able to apply as a black box to show that certain dynamical properties of the action imply strong soficity. We postpone discussion of our methods to later in the introduction and illustrate the utility of our methods by giving the major applications of our main result first. For example, Theorems \ref{T:res finite case intro}-\ref{T:topological conj intro} are precise examples of when we are able to deduce strong soficity from relatively general dynamical  conditions.

\begin{thm}\label{T:res finite case intro}
Let $G$ be a residually finite group, and let $(G_{n})_{n=1}^{\infty}$ be a decreasing sequence of finite-index, normal subgroups with $\bigcap_{n=1}^{\infty}G_{n}=\{1\}.$  Let $G\actson X$ be an algebraic action and suppose that $G\actson (X,m_{X})$ is sofic with respect to $(G_{n})_{n}.$ If $G\actson (X,m_{X})$ is ergodic, then $G\actson (X,m_{X})$ is strongly sofic with respect to $(G_{n})_{n}.$ In particular:
\begin{enumerate}[(i)]
\item The topological entropy of $G\actson X$ and the measure entropy of $G\actson (X,m_{X})$ agree, i.e. $h_{(G_{n})_{n},\topo}(G\actons X)=h_{(G_{n})_{n}}(G\actson (X,m_{X})),$
\item The Outer Pinsker factor of $G\actson (X,m_{X})$ with respect to $(G_{n})_{n}$ is algebraic: i.e. it is of the form $G\actson (X/Y,m_{X/Y})$ for a closed, normal, $G$-invariant subgroup $Y$ of $X,$
\item $G\actson X$ has completely positive topological entropy in the presence if and only if $G\actson (X,m_{X})$ has completely positive measure entropy in the presence.
\end{enumerate}

\end{thm}

The significant difference between Theorem \ref{T:res finite case intro} is that, in the residually finite case, once the algebraic is sofic and ergodic, one obtains \emph{for free} that it is strongly sofic, and therefore we automatically deduce all these nice corollaries we have one equality of entropy, structure of the Pinsker etc. \emph{without} having to do any additional technical and involved proof.
At this point we mention that there are counterexamples to equality of measure and topological entropy of algebraic actions, but in all known counterexamples  it is true that $G\actson (X,m_{X})$ has entropy $-\infty.$ Experts in the subject view these counterexamples as unsatisfactory, and the question of whether the topological and measure entropy of algebraic actions agree when the measure entropy is not $-\infty$ is still open. The above theorem gives an affirmative answer when $G$ is residually finite and the sofic approximation comes from a decreasing chain of finite-index, normal, subgroups with trivial intersection, and the action is ergodic. The assumption that $G$ is residually finite is  a special case, but it is a common ``test case" to check if results should be true for general sofic groups. Thus Theorem \ref{T:res finite case intro} should be regarded as strong positive evidence the topological and measure entropy agree for algebraic actions with nonnegative measure entropy.

\begin{thm}\label{T:cpe intro}
Let $G$ be a residually finite group, and let $(G_{n})_{n=1}^{\infty}$ be a decreasing sequence of finite-index, normal subgroups. Let $G\actson X$ be an algebraic action. Then $G\actson X$ has completely positive lower topological entropy in the presence if and only if $G\actson (X,m_{X})$ has completely positive lower measure entropy in the presence. In fact, if $G\actson X$ has completely positive lower topological entropy in the presence, then $G\actson (X,m_{X})$ is strongly sofic with respect to $(G_{n})_{n}.$

\end{thm}
As in \cite{Me13}, the reader should read ``completely positive topological/measure entropy in the presence" as the correct generalization of completely positive topological/measure entropy from the amenable case to the nonamenable case.
Again there is an interesting open question lurking in the background, which is whether completely positive lower topological entropy in the presence is equivalent to completely positive lower measure entropy in the presence for algebraic actions. This would be a perfect analogue of \cite{Berg, ChungLi} in the nonamenable case. Theorem \ref{T:cpe intro} should be regarded as strong evidence that this conjecture is true in general, since it is true in the residually finite case.

We mention here that recent results of Alpeev \cite{AlpeevMaximal} imply that every sofic group has a sofic approximation so that every \emph{sofic} algebraic action is strongly sofic with respect to that sofic approximation. One important difference between Theorem \ref{T:cpe intro} and the results of Alpeev is that Alpeev has to  assume that $G\actson (X,m_{X})$ is sofic, in order to deduce that completely positive topological entropy and completely positive measure-theoretic entropy agree with respect to the sofic approximation he constructs.  Here we obtain as a \emph{corollary} that $G\actson (X,m_{X})$ is sofic with respect to $(\sigma_{k})_{k}$ provided $G\actson X$ has completely positive lower topological entropy in the presence. Soficity of the action $G\actson (X,m_{X})$ is far from obvious from the assumption that $G\actson X$ has completely positive lower topological entropy in the presence. Additionally, Theorem \ref{T:cpe intro} applies to a specific and concrete sofic approximation, whereas the results of Alpeev only apply to a sofic approximation he constructs by an abstract method. To the best of our knowledge, there isn't a concrete example of a nonamenable sofic group $G,$ and a concrete sofic approximation of $G$ which satisfies the hypotheses Alpeev requires in order to apply his methods.

We remark on one other interesting aspect of Theorem \ref{T:cpe intro}, which is that it the proof that $G\actson (X,m_{X})$ is sofic using Theorem \ref{T:cpe intro} is rather different than most other proofs of soficity. Most proofs of soficity of actions go as follows: fix a sofic approximation $\sigma_{k}\colon G\to S_{d_{k}}$ of $G,$ and let $G\actson (X,\mu)$ be a probability measure-preserving action, without loss of generality we may assume $X$ is compact and the actions is by homeomorphisms. One then chooses a ``random map" $\phi\colon \{1,\dots,d_{k}\}\to X$ (with respect to some sequence of probability measures $\mu_{k}$ on $X^{d_{k}}$), and shows that with high probability this maps is ``almost equivariant and almost measure-preserving". This often requires heavily involved calculations, and delicate probabilistic arguments (see e.g. (see  \cite[Theorem 8.1]{Bow}, \cite[Theorem 8.2]{GabSew}).   The proof of soficity using Theorem \ref{T:cpe intro} is quite different in this regard, relying on an abstract apparatus.   Difficult estimates and inequalities are not required, and the results follow from fairly general principles. This is why Theorem \ref{T:cpe intro} applies to algebraic actions which satisfy an abstract dynamical assumption (complete positive entropy), instead of only to a fairly specific class of actions as in previous results (the main exception being the results of \cite{GabSew}). The main other argument for soficity that may be regarded as essentially an entirely ``soft" proof is that of Popa in \cite[Theorem 0.3]{PoppArg}, which involves an ultraproduct argument. As we shall reveal later, our proofs also involve ultraproduct analysis, though in quite a different manner than Popa.
%
%
%Using the abstract machinery we develop, we can show that ergodic algebraic actions of residually finite groups sofic under quite general hypothesis.
%
%
%
%
%\begin{thm}\label{T:dense periodic points intro}
%Let $G$ be a residually finite group, and let $G\actson X$ be an algebraic action and suppose that $G\actson (X,m_{X})$ is ergodic. Let $(G_{n})_{n=1}^{\infty}$ be a decreasing sequence of finite-index, normal subgroups of $X$ with $\bigcap_{n=1}^{\infty}G_{n}=\{1\}.$  If the union of the $G_{n}$-periodic points is dense in $X,$ then $G\actson (X,m_{X})$ is strongly sofic with respect to  $(G_{n})_{n}.$ In particular, items $(i)$-$(iii)$ of Theorem \ref{T:res finite case intro} hold.
%
%\end{thm}
%
%The preceding theorem automatically has the following Corollary.
%
%\begin{cor}\label{C:soficity with dense periodic points}
%Let $G$ be a residually finite group, and $G\actson X$ an algebraic action. If $X$ has a dense set of periodic points for $G\actson X,$ then $G\actson (X,m_{X})$ is sofic.
%
%\end{cor}
%
%As is the case when $G\actson X$ has completely positive topological entropy in the presence, the deduction that $G\actson (X,m_{X})$ is sofic in Corollary \ref{C:soficity with dense periodic points}  does not rely on a delicate probabilistic argument, and we instead apply abstract machinery to quickly deduce that the action is sofic.
For each of Theorems \ref{T:res finite case intro},\ref{T:cpe intro}, we are able to deduce strong soficity by only assuming something on the topological structure of the action $G\actson X$ ( having completely positive topological entropy). In certain cases, we can give an indication as to why one would expect to deduce strong soficity from structure of the topological dynamical system $G\acston X.$ We show that strong soficity of $G\actson (X,m_{X})$ is, in fact, a topological conjugacy invariant amount the class of algebraic actions, provided that $G\actson (X,m_{X})$ is ergodic.

\begin{thm}\label{T:topological conj intro}
Let $G$ be a countable, discrete, sofic group with sofic approximation $\sigma_{k}\colon G\to S_{d_{k}}.$ Let $G\actson X,G\actson Y$ be two algebraic actions, and suppose that $G\actson (X,m_{X}),G\actson (Y,m_{Y})$ are ergodic. If $Y$ is a topological $G$-factor of $X,$ and $G\actson (X,m_{X})$ is strongly sofic, then $G\actson (Y,m_{Y})$ is strongly sofic. This is true even when the factor map $X\to Y$ is not a group homomorphism.

\end{thm}
At this point, we will explain briefly the abstract machinery alluded to after Theorem \ref{T:cpe intro} which underlies all the proofs of the paper.
To discover this machinery, we first reduce our problems from showing strong soficity, to local weak$^{*}$-convergence (see \ref{D:important notions for the paper} for the definition of local weak$^{*}$-convergence). The main fact we use is the following: if $G\actson X$ is an algebraic action, and $G\actson (X,m_{X})$ is ergodic, then $G\actson (X,m_{X})$ is strongly sofic if and only if there is a sequence of measures $\mu_{k}\in \Prob(X^{d_{k}})$ which are asymptotically supported on topological microstates and locally weak$^{*}$-converge to $m_{X}$ (by \cite[Corollary 5.7]{AustinAdd} and \cite[Corollary 2.14]{Me12}). Thus we aim to show that existence of a sequence of measures $\mu_{k}\in \Prob(X^{d_{k}})$ which are asymptotically supported on topological microstates and locally weak$^{*}$-converge to $m_{X}.$
 In fact, all of Theorems  \ref{T:res finite case intro}-\ref{T:res finite case intro}, apply to the situation when $G\actson (X,m_{X})$ is not ergodic, but with the conclusion of strong soficity replaced by the existence of a sequence of measures asymptotically supported on topological microstates and locally weak$^{*}$-converge to the Haar measure.

After reducing to the study of local weak$^{*}$-convergence, we apply an ultrafilter/ultraproduct analysis. An ultrafilter is an abstract object that allows one to take generalized limits (called ultralimits) of any bounded sequence of complex numbers in such a way that preserves the usual order and algebraic structure of $\R,\C$ (e.g. preservation of inequalities, sums, products, etc).
The existence of a free ultrafilter on the natural numbers is an elementary consequence of Zorn's lemma. For the reader uninitiated into usage of ultrafilters, the half-page material contained on ultrafilter in Section 3.1 of \cite{KapovichLeeb} is already enough for our purposes, which can be learned by the dedicated reader in a matter of days, or arguably hours. Our ultrafilter arguments are thus elementary, in the sense that they are easy consequences of the usual axioms of set theory, and require minimal background and preparation to understand. The use of ultrafilters is common and well-accepted in the study of sofic groups (see \cite{CapLup, ESZ2, pestov_kwiatkowska_2012, LyonsThom, Pest, ThomDiophantine}) and geometric group theory in general (see \cite{DrutuSapir, RoeCoarse}).  Ultrafilter have already been used several times in the study of measure entropy for actions of nonamenable groups (see \cite{BowenGroupoid, KerrLi2, LiLiang2, LiangPeters, Me7, Me9}), and our paper is another entry into this tradition.

Given an ultrafilter $\omega$ on the natural numbers, and a sequence $(d_{k})_{k}$ of natural numbers, one can construct  the \emph{Loeb} measure space $(Z_{\omega},u_{\omega})$ which is, in a sense, the ``ultralimit" of the space probability space $(\{1,\dots,d_{k}\},u_{d_{k}})$ (where $u_{d_{k}}$ is the uniform measure on the finite set $\{1,\dots,d_{k}\}$). If we fix a compact, metrizable space $X,$ then any sequence of functions $f_{k}\colon \{1,\dots,d_{k}\}\to X$ induces a measurable map $(f_{k})_{k\to\omega}\colon Z_{\omega}\to X.$
Thus the Loeb measure space allows us to take ``arbitrary limits" of sequences of functions, just like ultralimits allow us to take arbitrary limits of bounded sequences. We use $\Meas(Z_{\omega},X)$ for the space of $u_{\omega}$-measurable functions $Z_{\omega}\to X$ modulo the equivalence relation of equality almost everywhere.

This is related to local weak$^{*}$ convergence for the following reason. Suppose we fix a sofic group $G,$ and a sofic approximation $\sigma_{k}\colon G\to S_{d_{k}}$ (where $S_{n}$ is the symmetric group on $n$ letters). Suppose that $X$ is a compact, metrizable space and $G\actson X$ by homeomorphisms. Given a sequence $(\mu_{k})_{k}\in \Prob(X^{d_{k}})$ we let $(\mu_{k,j})_{j=1}^{d_{k}}$ be its marginals onto the different coordinates.  The collection of marginals may be regarded as a function $\{1,\dots,d_{k}\}\to \Prob(X).$ So by our comments in the preceding paragraph, we may take an ``ultralimit" of these functions and end up with an element $\mathcal{E}((\mu_{k})_{k})\in \Meas(Z_{\omega},\Prob(X)).$ We let $\mathcal{L}_{\omega}(G\actson X)$ be the set of all $\mathcal{E}((\mu_{k})_{k})$ where $(\mu_{k})_{k}$ is asymptotically supported on topological microstates for $G\actson X$ as $k\to\omega.$ It is an exercise in understanding the definitions to see that if  $\mu=\mathcal{E}((\mu_{k})_{k})\in \Prob(X),$ then $\mu_{k}\to^{lw^{*}}\mu$ if and only if $\mu_{k}$ locally weak$^{*}$ converges to $\mu$ as $k\to\omega.$  We may thus view $\mathcal{L}_{\omega}(G\actson X)$ as a space of ``generalized local weak$^{*}$ limits" of measures asymptotically supported on topological microstates.

There is one other important definition we need for our main theorem. Let $\omega,(\sigma_{k})_{k},G,X$ be as in the preceding paragraph. Given a closed subset $Y$ of $X,$ we say that $Y$ \emph{absorbs all topological microstates for $G\actson X$ with respect to $(\sigma_{k})_{k},\omega$} if for every (almost surely) $G$-equivariant $\Theta\in \Meas(Z_{\omega},X)$ we have that $\Theta(z)\in Y$ for almost every $z\in Z_{\omega}.$  This is equivalent to saying that for any sequence $(\phi_{k})_{k}$ of topological microstates, we have that $\phi_{k}(j)$ is ``close" to $X$ for ``most $j$" as $k\to\omega.$

Let $\mathcal{F}(X)$ be the set of nonempty closed subset of $X$ equipped with the Hausdorff (in the sense of Hausdorff metric) topology. We can also makes of absorbing all topological microstates not just for elements of $\mathcal{F}(X),$ but also for a $ Y\in \Meas(Z_{\omega},\mathcal{F}(X)).$ This just means that for every almost surely $G$-equivariant $\Theta\in \Meas(Z_{\omega},X)$ we have $\Theta(z)\in Y(z)$ for almost every $z\in Z_{\omega}$ (see Proposition \ref{P:what it means to absorb} for a way to say this in terms of topological microstates).  If $X$ is a compact group, and the $G$ action is by continuous automorphisms, then it is especially interesting to analyze elements of $\Meas(Z_{\omega},\mathcal{F}(X))$ which absorb all topological microstates and so that $Y(z)$ is a closed \emph{subgroup} of $X.$ We use $\Sub(X)$ for the space of closed subgroups of $X,$ which is a closed subset of $\mathcal{F}(X).$ Lastly, if $Y\in \Meas(Z_{\omega},\Sub(X)),$ we let $m_{Y}\in \Meas(Z_{\omega},\Prob(X))$ be given by $m_{Y}(z)=m_{Y(z)}.$ We let $\mathcal{S}_{\omega}(G\actson X)=\{Y\in \Meas(Z_{\omega},\Sub(X)):m_{Y}\in \mathcal{L}_{\omega}(G\actson X).$ We are now ready to state the two main ``max-min" theorems of the paper, which gives a precise duality between being a generalized local weak$^{*}$ limit and absorbing all topological microstates.

%
%We automatically obtain the following corollary, which gives a soft characterization (in terms of generalized  subgroups absorbing topological microstates) of when the Haar measure is a local weak$^{*}$ limit of measures supported on topological microstates.

\begin{thm}\label{T:maintheoremintro}
Let $G$ be a sofic group and $\sigma_{k}\colon G\to S_{d_{k}}$ a sofic approximation. Fix a free ultrafilter $\omega$ on the natural numbers. Let $G\actson X$ be an algebraic action. Then
\begin{enumerate}[(i)]
\item There is a maximal element $Y\in \Meas(Z_{\omega},\Sub(X))$ satisfying the condition that $m_{Y}$ is a generalized local weak$^{*}$-limit of measures supported on topological microstates for $G\actson X$ with respect to $(\sigma_{k})_{k},\omega,$ i.e. $m_{Y}\in \mathcal{L}_{\omega}(G\actson X).$   \label{T:existence intro}
\item If $Y$ is an in (i), then $Y$ is also the minimal element of $\Meas(Z_{\omega},\Sub(X))$ which absorbs all topological microstates with respect to $(\sigma_{k})_{k},\omega.$
\item If $Y$ is an in (i), then $Y=X$ if and only if there is a sequence of measures $\mu_{k}\in \Prob(X^{d_{k}})$ which are asymptotically supported on topological microstates and locally weak$^{*}$-converge to $m_{X}$ as $k\to\omega.$
\end{enumerate}
\end{thm}

\begin{cor}
Let $G$ be a sofic group and $\sigma_{k}\colon G\to S_{d_{k}}$ a sofic approximation. Fix a free ultrafilter $\omega$ on the natural numbers. Let $G\actson X$ be an algebraic action. Then the following are equivalent:
\begin{enumerate}[(i)]
\item There does not exist a sequence $\mu_{k}\in \Prob(X^{d_{k}})$ with $\mu_{k}\to^{lw^{*}}m_{X}.$
\item There is a sequence $(Y_{k})_{k}\in \Sub(X)^{d_{k}}$ with $\lim_{k\to\omega}(Y_{k})_{*}(u_{d_{k}})\ne \delta_{X}$ and so that $(Y_{k})_{k}$ absorbs all topological microstates for $G\actson X$ along $\omega.$
\end{enumerate}
\end{cor}

Observe that if $G$ is a sofic group with sofic approximation $\sigma_{k}\colon G\to S_{d_{k}},$ then we automatically have an induced embedding of $G$ into the metric ultraproduct $\prod_{k\to\omega}S_{d_{k}}.$ We have a natural action of $\prod_{k\to\omega}S_{d_{k}}$ on $(Z_{\omega},u_{\omega}).$ If we add the assumption that the centralizer $G$ inside $\prod_{k\to\omega}S_{d_{k}}$ acts ergodically on $(Z_{\omega},u_{\omega})$ in Theorem \ref{T:maintheoremintro} then we obtain a nicer statement. We mention that consideration of ergodic centralizer has already been of relevance to sofic entropy theory (see  \cite[Section 5]{KerrLi2}).

\begin{cor}\label{C:easier in ergodic commutant}
Let $G$ be a sofic group and $\sigma_{k}\colon G\to S_{d_{k}}$ a sofic approximation. Fix a free ultrafilter $\omega$ on the natural numbers, and suppose that the centralizer of $G$ inside $\prod_{k\to\omega}S_{d_{k}}$ acts ergodically on the Loeb measure space. Let $G\actson X$ be an algebraic action. Let $Y$ be the minimal subgroup of $X$ which absorbs all topological microstates with respect to $\omega.$ Then $Y$ is the largest subgroup of $X$ so that there is a sequence $\mu_{k}\in \Prob(X^{d_{k}})$ with $\mu_{k}\to^{lw^{*}}m_{Y}$ as $k\to\omega.$

\end{cor}

If $G$ is a residually finite group, and $\sigma_{k}$ comes from a sequence of finite quotients as in Theorem \ref{T:res finite case intro}, then it is known by \cite[Theorem 5.7]{KerrLi2}  that the centralizer of $G$ inside  $\prod_{k\to\omega}S_{d_{k}}$ acts ergodically on $Z_{\omega}.$ This is why the assumptions of residual finiteness are present in Theorems \ref{T:res finite case intro},\ref {T:cpe intro}, in each case this can be replaced by ergodicity of the action of centralizer of $G$ inside  $\prod_{k\to\omega}S_{d_{k}}$ on $Z_{\omega}.$ Theorems \ref{T:res finite case intro},\ref{T:cpe intro} are easy and effortless corollaries of Corollary \ref{C:easier in ergodic commutant}.

The reader may be  aware of an alternate way to obtain a ``generalized local weak$^{*}$ limit", which has already appeared in the probability literature (see e.g. \cite{LW*Measures, LWGraphs}) and has been recently studied in the sofic entropy context by Ab\'{e}rt-Weiss \cite{AWVertical}. Observe that a sequence of measures $\mu_{k}\in \Prob(X^{d_{k}})$ locally weak$^{*}$ converges to $\mu$ if and only if
\[\frac{1}{d_{k}}\sum_{j=1}^{d_{k}}\delta_{\mu_{k,j}}\to^{wk^{*}}_{k\to\infty}\delta_{\mu}\in \Prob(\Prob(X))\]
where $\mu_{k,j}$ is the $j^{th}$ marginal of $\mu_{k}.$ The idea would be to then consider an arbitrary sequence $\mu_{k}\in \Prob(X^{d_{k}})$ and take a limit along an ultrafilter of $\frac{1}{d_{k}}\sum_{j=1}^{d_{k}}\delta_{\mu_{k,j}}$ and get an element of $\Prob(\Prob(X))$ as a ``generalized local weak$^{*}$-limit." This is a tempting alternative to what we have done, since $\Prob(\Prob(X))$ is metrizable and $\Meas(Z_{\omega},\Prob(X))$ is not. Unfortunately, passage to the space $\Prob(\Prob(X))$ obliterates all of the algebraic structure of $\Prob(X))$, e.g. one cannot take convolutions of measures on $\Prob(\Prob(X))$ in a manner that respects local weak$^{*}$-limits in the sense of Ab\'{e}rt-Weiss. The object $\Meas(Z_{\omega},\Prob(X))$ is, in a precise sense, finer than that $\Prob(\Prob(X))$ (e.g. the Ab\'{e}rt-Weiss notion of a generalized weak$^{*}$-limit factors through ours) and because this remembers the algebraic structure of $\Prob(X).$ This is crucial for the proof of our main theorem. We refer the reader to  the comments following Definition \ref{D:coarser lw* limit} for  a more precise discussion of both notions of generalized local weak$^{*}$ limits, why ours is suited for the results of the paper, and why the Ab\'{e}rt-Weiss approach does not work to prove the main results of the paper.

By the above discussion, for a $G$-invariant element of $\mu\in\Prob(\Prob(X)),$ we can thus makes sense of what it means for a sequence $\mu_{k}\in \Prob(X^{d_{k}})$  to locally weak$^{*}$ converge to $\mu.$ As was discovered by Ab\'{e}rt-Weiss \cite{AWVertical}, we can thus define a lw$^{*}$-entropy of $\mu,$ which we denote by $h_{(\sigma_{k})_{k}}^{lw^{*}}(\mu).$ Ab\'{e}rt-Weiss have shown that this is an isomorphism invariant of $\mu,$ and one may also argue as in Section 6.1 of \cite{AustinAdd} to see this. While it may be the case that the topological entropy of $G\actson X$ is not equal to the measure entropy of $G\actson (X,m_{X}),$ in the context of lw$^{*}$-entropy  we can show that  it  that there is a ``subgroup" $Y$ for which the topological entropy of $G\actson X$ is equal to the lw$^{*}$ entropy of $G\actson (Y,m_{Y}).$ The caveat here is that one needs to regard $Y$ as a \emph{$G$-invariant random subgroup} of $X,$ i.e.a $G$-invariant element of $\Prob(\Sub(X)).$ For $Y\in \Prob(\Sub(X)),$ we let $m_{Y}$ be the element of $\Prob(\Prob(X))$ which is the pushforward of $Y$ under the map $K\mapsto m_{K}.$

\begin{thm}\label{T:lw*intro theorem}
Let $G$ be a sofic group and $\sigma_{k}\colon G\to S_{d_{k}}$ a sofic approximation. Fix a free ultrafilter $\omega$ on the natural numbers, and let $G\actson X$ be an algebraic action.
\begin{enumerate}[(i)]
\item There is a $G$-invariant random subgroup $Y$ of $X$ so that $h_{(\sigma_{k})_{k\to\omega}}^{lw^{*}}(m_{Y})=h_{(\sigma_{k})_{k},\topo}(G\actson X).$
\item If the centralizer of $G$ inside $\prod_{k\to\omega}S_{d_{k}}$ acts ergodically on $(Z_{\omega},u_{\omega}),$ then there is a subgroup $Y$ of $X$ so that
\[h_{(\sigma_{k})_{k\to\omega}}^{lw^{*}}(G\actson (Y,m_{Y}))=h_{(\sigma_{k})_{k\to\omega}}(G\actson X)=h_{(\sigma_{k})_{k\to\omega}}(G\actson Y).\] \label{I:equality of lw* entropy ergodic commutant case}
\end{enumerate}
\end{thm}
\vspace{-.3in}

The statement of the above theorem naturally leads to the consideration and importance of $G$-invariant random subgroups. The notion of \emph{invariant random subgroups}, i.e. probability measures on $\Prob(\Sub(H))$ for a discrete group $H$ which are invariant under the conjugation action of $H$. This name was given by Ab\'{e}rt-Glasner-Virag \cite{AVY14}, but related ideas had been in the mathematical community for some time, first appearing in work of Zimmer (see \cite{StuckZimmer}). Related results had also been proved before Ab\'{er}t-Glasner-Virag by Aldous-Lyons \cite{AldousLyons}, Bergeron-Gaboriau \cite{BergeronGaboriau}, and Vershik \cite{Vershik12}.
The study of invariant random subgroups has been quite active in recent years with connections to $L^{2}$-invariants \cite{7s12}, geometric group theory \cite{IRSGrigor, OsinIRS}, ergodic theory \cite{AVY14, 7s12, StuckZimmer, BowenIRSFurst, RobinAF, BurtonKechrisMax}, as well as many other subjects.
While not quite the same object, we think that $G$-invariant random subgroups in the algebraic action situation will be important to the study of entropy theory for the same reasons that invariant random subgroups are of relevance to several fields of mathematics.

We remark that  part (\ref{I:equality of lw* entropy ergodic commutant case}) of Theorem \ref{T:lw*intro theorem} is optimal in the residually finite case. We make this precise by the following result.

\begin{thm}\label{T:properties of Y intro}
Let $G$ be a countable, discrete, sofic group with a sofic approximation $\sigma_{k}\colon G\to S_{d_{k}}.$ Let $\omega$ be a free ultrafilter on the natural numbers, and suppose that the centralizer of $G$ inside $\prod_{k\to\omega}S_{d_{k}}$ acts ergodically on $(Z_{\omega},u_{\omega})$. Let $Y$ be the maximal element of $\mathcal{S}_{\omega}(G\actson X)$ (so $Y\in \Sub(X))$ by Corollary \ref{C:easier in ergodic commutant}). Then we have the following properties of $Y:$\begin{enumerate}[(a)]
\item For every $G$-invariant $K\in \Sub(X)$ with $Y\subsetneq K,$ we  have that $G\actson (K,m_{K})$ is not sofic with respect to $(\sigma_{k})_{k}.$ In particular, if $Y\ne X,$ then the action $G\actson (X,m_{X})$ is not sofic with respect to $(\sigma_{k})_{k}.$
\item $Y$ is the largest subgroup so that $h_{(\sigma_{k})_{k\to\omega}}^{lw^{*}}(G\actson (Y,m_{Y}))\ne \infty,$
\item $Y$  is the smallest subgroup of $X$ so that $h_{(\sigma_{k})_{k\to\omega},\topo}(G\actson X)=h_{(\sigma_{k})_{k\to\omega},\topo}(G\actson Y),$
\item $Y$ is the largest subgroup of $X$ so that $h_{(\sigma_{k})_{k\to\omega}}^{lw^{*}}(G\actson (Y,m_{Y}))=$ $h_{(\sigma_{k})_{k\to\omega},\topo}(G\actson Y)$.
\end{enumerate}\end{thm}

\vspace{-.01in}

The above result  makes it clear that the maximal element $Y$ of $\mathcal{S}_{\omega}(G\actson X)$ is of utmost importance for the entropy theory of $G\actson X.$
We make a few remarks on the use of ultrafilters in the paper. While ultrafilter analysis is quite efficient for our purposes, it is true that ultrafilters are quite abstract objects. it is not possible to concretely write down a free ultrafilter on the natural numbers. Because of this, it is desirable to obtain ultrafilter-free version of ours results, when possible. Fortunately, it well-known how to establish statements which are true for all sufficiently large $n$ from statements that are true for every free ultrafilter. This is why we are able to prove Theorems \ref{T:res finite case intro}-\ref{T:cpe intro} using Theorem \ref{T:maintheoremintro}. It is also not hard to state ultrafilter-free version of Theorem \ref{T:maintheoremintro} and Corollary \ref{C:easier in ergodic commutant}, though the statements are a little more awkward. The reader may consult Corollaries \ref{C:messy corollary} and \ref{C: ultraftiler free ergodic commutant} in this paper for the precise ultrafilter-free version of Theorem \ref{T:maintheoremintro} and Corollary \ref{C:easier in ergodic commutant}.

We close with remarks on the organization of the paper. We start in Section \ref{S:background setup} by giving some setup to the statement of the main theorem. Specifically in Section \ref{S:background setup} we give the precise definitions of the ultraproduct spaces we are working with for the majority of the paper, as well as the algebraic and lattice structure of these spaces. We end Section \ref{S:background setup}   by stating Theorem \ref{T:main theorem restated},  which is a lattice-theoretic result about these ultraproduct spaces, and then  begin Section \ref{S:major applications} deducing the main theorem (Theorem \ref{T:maintheoremintro}) from this lattice-theoretic result. This lattice-theoretic structure on ultraproduct spaces, is extremely powerful for the purposes of this paper. To illustrate this, in Section \ref{S:major applications}, we show that Corollary \ref{C:easier in ergodic commutant},  \ref{T:res finite case intro}-\ref{T:cpe intro}, \ref{T:lw*intro theorem}, \ref{T:properties of Y intro}, are quick corollaries of Theorem  \ref{T:maintheoremintro}, and that Theorem \ref{T:maintheoremintro} follows easily from \ref{T:main theorem restated}.     We then spend all of Section \ref{S:main proof} proving Theorem \ref{T:main theorem restated}. The proof is short, but requires a mild amount of preliminary material on operators on Hilbert spaces, which is why we postpone the proof until later in the paper. We then spend Section \ref{S:ultrafilter free versions} giving ultrafilter-free version of the main results, and in Section \ref{S:ergodic commutant sketch} we give a sketch of a proof of a simpler case of Corollary \ref{C:easier in ergodic commutant} that does not use ultrafilters. We spend Section \ref{S:Not SS actions} studying a few algebraic actions which are not strongly sofic, and computing (or at least obtaining information about) the maximal (generalized) subgroup whose Haar measure is a (generalized) local weak$^{*}$-limit. Appendix \ref{A:Loeb Measures} gives proofs of the main results on spaces of measurable maps on Loeb spaces that we need, and Appendix \ref{A:empirical limit} shows how we can also use our techniques to deduce the existence of a sequence of measures which locally and empirically converge to the Haar measure from the assumption of soficity of an action.

\textbf{Acknowledgments.} I thank Miklos Ab\'{e}rt, Andrei Alpeev, Tim Austin,  Lewis Bowen, Hanfeng Li and Brandon Seward for interesting discussions related to this paper. Much of this work was still done when I was a professor at Vanderbilt University. I thank Vanderbilt for providing an extremely productive work environment.

\section{Reduction of the main theorem: Lattice Structure on an ultraproduct of spaces of subgroups}

The main purpose of this section is to setup the appropriate the background for a lattice-theoretic statement which is sufficient to deduce Theorem  \ref{T:maintheoremintro}.

\subsection{Background and statement of the main reduction}\label{S:background setup}

We give some necessary background material for our proof of Theorem \ref{T:maintheoremintro}, which includes precise definitions of the spaces $\mathcal{L}_{\omega}(G\actson X),\mathcal{S}_{\omega}(G\actson X),$ and explain the algebraic and order structure that the spaces $\mathcal{L}_{\omega}(G\actson X),\mathcal{S}_{\omega}(G\actson X)$ posses.

We start by reminding the reader of the definition of ultrafilters and ultralimits.

\begin{defn}
Let $I$ be a set. An \emph{ultrafilter} on $I$ is a subset $\omega$ of the power set of $I$ satisfying the following axioms:
\begin{itemize}
\item $A,B\in \omega$ implies that $A\cap B\in \omega,$
\item if $A\in\omega$ and $A\subseteq B\subseteq I,$ then $B\in \omega,$
\item for every $A\in \omega,$ exactly one of $A,A^{c}\in \omega.$
\end{itemize}
We say that $\omega$ is \emph{free} if no finite set is in $I.$

Given a free ultrafilter $\omega$ on $I,$ a compact Hausdorff space $K,$ and $(x_{i})_{i\in I}\in K^{I},$ we say $x_{i}$ converges to $p\in K$ along $\omega$ if for every neighborhood $U$ of $p$ we have that $\{i\in I:x_{i}\in U\}\in \omega.$ We call $p$ the ultralimit of $x_{i}$ and write $\lim_{i\to\omega}x_{i}=p.$  It is an exercise to show that every element of $K^{I}$ converges along $\omega$ and the resulting ultralimit is unique.
\end{defn}

We will also frequently need to use notation for probability measures on compact spaces. If $X$ is a compact, metrizable space we let $\Prob(X)$ be the space of all Borel probability measure on $X.$ If $G$ is a countable, discrete, group and $G\actson X$ by homeomorphisms, we let $\Prob_{G}(X)$ be the set of $G$-invariant elements of $\Prob(X).$  If $A$ is a finite set, we let $u_{A}$ be the uniform measure on $A,$ if  $A=\{1,\cdots,n\},$ we typically use $u_{n}$ instead of $u_{\{1,\cdots,n\}}.$

\begin{defn}
Let $\omega$ be a free ultrafilter on the natural numbers, and let $(d_{k})_{k}$ be a sequence of natural numbers. Let
\[Z_{\omega}=\frac{\prod_{k}\{1,\dots,d_{k}\}}{(j_{k})_{k}\thicksim (j_{k}')_{k}\mbox{ if and only if } \{k:j_{k}=j_{k}'\}\in\omega}.\]
Given a sequence $A_{k}\subseteq \{1,\dots,d_{k}\},$ we let $\prod_{k\to\omega}A_{k}$ be the image in $Z_{\omega}$ of $\prod_{k}A_{k}$ under the quotient map $\prod_{k}\{1,\dots,d_{k}\}\to Z_{\omega}.$ We call the collection of $\prod_{k\to\omega}A_{k}$ ranking over all sequences $(A_{k})_{k}$ with $A_{k}\subseteq \{1,\dots,d_{k}\}$ the algebra of \emph{internal subsets of $Z_{\omega}$}. Let $\mathcal{B}$ be the $\sigma$-algebra of subsets of $Z_{\omega}$ generated by the internal subsets of $X.$ By \cite{Loeb}, there is a unique measure $u_{\omega}\colon \mathcal{B}\to [0,1]$ so that $u_{\omega}\left(\prod_{k\to\omega}A_{k}\right)=\lim_{k\to\omega}\frac{|A_{k}|}{d_{k}}.$ The triple $(X,\mathcal{B},u_{\omega})$ is called the \emph{Loeb measure space}. It can be shown that for every $A\in \mathcal{B},$ there is a sequence $(A_{k})_{k\to\omega}$ so that $u_{\omega}(A\Delta (A_{k})_{k\to\omega})=1.$ Given $(j_{k})_{k}\in \prod_{k}\{1,\dots,d_{k}\},$ we let $(j_{k})_{k\to\omega}$ be equivalence class of $(j_{k})_{k}$ as an element of $Z_{\omega}.$
\end{defn}

We will mainly be interested in measurable maps from the Loeb measure space into some metrizable space $X$.  By measurable, we mean that $f^{-1}(U)$ is measurable for every open subset $U$ of $X.$ It follows that if $f\colon Z_{\omega}\to X$ is measurable, then  $f^{-1}(E)$ is measurable for every Borel $E\subseteq X$ (Borel for us will mean the $\sigma$-algebra generated by the open subsets of $X$). We let $\Meas(Z_{\omega},X)$ be the space of measurable maps from the Loeb measure space into $X,$ where we identify two such maps if they agree almost everywhere. Given a metric $\rho$ on $X,$ we define a metric $\rho_{m}$ on $\Meas(Z_{\omega},X)$ by
\[\rho_{m}(\phi,\psi)=\int_{X}\rho(\phi(z),\psi(z))\,du_{\omega}(z).\]
It is straightforward to check that $\rho_{m}(f_{n},f)\to 0$ if and only if for every $\varepsilon>0$ we have that $\mu(\{x\in X:\rho(f_{n}(x),f(x))<\varepsilon\})\to 1.$
We need a few simple operations and terms related to the Loeb measure space. If $(X,\rho_{X}),(Y,\rho_{Y})$ are metric spaces and $f\colon X\to Y$ is Borel, we define $f_{*}\colon \Meas(Z_{\omega},X)\to \Meas(Z_{\omega},Y)$ by $f_{*}(\phi)=f\circ \phi.$

We collect some facts about the Loeb measure space that we will need later.

\begin{prop}\label{P:just the basics Loeb}
Let $(X,\rho_{X}),(Y,\rho_{Y})$ be metric spaces, $(d_{k})_{k}$ a sequence of natural numbers,  and $\omega$ a free ultrafilter on the natural numbers.
\begin{enumerate}[(i)]
\item \label{I:induced completeness} If $(X,\rho_{X})$ is a  complete metric space, then so is $(\Meas(Z_{\omega},X),\rho_{X,m}).$
\item \label{I:induced continuity Loeb} Suppose that $f\colon X\to Y$ is uniformly continuous.Then the induced map $f_{*}\colon \Meas(Z_{\omega},X)\to \Meas(Z_{\omega},Y)$ is uniformly continuous.
\item \label{I:closed image Loeb} Suppose that $f\colon X\to Y$ is uniformly continuous, and that there is a uniformly continuous map $g\colon X\to Y$ so that $g(f(x))=x$ for all $x\in X.$ Then $f_{*}(\Meas(Z_{\omega},X))$ is closed.

\end{enumerate}
\end{prop}

Items (\ref{I:induced completeness})--(\ref{I:closed image Loeb}) are true with the Loeb measure replaced by a general measure space and are simple enough that their proofs are left as exercises to the reader. We only collect them here because we will use items (\ref{I:induced completeness})--(\ref{I:closed image Loeb}) frequently throughout the paper.

We will often be interested in $\Meas(Z_{\omega},X)$ when $(X,\rho)$ is compact. Fix a compact metric space $(X,\rho).$ In this case, we can simplify much of the discussion. For example, it is easy to see that given a sequence $(f_{n})_{n}\in \Meas(Z_{\omega},X),$ and $f\in \Meas(Z_{\omega},X)$ we have that $\rho_{m}(f_{n},f)\to 0$ if and only if for every neighborhood $\mathcal{O}$ of the diagonal in $X\times X,$ then $\mu(\{x\in X:(f(x),f_{n}(x))\in \mathcal{O}\})\to 1.$  It thus follows that the topology induced by $\rho_{m}$ does not depend upon $\rho.$ We call this the topology of convergence in measure. Because of these comments, in the compact case we will typically not specify a metric $\rho$ on $X,$ and instead can simply work with open neighborhoods of the diagonal in $X\times X.$ We need one last construction in the compact case.

For notation, given a set $A$ and an $n\in \N,$ we identify $A^{n}$ with all functions $\{1,\cdots,n\}\to A.$
 Fix a sequence $(f_{k})_{k}$ with $f_{k}\in X^{d_{k}}.$ Define a map $(f_{k})_{k\to\omega}\colon Z_{\omega}\to X$ by $(f_{k})_{k\to\omega}((j_{k})_{k\to\omega})=\lim_{k\to\omega}f_{k}(j_{k}).$ It is clear that if $\phi\colon X\to Y$ is Borel, then $\phi_{*}((f_{k}))_{k\to\omega}=(\phi\circ f_{k})_{k\to\omega}.$  It is shown in the appendix (see Proposition \ref{P:loeb basics yo} (\ref{I:borelmaps and shiz})) that $(f_{k})_{k\to\omega}$ is always a Borel map. Given a sequence of sets $E_{k}\subseteq X^{d_{k}},$ we let
\[\prod_{k\to\omega}E_{k}=\left\{(f_{k})_{k\to\omega}:(f_{k})_{k}\in \prod_{k}E_{k}\right\}.\]
A set of the form $\prod_{k\to\omega}E_{k}$ is called an \emph{internal set}.

Since we will use it quite often, we state some basic properties of internal sets in the following proposition. The proof of this proposition is given in Appendix \ref{A:Loeb Measures}.

\begin{prop}\label{P:basics compact Loeb}
Let $X$ be a compact metrizable space, $(d_{k})_{k}$ a sequence of natural numbers, and $\omega$ a free ultrafilter on the natural numbers.
\begin{enumerate}[(i)]
\item \label{I:going backwards loeb} For any $f\in \Meas(Z_{\omega},X),$ there exists $(f_{k})_{k}\in \prod_{k}X^{d_{k}}$ so that $f=(f_{k})_{k\to\omega}$ almost everywhere.
\item \label{I:some closed set etc Loeb} Internal subsets of $\Meas(Z_{\omega},X)$ are always closed.
\item If $Y$ is a Polish space, $f\colon X\to Y$ is continuous, and $E\subseteq \Meas(Z_{\omega},X)$ is a countable intersection of internal sets, then $f_{*}(E)$ is closed.
\item If $R\in [0,\infty)$ and $f\in L^{\infty}(Z_{\omega},u_{\omega}),$ and $f=(f_{k})_{k\to\omega}$ for some $f_{k}\in \ell^{\infty}(d_{k})$ with $\|f_{k}\|_{\infty}\leq R,$ then
\[\int f(z)\,du_{\omega}(z)=\lim_{k\to\omega}\frac{1}{d_{k}}\sum_{j=1}^{d_{k}}f_{k}(j).\]
\item If $f\in \Meas(Z_{\omega},X)$ and $(f_{k})_{k}\in \prod_{k}X^{d_{k}}$ is such that $f=(f_{k})_{k\to\omega}$ almost everywhere, then $f_{*}(u_{\omega})=\lim_{k\to\omega}(f_{k})_{*}(u_{d_{k}})$ in the weak$^{*}$ topology.
\end{enumerate}
\end{prop}

Fix a sequence $(d_{k})_{k}$ of natural numbers, and a free ultrafilter $\omega$ on the natural numbers. Let
\[\prod_{k\to\omega}S_{d_{k}}=\frac{\prod_{k}S_{d_{k}}}{\left\{(\sigma_{k})_{k}\in \prod_{k}S_{d_{k}}:\lim_{k\to\omega}d_{\Hamm}(\sigma_{k}(g),\id)=0\right\}}.\]
Given $(p_{k})_{k}\in \prod_{k}S_{d_{k}},$ we let $(p_{k})_{k\to\omega}\in \prod_{k\to\omega}S_{d_{k}}$ be the image of $(p_{k})_{k}$ under the quotient map. Observe that we have a well-defined, measure-preserving action $\prod_{k\to\omega}S_{d_{k}}\actson Z_{\omega}$ given by $(p_{k})_{k\to\omega}((j_{k})_{k\to\omega})=(p_{k}(j_{k}))_{k\to\omega}.$
Suppose that $G$ is a sofic group, and $\sigma_{k}\colon G\to S_{d_{k}}$ is a sofic approximation, we then have an induce injective homomorphism $\sigma_{\omega}\colon G\to \prod_{k\to\omega}S_{d_{k}}$ by $\sigma_{\omega}(g)=(\sigma_{n}(g))_{n\to\omega}.$ So we have an induced action of $G$ on $Z_{\omega}$ via $\sigma_{\omega}.$

Let $X$ be a compact, metrizable space and $G\actson X$ by homeomorphisms. A \emph{sequence of topological microstates with respect to $(\sigma_{k})_{k},\omega$} is a sequence $\phi_{k}\colon \{1,\dots,d_{k}\}\to X$ so that \[\lim_{k\to\omega}u_{d_{k}}(\{j:(g\phi(j),\phi(\sigma_{k}(g)(j))\in U\})=1\] for every neighborhood $U$ of the diagonal, and every $g\in G.$ We shall typically drop  ``with respect to $(\sigma_{k})_{k}"$, if $(\sigma)_{k}$ is clear from the context and simply say ``with respect to $\omega$." If a sequence $(\phi_{k})_{k}\in \prod_{k}X^{d_{k}}$ satisfies $\lim_{k\to\infty}u_{d_{k}}(\{j:(g\phi(j),\phi(\sigma_{k}(g)(j))\in U\})=1$ or every neighborhood $U$ of the diagonal, and every $g\in G,$ we call it a sequence of topological microstates with respect to $(\sigma_{k})_{k},$ again we often drop ``with respect to $(\sigma_{k})_{k}$" if it is clear from the context.

We mention for later use the related notion of measure-theoretic microstates. Suppose $X,(\sigma_{k})_{k},\omega$ are as in the preceding paragraph, and that $\mu\in \Prob_{G}(X).$ A \emph{sequence of measure-theoretic microstates for $G\actson (X,\mu)$ with respect to $(\sigma_{k})_{k},\omega$} is a sequence $(\phi_{k})_{k}$ of topological microstates which satisfy that $\lim_{k\to\omega}(\phi_{k})_{*}(u_{d_{k}})=\mu,$ where the limit is taken in the weak$^{*}$ topology.

\begin{prop}\label{P:characterize top micro}
Let $G$ be a sofic group with sofic approximation $\sigma_{k}\colon G\to S_{d_{k}},$ and let $\omega$ be a free ultrafilter on the natural numbers. Fix a compact, metrizable space $X$ and an action $G\actson X$ by homeomorphisms.
\begin{enumerate}[(i)]
\item \label{I:get this equivariant map} Suppose that $\phi_{k}\colon\{1,\dots,d_{k}\}\to X$ is a sequence of topological microstates with respect to $\omega,$ then the induced map $(\phi_{k})_{k\to\omega}\colon Z_{\omega}\to X$ is $G$-equivariant and measurable.
\item \label{I:get these microstates} Given  a $G$-equivariant, measurable map $\Phi\colon Z_{\omega}\to X,$ there is a sequence $(\phi_{k})_{k}$ of maps $\phi_{k}\colon \{1,\dots,d_{k}\}\to X$ which are topological microstates with respect to $\omega$ and so that $\Phi=(\phi_{k})_{k\to\omega}$ almost everywhere.
\end{enumerate}

\end{prop}

\begin{proof}
Throughout we fix a compatible metric $\rho$ on $X.$

(\ref{I:get this equivariant map}):
Fix a $g\in G$ and an $\varepsilon>0.$ For $k\in \N,$ let $E_{k}=\{1\leq j\leq d_{k}:\rho(\phi_{k}(\sigma_{k}(g)(j)),g\phi_{k}(j))<\varepsilon\},$ and let $E=\{z\in Z_{\omega}:\rho(\Phi(\sigma_{\omega}(g)(z)),g\Phi(z))<\varepsilon\}.$ We then have that $E=\prod_{k\to\omega}E_{k},$ so
\[u_{\omega}(E)=\lim_{k\to\omega}u_{d_{k}}(E_{k})=1,\]
the last equality following by applying the definition of being a sequence of topological microstates to  $U=\{(x,y)\in X:\rho(x,y)<\varepsilon\}.$

(\ref{I:get these microstates}):
By Proposition \ref{P:loeb basics yo} (\ref{I:sequences and aljsakljal}), we may find a sequence of maps $(\phi_{k})_{k}\in \prod_{k}X^{d_{k}}$ so that $(\phi_{k})_{k\to\omega}=\Phi$ almost everywhere. Let $U$ be a neighborhood of the diagonal in $X\times X.$ Since $X$ is compact, we can find an $\varespilon>0$ so that $U\supseteq \{(x,y):\rho(x,y)<\varespilon\}.$ For every $g\in G,$ we then have
\begin{align*}
\lim_{k\to\omega}u_{d_{k}}(\{j:(\phi(\sigma_{k}(g)(j)),g\phi(j))\in U\})&\geq \lim_{k\to\omega}u_{d_{k}}(\{j:\rho(\phi(\sigma_{k}(g)(j)),g\phi(j))<\varespilon\})\\
&=u_{\omega}(\{z:\rho(\Phi(gz),g\Phi(z))<\varepsilon\})=1.
\end{align*}
Since $g\in G$ was arbitrary, we see that $(\phi_{k})_{k}$ is a sequence of topological microstates with respect to $\omega.$
\end{proof}

Because of the above proposition, we call a measurable, almost surely $G$-equivariant map $\Theta\colon Z_{\omega}\to X$ a topological microstate with respect to $(\sigma_{k})_{k},\omega.$ We typically drop the phrase ``with respect to $(\sigma_{k})_{k},\omega"$ if $(\sigma)_{k},\omega$ are clear from the context. If $\mu\in \Prob_{G}(X),$ then a measurable, almost surely $G$-equivariant map $\Theta\colon Z_{\omega}\to X$ with $(\Theta)_{*}(u_{\omega})=\mu$ will be called a \emph{measure microstate} with respect to $(\sigma_{k})_{k},\omega.$ As in the topological case,we typically drop the phrase  ``with respect to $(\sigma_{k})_{k},\omega"$ if $(\sigma)_{k},\omega$ are clear from the context. If there exists a measure microstate $\Theta\colon Z_{\omega}\to X,$ we then say that $G\actson (X,\mu)$ is \emph{sofic with respect to $(\sigma_{k})_{k},\omega.$}

As we stated in the introduction, much of the paper is related to the concept of local weak$^{*}$ convergence, we recall the notion here.
\begin{defn}\label{D:important notions for the paper}
Fix a free ultrafilter $\omega$ on the natural numbers. Let $G$ be a countable, discrete, sofic group with sofic approximation $\sigma_{k}\colon G\to S_{d_{k}},$ and let $X$ be a compact, metrizable space with $G\actson X$ by homeomorphisms.  Given a sequence $\mu_{k}\in \Prob(X^{d_{k}}),$ we say that:
\begin{itemize}
\item $\mu_{k}$ is \emph{asymptotically supported on topological microstates with respect to $(\sigma_{k})_{k}$ as $k\to\omega$} if for every $g\in G,$ and every open neighborhood $U$ of the diagonal in $X\times X$ we have:
\[\lim_{k\to\omega}\mu_{k}(\{\phi:u_{d_{k}}(\{j:(\phi(\sigma_{k}(g)(j)),g\phi(j))\in U\})\})=1,\]
\item for $\mu \in \Prob_{G}(X),$ we say that $\mu_{k}$ \emph{locally weak$^{*}$} converges to $\mu$ with respect to $(\sigma_{k})_{k}$ as $k\to\omega$ if for every weak$^{*}$ neighborhood $\mathcal{O}$ of $\mu$ we have
\[\lim_{k\to\omega}u_{d_{k}}(\{j:(\mathcal{E}_{j})_{*}(\mu_{k})\in \mathcal{O}\})=1.\]
Here $\mathcal{E}_{j}\colon X^{d_{k}}\to X$ is given by $\mathcal{E}_{j}(x)=x(j)$ for $1\leq j\leq d_{k}.$
\item For $\mu\in \Prob_{G}(X),$ we say that $\mu_{k}$ is asymptotically supported on measure  microstates for $G\actson (X,\mu)$ with respect to $(\sigma_{k})_{k}$ as $k\to\omega$ if $\mu_{k}$ is asymptotically supported on topological microstates for $G\actson X$ with respect to $(\sigma_{k})_{k}$ as $k\to\omega,$ and if for every weak$^{*}$-neighborhood $\mathcal{O}$ of $\mu$ we have
\[\lim_{k\to\omega}\mu_{k}(\{\phi:\phi_{*}(u_{d_{k}})\in \mathcal{O}\})=1.\]
\end{itemize}
\end{defn}

We shall typically drop some (or all) of ``for $G\actson (X,\mu)$", ``for $G\actson X$", or ``with respect to $(\sigma_{k})_{k}$" if the action or the sofic approximation is clear from the context (which it usually is). Thus we will often say ``$\mu_{k}$ is asymptotically supported on topological microstates as $k\to\omega.$" The notion of local weak$^{*}$-convergence has been well studied from the point of view of probability (see e.g. \cite{LW*AMW, LW*Measures,LWMEta, Newman1996SpatialIA}). Its relevance to sofic entropy was first systematically researched in \cite{AustinAdd}. We remark that related concepts were first introduced to the sofic entropy community in \cite{BowenEntropy} in the residually finite case (and in \cite{Me5} in the sofic case),  in \cite[Theorem 4.1]{BowenEntropy} \cite[Lemma 5.4]{Me5} it is already shown that local weak$^{*}$ convergence implies being asymptotically supported on measures  microstates if $G\actson (X,\mu)$ is ergodic.

For $G,X,(\sigma_{k})_{k},\omega$  as in Definition \ref{D:important notions for the paper}, we let $\mathcal{P}_{\omega}(G\actson X)$ be the set of all sequences $(\mu_{k})_{k}\in \prod_{k}X^{d_{k}}$ which are asymptotically supported on topological microstates as $k\to\omega.$ The space $\mathcal{P}_{\omega}(G\actson X)$ depends upon $(\sigma_{k})_{k},$ but we will suppress this from the notation as $(\sigma_{k})_{k}$ will essentially always be clear from the context.

We now turn to the  space of measurable functions on the Loeb measure space which will be most important for our study of local weak$^{*}$ convergence.

\begin{defn}
Let $G,(\sigma_{k})_{k},\omega$ be as in Definition \ref{D:important notions for the paper}. For $k\in \N,$ we set $\mathcal{E}^{(k)}\colon \Prob(X^{d_{k}})\to \Prob(X)^{d_{k}}$ by $\mathcal{E}^{(k)}(\mu)(j)=(\mathcal{E}_{j})_{*}(\mu_{k}).$  Define $\mathcal{E}\colon \prod_{k}\Prob(X^{d_{k}})\to \Meas(Z_{\omega},\Prob(X))$ by
\[\mathcal{E}((\mu_{k})_{k})=(\mathcal{E}^{(k)}(\mu_{k}))_{k\to\omega},\mbox{ for  $\mu=(\mu_{k})_{k}\in\prod_{k}\Prob(X^{d_{k}})$.}\]
We set $\mathcal{L}_{\omega}(G\actson X)=\mathcal{E}(\mathcal{P}_{\omega}(X,G)).$
\end{defn}

 Observe that if $\mu\in \Prob(X),$ then $\mu\in\mathcal{L}_{\omega}(G\actson X)$ if and only if there is a sequence $(\mu_{k})_{k}\in \prod_{k}\Prob(X^{d_{k}})$ of measures which are asymptotically supported on topological microstates and have $\mu_{k}\to^{lw^{*}}\mu$ as $k\to\omega.$ In this way we can think of $\mathcal{L}_{\omega}(G\actson X)$ as a space of generalized local weak$^{*}$ limits of measures supported on topological microstates. Moreover, the space $\mathcal{L}_{\omega}(G\actson X)$ has the added advantage that \emph{any} sequence $(\mu_{k})_{k}\in \prod_{k}\Prob(X^{d_{k}})$ which is asymptotically  supported on topological microstates as $k\to\omega$ has a ``local weak$^{*}$ limit" in $\mathcal{L}_{\omega}(G\actson X),$ namely $\mathcal{E}((\mu_{k})_{k}).$

The following is proved exactly as in Proposition \ref{P:characterize top micro}. Note that if $G,X$ are as above, then we have an induced action $G\actson \Prob(X)$ by $g\mu=g_{*}(\mu).$

\begin{prop}
Let $G$ be a countable discrete group with sofic approximation $\sigma_{k}\colon G\to S_{d_{k}}.$ Fix an action $G\actson X$ by homeomorphisms, where $X$ is a compact metrizable space, and a free ultrafilter on the natural numbers. Then every element of $\mathcal{L}_{\omega}(G\actson X)$ is $G$-equivariant with respect to the actions $G\actson \Prob(X),G\actson Z_{\omega}.$
\end{prop}

It will be helpful to frequently use the following topological fact about this space of ``generalized local weak-$^{*}$ limits". We will need to recall the following definition of Kerr-Li (see \cite[Definition 2.2]{KLi2}). For this definition, we use the following notation: if $\rho$ is a pseudometric on a set $E,$ then for any $n\in \N,p\in [1,\infty),$ we denote $\rho_{p}$ the pseudometric on $E^{n}$ defined by $\rho_{p}(\phi,\psi)=\left(\frac{1}{n}\sum_{j=1}^{n}\rho(\phi(j),\psi(j))^{p}\right)^{1/p}.$

\begin{defn}
Let $G$ be a countable group, $d\in \N,$ and $\sigma\colon G\to S_{d}$ a map (not assumed to be a homomorphism). Let $X$ be a compact, metrizable group with $G\actson X$ by homeomorphisms. For a finite $F\subseteq G$ and $\delta>0,$ we let $\Map(\rho,F,\delta,\sigma)$ be the set of all $\phi\colon\{1,\dots,d\}\to X$ so that:
\[\max_{g\in F}\rho_{2}(g\phi,\phi\circ \sigma_{k}(g))<\delta.\]

\end{defn}

\begin{thm}\label{T:topology is helpful}
Let $G$ be a countable discrete group with sofic approximation $\sigma_{k}\colon G\to S_{d_{k}}.$ Fix an action $G\actson X$ by homeomorphisms, where $X$ is a compact metrizable space, and a free ultrafilter on the natural numbers. Then $\mathcal{L}_{\omega}(G\actson X)$ is a countable intersection of internal sets, and is thus closed.
\end{thm}

\begin{proof}

Fix a compatible metric $\rho$ on $X.$ Choose an increasing sequence of finite subsets $F_{n}$ of $G$ with $\bigcup_{n}F_{n}=G,$ and a decreasing sequence $\delta_{n}$ of positive real numbers converging to zero.
 We claim that
\begin{equation}\label{E:get these interesections}
\mathcal{L}_{\omega}(G\actson X)=\bigcap_{n}\prod_{k\to\omega}\mathcal{E}^{(k)}\left(\left\{\mu \in \Prob(X^{d_{k}}):\mu\left(\Map(\rho,F_{n},\delta_{n},\sigma_{k})\right)\geq 1-\delta\right\}\right)
\end{equation}
Once we know (\ref{E:get these interesections}), it follows from Proposition \ref{P:basics compact Loeb} (\ref{I:some closed set etc Loeb}) that $\mathcal{P}_{\omega}(X,G)$ is closed. So it suffices to show (\ref{E:get these interesections}). It is clear that
\[\mathcal{L}_{\omega}(G\actson X)\subseteq=\bigcap_{n}\prod_{k\to\omega}\mathcal{E}^{(k)}\left(\left\{\mu \in \Prob(X^{d_{k}}):\mu\left(\Map(\rho,F_{n},\delta_{n},\sigma_{k})\right)\geq 1-\delta\right\}\right).\]
Fix a
\[\mu \in \bigcap_{n}\prod_{k\to\omega}\mathcal{E}^{(k)}\left(\left\{\mu \in \Prob(X^{d_{k}}):\mu\left(\Map(\rho,F_{n},\delta_{n},\sigma_{k})\right)\geq 1-\delta\right\}\right),\] it then suffices to show that $\mu \in \mathcal{E}(\mathcal{P}_{\omega}(X,G)).$ Write $\mu=(\tilde{\mu}_{k})_{k\to\omega}$ for some $(\widetilde{\mu}_{k})_{k}\in \prod_{k}\Prob(X^{d_{k}}).$

Fix a decreasing sequence of weak$^{*}$-neighborhoods $\mathcal{O}_{n}$ of $\{0\}$ in $M(X)$ with $\left(\bigcap_{n=1}^{\infty}\mathcal{O}_{n}\right)\cap \Ball(M(X))=\{0\}.$ Since
\[\mu \in\bigcap_{n}\prod_{k\to\omega}\mathcal{E}^{(k)}\left(\left\{\mu \in \Prob(X^{d_{k}}):\mu\left(\Map(\rho,F_{n},\delta_{n},\sigma_{k})\right)\geq 1-\delta\right\}\right),\]
for every $n\in \N$ we may find a $\nu^{(n)}\in \prod_{k}\mathcal{E}^{(k)}\left(\left\{\mu \in \Prob(X^{d_{k}}):\mu\left(\Map(\rho,F_{n},\delta_{n},\sigma_{k}\right)\geq 1-\delta\right\}\right)$ so that $\mu=(\mathcal{E}^{(k)}(\nu^{(n)}_{k}))_{k\to\omega}.$ We may find a decreasing sequence $B_{n}$ of subsets of $\N$ so that:
\begin{itemize}
\item $B_{n}\in \omega$ for all $n\in \N,$ and
\item for all $k\in B_{n}$ we have $u_{d_{k}}(\{j:\tilde{\mu}_{k,j}-\nu^{(n)}_{k,j}\in \mathcal{O}_{n}\})\geq 1-2^{-n}.$
\end{itemize}
Let $B_{0}=\N\setminus B_{1}.$ Given $k\in \N,$ let $n(k)\in \N\cup\{0\}$ be such that $n\in B_{n(k)}\setminus B_{n(k)+1}.$ Define $\mu_{k}=\nu^{n(k)}_{k}.$ It is then easy to see that $(\mu_{k})_{k}\in \mathcal{P}_{\omega}(X,G)$ and that $\mathcal{E}((\mu_{k})_{k})=\mu.$ This shows (\ref{E:get these interesections}) and completes the proof.

\end{proof}

We will also need the notion of absorbing topological microstates which uses the ultraproduct framework we have setup so far.

\begin{defn}
Fix a sofic group $G$, a sofic approximation $\sigma_{k}\colon G\to S_{d_{k}},$ and a free ultrafilter $\omega$ on the natural numbers. Let $X$ be a compact, metrizable space with $G\actons X$ by homeomorphisms. We say that a measurable map $Y\colon Z_{\omega}\to \mathcal{F}(X)$ \emph{absorbs all topological microstates for $G\actson X$ with respect to  $(\sigma_{k})_{k},\omega$} if topological microstate $\Phi\colon Z_{\omega}\to X$ we have $\Phi(z)\in Y(z)$ for $u_{\omega}$-almost every $z\in Z_{\omega}.$
\end{defn}

We will often drop the phrase `` with respect to $(\sigma_{k})_{k},\omega$" if $(\sigma_{k})_{k},\omega$ are clear from the context.

By Proposition \ref{P:basics compact Loeb} (\ref{I:going backwards loeb}), we know that every measurable map $\Phi\colon Z_{\omega}\to X$ is of the form $(\phi_{k})_{k\to\omega}$, up to sets of measure zero. So another way to say this is as follows: fix a sequence $(Y_{k})_{k}\in \prod_{k}\mathcal{F}(X^{d_{k}}).$ Then $Y=(Y_{k})_{k\to\omega}$ absorbs all topological microstates if and only if  for every sequence $(\phi_{k})_{k}\in \prod_{k}X^{d_{k}}$ which are topological microstates with respect to $\omega,$ and every $\varepsilon>0$ we have
\[\lim_{k\to\omega}u_{d_{k}}(\{j:\rho(\phi_{k}(j),Y_{k}(j))<\varepsilon\})=1.\]

We introduce a kind of order relation that we will make good use of later. Given $\mu\in \Meas(Z_{\omega},\Prob(X))$ and $Y\in \Meas(Z_{\omega},\mathcal{F}(X)),$ we will say that \emph{$Y$ supports $\mu$} and write $\mu\preceq Y$ if $\mu(z)(Y(z))=1$ for almost every $z\in Z_{\omega}.$

We also induce a partial order on $\Meas(Z_{\omega},\mathcal{F}(X))$ as follows: given $Y_{1},Y_{2}\in \Meas(Z_{\omega},\mathcal{F}(X))$ we say that $Y_{1}\leq Y_{2}$ if $Y_{1}(z)\subseteq Y_{2}(z)$ for almost every $z\in Z_{\omega}.$

We present a few basic facts about the above order relations that we will need later.

\begin{prop}\label{P:what it means to absorb}
Let $G$ be  a countable, discrete, sofic group with sofic approximation $\sigma_{k}\colon G\to S_{d_{k}},$ and let $X$ be a compact, metrizable space with $G\actson X$ by homeomorphisms.
Let $Y\in \Meas(Z_{\omega},\mathcal{F}(X)),$ and let $(Y_{k})_{k}\in \prod_{k}\mathcal{F}(X)^{d_{k}}$ be such that $Y=(Y_{k})_{k\to\omega}$ almost everywhere.  Then the following are equivalent:
\begin{enumerate}[(i)]
\item \label{I:absorbs stuff alala} $Y$ absorbs all topological microstates with respect to $\omega,$
\item \label{I:sequences and stuff yo!!;pkjsa} for any sequence $(\phi_{k})_{k}$ of topological microstates with respect to $\omega$ we have
\[\lim_{k\to\omega}\frac{1}{d_{k}}\sum_{j=1}^{d_{k}}\rho(\phi_{k}(j),Y_{k}(j))=0.\]
\item \label{I:finitary version} For every $\varepsilon>0,$ there is a finite $F\subseteq G,$ a  $\delta>0,$ and a $B\in\omega$ so that for every $k\in B,$ and every $\phi\in\Map(\rho,F,\delta,\sigma_{k}),$ we have that $\frac{1}{d_{k}}\sum_{j=1}^{d_{k}}\rho(\phi(j),Y_{k}(j))<\varespilon,$
\item \label{I:absorbs measures} for any $\mu \in\mathcal{L}_{\omega}(G\actson X)$ we have that $\mu\preceq Y.$
\end{enumerate}
\end{prop}

\begin{proof}
Without loss of generality, we may assume that $Y=(Y_{k})_{k\to\omega},$ as this does not change any of the statements.

(\ref{I:absorbs stuff alala}) implies (\ref{I:sequences and stuff yo!!;pkjsa}): Let $(\phi_{k})_{k}$ be a sequence of topological microstates with respect to $\omega$ and let $\Phi=(\phi_{k})_{k\to\omega}.$  Let $M$ be the diameter of $(X,\rho).$ Fix $\varespilon>0,$ and let $E_{k}=\{j:\rho(\phi_{k}(j),Y_{k}(j))<\varepsilon\}.$ We then have that $\prod_{k\to\omega}E_{k}\supseteq \{z:\rho(\Phi(z),Y(z))<\varespilon\}$ and thus $\lim_{k\to\omega}u_{d_{k}}(E_{k})=1.$ Hence we have that
\[\lim_{k\to\omega}\frac{1}{d_{k}}\sum_{j=1}^{d_{k}}\rho(\phi_{k}(j),Y_{k}(j))\leq \lim_{k\to\omega}\varepsilon+Mu_{d_{k}}(E_{k}^{c})=\varepsilon.\]
Since $\varespilon>0$ was arbitrary, we have shown that $\lim_{k\to\omega}\frac{1}{d_{k}}\sum_{j=1}^{d_{k}}\rho(\phi_{k}(j),Y_{k}(j))=0.$

(\ref{I:sequences and stuff yo!!;pkjsa}) implies (\ref{I:finitary version}):
Fix a $\varepsilon>0.$ Given a finite $F\subseteq G$ and a $\delta>0,$ let $E_{F,\delta}$ be the set of natural numbers $k$ so that for every $\phi\in\Map(\rho,F,\delta,\sigma_{k}),$ we have that $\frac{1}{d_{k}}\sum_{j=1}^{d_{k}}\rho(\phi(j),Y_{k}(j))<\varespilon.$  Fix an increasing sequence $F_{n}$ of finite subsets of $G$ with $\bigcup_{n}F_{n}=G$ and a decreasing sequence $\delta_{n}$ of positive numbers converging to zero.

Suppose that the (\ref{I:finitary version}) is false for this $\varepsilon>0.$ Then for every $n\in \N,$ we have that $E_{F_{n},\delta_{n}}^{c}\in \omega.$ Set $B_{n}=\bigcap_{l=1}^{n}E_{F_{l},\delta_{l}}\cap \{n,n+1,\cdots\}.$ Then $B_{n}$ are a decreasing family of elements of $\omega$ with $\bigcap_{n}B_{n}=\varnothing.$ For every $k\in B_{n},$ we can find a $\phi_{n,k}\in \Map(\rho,F,\delta,\sigma_{k})$ so that $\frac{1}{d_{k}}\sum_{j=1}^{d_{k}}\rho(\phi_{n,k}(j),Y_{k}(j))\geq \varespilon.$ Set $B_{0}=\N\setminus B_{1},$ and for each $k\in \N,$ let $n(k)\in \N\cup\{0\}$ be such that $k\in B_{n(k)}\setminus B_{n(k)+1}.$ For each $k\in \N,$ set $\phi_{k}=\phi_{n(k),k},$ then $(\phi_{k})_{k}$ are a sequence of topological microstates, but $\lim_{k\to\omega}\frac{1}{d_{k}}\sum_{j=1}^{d_{k}}\rho(\phi_{n(k),k}(j),Y_{k}(j))\geq \varepsilon,$ which is a contradiction.

(\ref{I:finitary version}) implies (\ref{I:absorbs measures}):
Let $\mu \in \mathcal{L}_{\omega}(G\actson X)$ and let $(\mu_{k})_{k}$ be a sequence of probability measures supported on topological microstates which have $\mu=\mathcal{E}((\mu_{k})_{k}).$ Fix $\varepsilon>0,$ it is enough to show that
\[u_{\omega}(\{z:\mu(z)(N_{\varepsilon}(Y(z)))>1-\varepsilon\})=1.\]
For $k\in \N,$ let $E_{k}=\{1\leq j\leq d_{k}: \mu_{k,j}(N_{\varepsilon/2}(Y_{k}(j)))>1-\varepsilon/2\}.$ Then
\[\{z:\mu(z)(N_{\varepsilon}(Y(z)))>1-\varepsilon\}\supseteq \prod_{k\to\omega}E_{k},\]
so it suffices to show that
\begin{equation}\label{E:got to do what ya go to do}
\lim_{k\to\omega}u_{d_{k}}(E_{k})=1.
\end{equation}
By (\ref{I:finitary version}), we may find a finite $F\subseteq G,$ a $\delta>0,$ and a $B\in\omega,$ so that for every $k\in B$ and every $\phi\in \Map(\rho,F,\delta,\sigma_{k})$ we have $\frac{1}{d_{k}}\sum_{j=1}^{d_{k}}\rho(\phi_{k}(j),Y_{k}(j))<\varepsilon^{2}.$ For every $k\in B,$
\begin{align*}
u_{d_{k}}(E_{k}^{c})=u_{d_{k}}(\{1\leq j\leq d_{k}:\mu_{k,j}(N_{\varepsilon/2}(Y_{k}(j))^{c})\geq \varepsilon/2\})&\leq \frac{2}{\varepsilon}\frac{1}{d_{k}}\sum_{j=1}^{d_{k}}\int \rho(\phi,Y_{k}(j))\,d\mu_{k}(\phi)\\
&=\frac{2}{\varespilon}\int_{X^{d_{k}}}\left(\frac{1}{d_{k}}\sum_{j=1}^{d_{k}} \rho(\phi,Y_{k}(j))\right)\,d\mu_{k}(\phi)\\
&\leq 2\varepsilon+M\mu_{k}(\Map(\rho,F,\delta,\sigma_{k})^{c})
\end{align*}
Since $(\mu_{k})_{k}$ is supported on topological microstates, we have
$\lim_{k\to\omega}u_{d_{k}}(E_{k}^{c})\leq 2\varespilon.$
Since $\varepsilon>0$ is arbitrary, we have shown $(\ref{E:got to do what ya go to do})$ and this proves (\ref{I:absorbs measures}).

(\ref{I:absorbs measures}) implies (\ref{I:absorbs stuff alala}):
Let $\Phi\colon Z_{\omega}\to X$ be a topological microstate. Choose a sequence $(\phi_{k})_{k}$ of topological microstates with respect to $\omega$ so that $\Phi=(\phi_{k})_{k\to\omega}$ almost everywhere. Let $\mu_{k}=\delta_{\phi_{k}},$ clearly $(\mu_{k})_{k}\in \mathcal{P}_{\omega}(X,G).$ The map $f\colon X\to \Prob(X)$ given by $f(x)=\delta_{x}$ is continuous, and $\mathcal{E}^{(k)}(\mu_{k})=f\circ \phi_{k}.$ So $\mathcal{E}((\mu_{k})_{k})=(f\circ \phi_{k})_{k\to\omega}=f_{*}(\Phi),$ and so $\mathcal{E}((\mu_{k})_{k})(z)=\delta_{\Phi(z)}$ for all $z\in Z_{\omega}.$  Thus the fact that $\mathcal{E}((\mu_{k})_{k})\preceq Y$ clearly implies that $\Phi(z)\in Y(z)$ for almost every $z\in Z_{\omega}.$

\end{proof}

\begin{prop}\label{L:topological poset}
Let $(d_{k})_{k}$ be a sequence of natural numbers, and $\omega$ a free ultrafilter on the natural numbers. For all $Y\in \Meas(Z_{\omega},\mathcal{F}(X))$ we have that
\[\{K\in \Meas(Z_{\omega},\mathcal{F}(X)):Y\geq K\},\]
\[\{C\in \Meas(Z_{\omega},\mathcal{F}(X)):Y\leq C\}\]
are closed.

\end{prop}

\begin{proof}
We only present the proof that $\{K\in \Meas(Z_{\omega},\mathcal{F}(X)):Y\geq K\}$ is closed, the proof that $\{C\in \Meas(Z_{\omega},\mathcal{F}(X)):Y\leq C\}$ is closed is the same. We know that $\Meas(Z_{\omega},\mathcal{F}(X))$ is metrizable. So it suffices to show that if $K_{n}$ is a sequence in $\Meas(Z_{\omega},\mathcal{F}(X))$ and $K_{n}\leq Y$ for all $n,$ and $K_{n}\to K,$ then $K\leq Y.$ Modifying $Y$ and each $K_{n}$ on a set of measure zero we may assume that $K_{n}(z)\leq Y(z)$ for all $z\in Z_{\omega}.$ By passing to a subsequence, and also modifying $K$ on a set of measure zero, we may assume that $K_{n}(z)\to K(z)$ for all $z\in Z_{\omega}.$ Fix a $z\in Z_{\omega}.$ Since $\{F\in \mathcal{F}(X):F\subseteq Y(z)\}$ is closed, we have that
\[K(z)=\lim_{n}K_{n}(z)\subseteq Y(z).\]
Since $z\in Z_{\omega}$ was arbitrary, we have that $K\leq Y.$

\end{proof}

We now specialize to the case of algebraic actions.  So let now $G,\sigma_{k},\omega$ be as in the setup to Proposition \ref{P:characterize top micro}, but assume now that $X$ is a compact, metrizable group and that $G\actson X$ by continuous automorphisms. We observe that:
\begin{itemize}
\item $\mathcal{L}_{\omega}(G\actson X)$ is a topological monoid under pointwise convolution,
\item $\mathcal{L}_{\omega}(G\actson X)$ is preserved under pointwise convex combinations,
\item $\mathcal{L}_{\omega}(G\actson X)$ is preserved under taking the $^{*}$ operation pointwise.
\end{itemize}
All of these follow from the analogous facts for sequence of topological microstates. For example, the first bullet point follows from the fact that if $(\mu_{k})_{k},(\nu_{k})_{k}$ are supported on topological microstates as $k\to\omega,$ then so is $(\mu_{k}*\nu_{k})_{k}.$
Special to the algebraic action case is a canonical subspace of $\Meas(Z_{\omega},\Sub(X)).$ Observe that we have a natural map $\mathcal{M}\colon \Sub(X)\to \Prob(X)$ given by $\mathcal{M}(Y)=m_{Y}.$ So by Proposition \ref{P:just the basics Loeb}, we have a uniformly continuous map $\mathcal{M}_{*}\colon \Meas(Z_{\omega},\Sub(X))\to \Meas(Z_{\omega},\Prob(X))$ given by $\mathcal{M}_{*}(\Phi)=\mathcal{M}\circ \Phi,$ and $\mathcal{M}_{*}$ is a homeomorphism onto its image. For $Y\in \Meas(Z_{\omega},\Sub(X)),$ we will typically write $m_{Y}$ for $\mathcal{M}_{*}(Y).$ We now let
\[\mathcal{S}_{\omega}(G\actson X)=\{Y\in \Meas(Z_{\omega},\Sub(X)):m_{Y}\in \mathcal{L}_{\omega}(G\actson X)\}.\]
We think of $\mathcal{S}_{\omega}(G\actson X)$ as all measurably varying subgroups which are local weak$^{*}$ limits of measures supported on topological microstates. Important for us is the following topological fact.

\begin{prop}\label{P:topology is helpful 2}
Let $G$ be a countable, discrete, sofic group with sofic approximation $\sigma_{k}\colon G\to S_{d_{k}}.$ Let $G\actson X$ be an algebraic action of $G.$ Then, for every free ultrafilter $\omega$ on the natural numbers, the space $\mathcal{S}_{\omega}(G\actson X)$ is a closed subspace of $\Meas(Z_{\omega},\Sub(X)).$

\end{prop}

\begin{proof}
Since $\mathcal{M}_{*}$ is a homeomorphism onto its image, and
\[\mathcal{S}_{\omega}(G\actson X)=(\mathcal{M}_{*})^{-1}(\mathcal{L}_{\omega}(G\actson X))\]\
this follows from Theorem \ref{T:topology is helpful}.

\end{proof}

%It will also be useful later to state the relation $\mu\succeq Y$ in a more algebraic way.
%
%\begin{prop}\begin{prop}\label{P:topology is helpful 2}
%Let $G$ be a countable, discrete, sofic group with sofic approximation $\sigma_{k}\colon G\to S_{d_{k}}.$ Let $G\actson X$ be an algebraic action of $G.$  Fix a free ultrafilter $\omega$ on the natural numbers, and let $\mu\in \Meas(Z_{\omega},\Prob(X)),Y\in \Meas(Z_{\omega},\Sub(X)).$ Then $\mu\preceq Y$ if and only if $\mu*Y.$

The order structure we put on $\Meas(Z_{\omega},\Sub(X))$ clearly turns $\Meas(Z_{\omega},\Sub(X)),\mathcal{S}_{\omega}(G\actson X)$ into partially ordered sets. These posets end up having nice lattice theoretic properties that greatly ease the proof of our main results. Recall that if $(P,\leq)$ is a partially ordered set, and $a,b\in P,$ then a \emph{join} of $P$ is an element $c\in P$ with $c\geq a,c\geq b$ and so that if $d\in P$ has $d\geq a,d\geq b,$ then $d\geq c.$ If any two elements of $P$ have a join, then $P$ is called a join-lattice. We say that $P$ is a \emph{complete join-lattice} if for $\emph{any}$ collection $(a_{\alpha})_{\alpha\in A}$ of elements of $P,$ there is an element $b\in P$ with $b\geq a_{\alpha}$ for all $\alpha$ and so that whenever $c\geq a_{\alpha}$ for all $\alpha,$ then $c\leq b,$ we call $P$ a \emph{complete join-lattice}.

 We now state the theorem which gives us the precise order properties in order to prove Theorem \ref{T:maintheoremintro}.

\begin{thm}\label{T:main theorem restated}
Let $G$ be a countable, discrete, sofic group with sofic approximation $\sigma_{k}\colon G\to S_{d_{k}}.$ Let $G\actons X$ be an algebraic action of $G,$ and fix a free ultrafilter $\omega$ on the natural numbers. We then have the following permanence properties of $\mathcal{P}_{\omega}(X,G),\mathcal{S}_{\omega}(X,G).$
\begin{enumerate}[(i)]
\item Given  $\mu\in \mathcal{L}_{\omega}(G\actson X),$ define $Y\in \Meas(Z_{\omega},\Sub(X))$ by $Y(z)=\ip{\supp \mu(z)}.$ Then $Y\in \mathcal{S}_{\omega}(G\actson X).$ \label{I: supports and shiz}
\item Given $Y_{1},Y_{2}\in \mathcal{S}_{\omega}(G\actson X),$ define $Y_{1}\vee Y_{2}\in \Meas(Z_{\omega},\Sub(X))$ by $(Y_{1}\vee Y_{2})(z)=\overline{\ip{Y_{1}(z),Y_{2}(z)}}.$ Then $Y_{1}\vee Y_{2} \in \mathcal{S}_{\omega}(G\actson X).$ \label{I:join of subgroups}
\item The spaces $\Meas(Z_{\omega},\Sub(X))$ and $\mathcal{S}_{\omega}(G\actson X)$ are both complete join-lattices. \label{I:lattices and stuff}
\end{enumerate}
\end{thm}

 We remark that is clear from part (\ref{I:join of subgroups})  of the above theorem that $\mathcal{S}_{\omega}(G\actson X),\Meas(Z_{\omega},X)$ are join-lattices.

Crucial in part (\ref{I:lattices and stuff}) of Theorem \ref{T:main theorem restated} is that part of the definition of being a complete join-lattice is being closed under \emph{arbitrary} joins, not just countably infinite ones. This is especially important since $\mathcal{S}_{\omega}(G\actson X)$ is typically \emph{not} separable. We will use Theorem \ref{T:main theorem restated} primarily to note that $\mathcal{S}_{\omega}(G\actson X)$ has a maximal element, which is a triviality from the fact that $\mathcal{S}_{\omega}(G\actson X)$ is a complete join-lattice. However, it is \emph{not} clear how to obtain the existence of a maximal element in $\mathcal{S}_{\omega}(G\actson X)$ if one only knows that $\mathcal{S}_{\omega}(G\actson X)$ is closed under countably infinite joins, since $\mathcal{S}_{\omega}(G\actson X)$ is not separable. This is one of the main difficulties in proving Theorem \ref{T:main theorem restated} ( the fact that  $\mathcal{S}_{\omega}(G\actson X)$ is closed under countably infinite joins is not obvious either), and is one of the main reasons behind the Hilbert space approach we take to proving Theorem \ref{T:main theorem restated}. This  Hilbert space approach also makes the proof of the other parts of Theorem \ref{T:main theorem restated} easy as well).

A bit of an analogy might be helpful to explain why $\Meas(Z_{\omega},\Sub(X))$ is a complete, and not just countably complete, lattice, as well as to explain the  approach we take to proving this fact. Instead of considering $\Meas(Z_{\omega},\Sub(X))$ let us consider the lattice of measurable functions on a probability space $(X,\mu)$ (potentially not countably generated) with values in a simpler partially order set, namely $[0,1]$ with the usual ordering. It is also true that $\Meas(X,\mu,[0,1])$ is a complete lattice. This seems preposterous, given that it is easy to construct an uncountable set $I$ and measurable functions $f_{i}\colon X\to [0,1]$ for $i\in I$ so that the function $f(x)=\sup_{i}f_{i}(x)$ is not measurable. Nevertheless, it is still true that $\Meas(X,\mu,[0,1])$ is a complete lattice. Namely, given any  collection $(f_{i})_{i\in I}$ of elements of $\Meas(X,\mu,[0,1])$ there is a unique (modulo null sets) $f\in \Meas(X,\mu,[0,1])$ which is minimal subject to the condition that for all $i\in I,$ we have that $f\geq f_{i}$ almost everywhere.  This fact is nonobvious, given that applying the supremum operation pointwise fails miserably. Because of this, the simplest proofs we know take a ``non-pointed" approach.

One proof of the completeness of $\Meas(X,\mu,[0,1])$  as join-lattice is to regard $\Meas(X,\mu,[0,1])$ as the subset of $L^{1}(X,\mu)^{*}$ consisting of positive linear functional $\phi\colon L^{1}(X,\mu)\to \C$ which have $\|\phi\|\leq 1.$ The space $L^{1}(X,\mu)^{*}$  has a natural poset structure isomorphic to the order structure on $L^{\infty}(X,\mu)$ and it is easy to check that if $(\phi_{i})_{i\in I}$ are an increasing net of positive linear functionals with $\|\phi_{i}\|\leq 1$ for all $i\in I,$ then there is a unique $\phi\in L^{1}(X,\mu)^{*}$ which satisfies
\[\phi(f)=\sup_{i}\phi_{i}(f)\mbox{ for all $f\in L^{1}(X,\mu)_{+}.$}\]
Further, we have that $\phi$ is positive and $\|\phi\|\leq 1.$ This proof is hard to adapt to our setting. One may regard $\mathcal{S}_{\omega}(G\actson X)$ as a subset of a dual space of a Banach space, but this space is a bit clumsy to work with for our purposes and takes a bit to understand for the uninitiated.

The second proof that $\Meas(X,\mu,[0,1])$  is a join-lattice, which is closer to ours, goes as follows. We may also regard $\Meas(X,\mu,[0,1])$ as a subset of $B(L^{2}(X,\mu))$ by associating each function $f$ to its multiplication operator $M_{f}.$ If $f\in \Meas(X,\mu,[0,1]),$ then $M_{f}$ is a positive operator, and this association is ordering preserving if we use the natural order on self-adjoint elements of $B(L^{2}(X,\mu)).$ Now, it is a fact in Hilbert space theory that if $(T_{i})_{i\in I}$ is a net of positive operators on a Hilbert space, and if
\begin{itemize}
\item $T_{i}\leq T_{j}$ if $i,j\in I$ and $i\leq j,$ and
\item $\|T_{i}\|\leq 1$ for all $i\in I,$
\end{itemize}
then there is a unique $T\in B(L^{2}(X,\mu))$ which is self-adjoint and is minimal subject to the condition that $T\geq T_{i}$ for all $i\in I.$ Indeed, for each $\xi\in L^{2}(X,\mu),$ the limit $\lim_{i}\ip{T_{i}\xi,\xi}$ exists and is $\sup_{i}\ip{T_{i}\xi,\xi},$ and so it can be shown that $T$ is the weak operator topology limit of the $T_{i}$ (it is in fact the strong operator topology limit of the $T_{i},$ but this is not necessary to know for our purposes). Since $\Meas(X,\mu,[0,1])$ is closed under finite joins, it just remains to note that $\{M_{f}:f\in L^{\infty}(X,\mu)\}$ is weak operator topology closed, which a is well known fact.

There is a completely natural way to associate to any measure on a compact group $X$ an operator on $L^{2}(X),$ and this association sends Haar measures on subgroups to orthogonal projections. Combining this with the natural way to represent $L^{\infty}(Z_{\omega},u_{\omega})$ as operators on $L^{2}(Z_{\omega},u_{\omega})$ will allows to generalize the above argument from the case of $\Meas(X,\mu,[0,1])$ to $\Meas(Z_{\omega},u_{\omega},\Sub(X)).$

We postpone the proof of Theorem \ref{T:main theorem restated} to Section \ref{S:main proof}, so that we can postpone stating the minor amount of  additional  background required for the proof. To convince the reader that this small additional background is worth understanding the proof of Theorem \ref{T:main theorem restated}, we   proceed in the next subsection to explain how Theorem \ref{T:main theorem restated} proves our main theorem, as well as give several corollaries to Theorem \ref{T:main theorem restated}.

\subsection{Deduction of the main theorem from the main reduction, and applications of the main result}\label{S:major applications}

In this section, we give several corollaries to Theorem \ref{T:main theorem restated}. We first explain how Theorem \ref{T:main theorem restated} implies the main theorem from the introduction, Theorem \ref{T:maintheoremintro}.
We in fact deduce a more general form of Theorem \ref{T:maintheoremintro}.

\begin{cor}\label{C:this the main theorem yo!}
Let $G$ be a countable, discrete, sofic group with sofic approximation $\sigma_{k}\colon G\to S_{d_{k}}.$ Let $G\actson X$ be an algebraic action of $G,$ and fix a free ultrafilter $\omega$ on the natural numbers. Let $Y$ be the maximal element of $\mathcal{S}_{\omega}(G\actson X)$ (whose existence is guaranteed by Theorem \ref{T:main theorem restated} (\ref{I:lattices and stuff})).  Then $Y$ is also the minimal element of $\Meas(Z_{\omega},\Sub(X))$ which absorbs all topological microstates.

\end{cor}

\begin{proof}
Let $Y$ be the maximal element of $\mathcal{S}_{\omega}(G\actson X),$ this exists (and is unique) because $\mathcal{S}_{\omega}(G\actson X)$ is a complete join-lattice. We only have to prove two claims.

\emph{Claim 1: $Y$ absorbs all topological microstates.}
To prove this, let $\mu \in \Meas(Z_{\omega},\Prob(X))$ and let $Y'\in \Meas(Z_{\omega},\Sub(X))$ be given by $Y'(z)=\overline{\ip{\supp \mu(z)}}.$ By Theorem \ref{T:main theorem restated} (\ref{I: supports and shiz}) we then have that $Y'\in \mathcal{S}_{\omega}(X,G).$ Hence $Y'\leq Y,$ and thus we have that $\mu\preceq Y.$ By Proposition \ref{P:what it means to absorb}, it follows that $Y$ absorbs all topological microstates.

\emph{Claim 2: If $\tilde{Y}\in \Meas(Z_{\omega},\Sub(X))$ absorbs all topological microstates, then $\tilde{Y}\geq Y.$}
Since $Y\in \mathcal{S}_{\omega}(X,G)$ we have that $m_{Y}\in \mathcal{P}_{\omega}(X,G),$ so by Proposition \ref{P:what it means to absorb}  it follows that $m_{Y}\preceq \tilde{Y}.$ But this clearly implies that $Y\leq \tilde{Y}.$

\end{proof}

As mentioned before, this gives us an equivalent  formulation of when the Haar measure is a local weak$^{*}$ limit of measures supported on topological microstates.

\begin{cor}\label{C:sick corollary  bro}
Let $G$ be a countable, discrete, group and let $\sigma_{k}\colon G\to S_{d_{k}}$ be a sofic approximation. Let $G\actson X$ be an algebraic action and fix a free ultrafilter $\omega$ on the natural numbers. The following are equivalent:
\begin{enumerate}[(i)]
\item \label{I:nonexistence shiz} there does not exists a sequence $\mu_{k}\in \Prob(X^{d_{k}})$ which is asymptotically supported on topological microstates as $k\to\omega$ and so that $\mu_{k}\to^{lw^{*}}m_{X}$ as $k\to\omega,$
\item there is an sequence $(Y_{k})_{k}\in \prod_{k}\Sub(X)^{d_{k}}$ so that $\lim_{k\to\omega}(Y_{k})_{*}(u_{d_{k}})\ne \delta_{X}$ and so that $(Y_{k})_{k}$ absorbs all topological microstates. \label{I:absorbing shiz}
\end{enumerate}
\end{cor}

\begin{proof}

(\ref{I:nonexistence shiz}) implies (\ref{I:absorbing shiz}): Let $Y$ be the maximal element of $\mathcal{S}_{\omega}(X,G),$ and let $(Y_{k})_{k}\in \prod_{k}\Sub(X)^{d_{k}}$ be such that $Y=(Y_{k})_{k\to\omega}$ almost everywhere. By (\ref{I:nonexistence shiz}) we have $Y\ne X,$ so $u_{\omega}(\{z:Y(z)\ne X\})>0.$ By Proposition \ref{P:computing integrals in ultraproducts} we have
\[\lim_{k\to\omega}(Y_{k})_{*}(u_{d_{k}})=Y_{*}(u_{\omega})\ne \delta_{X},\]
the last part following as $u_{\omega}(\{z:Y(z)\ne X\})>0.$ By Corollary \ref{C:this the main theorem yo!}, we have that $(Y_{k})_{k}$ absorbs all topological microstates along $\omega.$

(\ref{I:absorbing shiz}) implies (\ref{I:nonexistence shiz}): Let $Y=(Y_{k})_{k\to\omega},$ and let $\tilde{Y}$ be the maximal element of $\mathcal{S}_{\omega}(X,G)$ as in Corollary \ref{C:this the main theorem yo!}. Since $Y$ absorbs all topological microstates, we have that $\tilde{Y}\leq Y$ by Corollary \ref{C:this the main theorem yo!}. By Proposition \ref{P:computing integrals in ultraproducts}, we have that $\lim_{k\to\omega}(Y_{k})_{*}(u_{d_{k}})=Y_{*}(u_{\omega}).$ So $Y_{*}(u_{\omega})\ne \delta_{X},$ and thus $u_{\omega}(\{z:Y(z)\ne X\})>0,$ and since $\tilde{Y}(z)\subseteq Y(z)$ for $u_{\omega}$-almost every $z,$ we see that $\tilde{Y}\ne X.$ Since $\tilde{Y}$ is the maximal element of $\mathcal{S}_{\omega}(X,G)$ this clearly implies $(\ref{I:nonexistence shiz}).$

\end{proof}

The two theorems above may be regarded as the main theorems of the paper. They become drastically easier to state when the centralizer of $\sigma_{\omega}(G)$ acts ergodically, as we proceed to show now.

\begin{defn} Let $G$ be a countable discrete sofic group with sofic approximation $\sigma_{k}\colon G\to S_{d_{k}}.$ Fix a free ultrafilter $\omega$ on the natural numbers. We let $G'_{\omega}$ the centralizer of $\sigma_{\omega}(G)$ inside of $\prod_{k\to\omega}S_{d_{k}}.$
\end{defn}

\begin{cor}\label{L:ergodicity on Loeb space}
Let $G$ be a countable, discrete, group with sofic approximation $\sigma_{k}\colon G\to S_{d_{k}}.$ Let $G\actson X$ be an algebraic action, and fix a free ultrafilter $\omega$ on  the natural numbers. Suppose  that $G'_{\omega}\actson (Z_{\omega},u_{\omega})$ is ergodic. Let  $Y$ be the maximal element of $\mathcal{S}_{\omega}(X,G).$ Then $Y$ is, in fact, a subgroup of $X.$ In particular, the maximal subgroup $Y$ of $X$ so that there exists a sequence $\mu_{k}\in \Prob(X^{d_{k}})$ which is asymptotically supported on topological microstates and has $\mu_{k}\to^{lw^{*}}m_{Y}$ as $k\to\omega$ is equal to the minimal subgroup of $X$ which absorbs all topological microstates for $G\actons X.$

\end{cor}

\begin{proof}
It is clear that $G'_{\omega}$ acts on $\Meas(Z_{\omega},X)$ by $(\tau\cdot \Theta)(z)=\Theta(\tau^{-1}(z))$ for $\tau\in G'_{\omega},\Theta\in \Meas(Z_{\omega},X).$ It is also clear that $G'_{\omega}$ preserves the subset of $\Meas(Z_{\omega},X)$ consisting of topological microstates. Thus, by uniqueness of $Y,$ it follows that $Y\circ \tau^{-1}=Y$ for all $\tau\in G'_{\omega}.$ By ergodicity of $G'_{\omega}\actson Z_{\omega}$ it thus follows that $Y$ is essentially constant, i.e. $Y\in \Sub(X).$ The ``in particular" part is automatic from Corollary \ref{C:this the main theorem yo!}.

\end{proof}

\begin{cor}\label{C:simpler under ergodicity} Let $G$ be a countable, discrete, sofic group and let $(\sigma_{k})_{k}$ be a sofic approximation of $G.$ Fix a free ultrafilter $\omega$ on the natural numbers, and let $G\actson X$ be an algebraic action.  Suppose that $G'_{\omega}\actson (Z_{\omega},u_{\omega})$ is ergodic. Then the following are equivalent:
\begin{enumerate}[(i)]
\item there does not exist a sequence $\mu_{k}\in \Prob(Y^{d_{k}})$ so that $\mu_{k}\to^{lw^{*}}m_{X},$ \label{I:non existence again}
\item there is a closed, proper, $G$-invariant subgroup $Y$ of $X$ which absorbs all topological microstates for $G\actson X.$ \label{I:proper subgroup again}
\end{enumerate}
\end{cor}

\begin{proof}

This is proved exactly as in Corollary \ref{C:sick corollary  bro}.

\end{proof}

The above four theorems thus comprise the main theorems of the paper. To illustrate the utility of these results, we deduce the remaining main theorems of the introduction from them, as well as several other applications.

Most of our applications will be to sofic entropy theory, and we recall some of those basic notions here. Suppose that $G$ is a sofic group with sofic approximation $\sigma_{k}\colon G\to S_{d_{k}}.$ Let $X$ be a compact, metrizable space with $G\actson X$ by homeomorphisms.  Kerr-Li defined the topological entropy of $G\actson X$ with respect to $(\sigma_{k})_{k}$ in \cite{KLi}, which we denote by $h_{(\sigma_{k})_{k}}(G\actson X).$  Topological entropy is a way to ``measure" the ``size" of the topological microstates space. See \cite[Definition 2.2]{KLi2} for the definition. Their work followed the pioneering work of Bowen in \cite{Bow} who first defined sofic measure entropy for a large class of probability measure-preserving actions. Their work extended Bowen's the measure entropy for all probability measure-preserving actions. Given a free ultrafilter $\omega$ on the natural numbers, one may also define the topological entropy of $G\actson X$ with respect to $(\sigma_{k})_{k},\omega$ be replacing the limit supremum with a limit along an ultrafilter. We denote this entropy by $h_{(\sigma_{k})_{k\to\omega},\topo}(G\actson X).$ Similarly, one may define the \emph{lower} topological entropy by replacing the limit supremum in \cite[Definition 2.2]{KLi2} by a limit infimum. We will denote this entropy by $\underline{h}_{(\sigma_{k})_{k},\topo}(G\actson X).$ If we also have a $G$-invariant, completed Borel probability measure $\mu$ on $X,$ then we can define the measure entropy of $G\actson (X,\mu)$ by $h_{(\sigma_{k})_{k}}(G\actson (X,\mu)).$   See \cite[Definition 3.3]{KLi2} for a definition well-suited to our setting. We again remark that  replacing the limit supremum  in\cite[Definition 3.3]{KLi2} with a limit infimum gets us a lower measure entropy, denoted $\underline{h}_{(\sigma_{k})_{k}}(G\actson (X,\mu)),$ we also consider the measure entropy with respect to an ultrafilter which we denote $h_{(\sigma_{k})_{k\to\omega}}(G\actson (X,\mu)).$

If $Y$ is another compact, metrizable space with $G\actson Y$ by homeomorphisms, then a $G$-equivariant, surjective, continuous map $f\colon X\to Y$ will be called a \emph{factor} map, and $Y$ will be called a \emph{factor} of $X.$ In this setup, there is the notion of the \emph{entropy of $Y$ in the presence of $X$}. This was first defined in the measure context implicitly in the work of Kerr in \cite{KerrPartition}, and it was explicitly written down in the measure context in \cite{Me9}. Li-Liang then defined the topological notion of entropy in the presence in \cite{LiLiang2}. See \cite{Me12} Definitions 3.3, 3.4  for definitions suited for our purposes. We denote the topological entropy of $Y$ in the presence of $X$ by $h_{(\sigma_{k})_{k}}(G\actson Y:X).$ If $\mu$ is a $G$-invariant, completed Borel probability measure on $X$ and $\nu=f_{*}(\mu),$ we also have the measure entropy of $G\actson (Y,\nu)$ in the presence of  $G\actson (X,\mu),$ denoted $h_{(\sigma_{k})_{k}}(G\actson (Y,\nu):(X,\mu)).$ The remarks of the preceding paragraph for ultralimits as well as lower topological entropy in the presence apply mutatis mutandis to the case of topological/measure entropy in the presence.

We say that $G\actson X$ has completely positive topological entropy  if whenever $G\actson Y$ is a factor of $G\actson X,$ and $|Y|\geq 2,$ then the topological entropy of $G\actson Y$ is positive. If $\mu\in \Prob_{G}(X)$ we say that $G\actson (X,\mu)$ has completely positive measure entropy if whenever $G\actson (Z,\zeta)$ is a measurable factor of $G\actson (X,\mu)$ with $(Z,\zeta)$ Lebesgue and $\zeta$ is not a point mass, we have that the entropy of $G\actson (Z,\zeta)$ is positive. Similar remarks apply to completely positive lower topological entropy, or completely positive entropy with respect to an ultrafilter in either the topological or measurable contexts. Lastly, we say that $G\actson X$ has completely positive topological entropy in the presence if given any factor $Y$ of $X$ with $|Y|\geq 2$ we have that $h_{(\sigma_{k})_{k}}(G\actson Y:X)>0.$ Obvious modifications give the notion of completely positive measure entropy in the presence, as well as completely positive lower topological/measure entropy in the presence, or completely positive topological/measure entropy in the presence with respect to an ultrafilter.

We start with the equivalence of complete positive topological and complete positive measure entropy for algebraic actions, when we are the setting of a sofic approximation with ergodic centralizer. We need two preliminary results to prove this, as well as the notion of coinduced actions. Suppose that $H\leq G$ are countable, discrete groups, and that $Y$ is a compact, metrizable space with $H\actson Y$ by homeomorphisms. Let $X$ be the space of all functions $f\colon G\to Y$ so that $f(gh)=h^{-1}f(g)$ for all $g\in G,h\in H.$ Giving $X$ the product topology, we know that $X$ is a compact, metrizable space by Tychonoff's theorem. We define an action $G\actons X$ by $(gf)(x)=f(g^{-1}x)$ for all $f\in X,g,x\in G.$ The action $G\actson X$ is called the coinduced action of $H\actson Y.$

\begin{prop}\label{P:top cpe passes to subgroups}
Let $G$ be a countable, discrete, sofic group with sofic approximation $\sigma_{k}\colon G\to S_{d_{k}}.$ Let $X$ be a compact, metrizable space with $G\actson X$ by homeomorphisms. If $G\actson X$ has completely positive topological entropy with respect to $(\sigma_{k})_{k}$, then for every $H\leq G,$ we have that $H\actson X$ has completely positive topological entropy.
\end{prop}

\begin{proof}
Suppose that $H\actson Y$ is a topological factor of $H\actson X,$ with factor map $\phi\colon X\to Y,$ and so that $|Y|\geq 2.$ Let $G\actson Z$ be the coinduced topological action of $H\actson Y.$ Define $\psi\colon X\to Z$ by $\psi(x)(g)=\phi(g^{-1}x)$ for $x\in X,g\in G.$ It is not hard to show that $\psi$ is well-defined, namely that $\psi(x)(gh)=h^{-1}(\psi(x)(g))$ for all $x\in X,h\in H,g\in G.$ It is also straightforward to show that $\psi$ is $G$-equivariant. Also, since $\psi(x)(e)=\phi(x)$ for all $x\in X,$ we know that $\psi$ is nonconstant. Since $G\actson X$ has completely positive topological entropy, we have that
\[0<h_{(\sigma_{k})_{k},\topo}(G\actson \psi(X))\leq h_{(\sigma_{k})_{k},\topo}(G\actson Z)=h_{(\sigma_{k}|_{H})_{k},\topo}(H\actson Y),\]
the last line following by \cite[Proposition 6.22]{Me5}. Thus $0<h_{(\sigma_{k}|_{H})_{k},\topo}(H\actson Y)$ and  we have shown that $H\actson X$ has completely positive topological entropy.

\end{proof}

\begin{lem}\label{L:cpe implies mixing}
Suppose that $G$ is a countable, discrete, sofic group with sofic approximation $\sigma_{k}\colon G\to S_{d_{k}}.$ Let $X$ be a compact, metrizable group with $G\actson X$ by automorphisms. If $G\actson X$ has completely positive entropy with respect to $(\sigma_{k})_{k},$ then $G\actson (X,m_{X})$ is mixing.

\end{lem}

\begin{proof}

Suppose that $G\actson (X,m_{X})$ is not mixing. By \cite[Theorem 1.6]{Schmidt}, we can choose a nontrivial irreducible representation $\pi\colon X\to \mathcal{U}(\mathcal{H})$ so that
\[H=\{g\in G:\pi\circ \alpha_{g}\cong \pi\}\]
is infinite.
For $h\in H,$ we may choose a $U_{h}\in \mathcal{U}(\mathcal{H})$ so that $\pi(\alpha_{h}(x))=U_{h}\pi(x)U_{h}^{*}.$ Notice that if $V_{h}$ is another such unitary, then $V_{h}^{*}U_{h}\in \Hom(\pi,\pi)=\C1,$ by Schur's Lemma. Thus $U_{h}$ is well-defined in $\mathcal{U}(\mathcal{H})/S^{1}1.$ So we have a well-defined action $H\actson^{\beta} \mathcal{U}(\mathcal{H})$ by
$\beta_{h}(V)=U_{h}VU_{h}^{*}.$

With the given action of $H$ on $X,$ $\mathcal{U}(\mathcal{H}),$ we now have that $\pi$ is an $H$-equivariant, continuous map. Moreover, $H\actson \mathcal{U}(\mathcal{H})$ has an $H$-invariant metric, namely the distance given by the operator norm. Since $H$ is infinite, and $H\actson \mathcal{U}(\mathcal{H})$ is isometric, we know that $H\actson \mathcal{U}(\mathcal{H})$ has entropy at most zero by \cite[Theorem 8.1]{KerrLi2}. Thus $H\actson \pi(X)$ is a topological $H$-factor with entropy at most zero.
 By Proposition \ref{P:top cpe passes to subgroups}, we know that $G\actson X$ does not have completely positive topological entropy.

\end{proof}

Suppose that $G$ is a countable, discrete, sofic group with sofic approximation $(\sigma_{k})_{k}.$ Let $X$ be a compact, metrizable space with $G\actson X$ by homeomorphisms, and let $\mu\in \Prob_{G}(X).$ Recall that $G\actson (X,\mu)$ is \emph{strongly sofic} with respect to $(\sigma_{k})_{k}$ if there is a sequence $\mu_{k}\in \Prob(X^{d_{k}})$ so that:
\begin{itemize}
\item $\mu_{k}\to^{lw^{*}}\mu,$
\item $\mu_{k}$ is asymptotically supported on topological microstates,
\item $\mu_{k}\otimes \mu_{k}\left(\left\{(\phi,\psi):\left|\frac{1}{d_{k}}\sum_{j=1}^{d_{k}}f(\phi(j),\psi(j))-\int f\,d(\mu\otimes \mu)\right|<\varepsilon\right\}\right)\to 1$ for all $f\in C(X\times X),$ $\varepsilon>0.$
\end{itemize}
Given a free ultrafilter $\omega$ on the natural numbers, we can make sense of $\mu_{k}\to^{lde}_{k\to\omega}\mu,$ by replacing the limits in each of the three items with limits along  $\omega,$ and in the second item only requiring that $\mu_{k}$ be asymptotically supported on topological microstates as $k\to\omega.$ We say that $G\actson (X,\mu)$ is strongly sofic with respect to $(\sigma_{k})_{k},\omega$ if there is a sequence $(\mu_{k})_{k}\in \prod_{k}\Prob(X^{d_{k}})$ with $\mu_{k}\to^{lde}_{k\to\omega}\mu.$

\begin{cor}\label{C: equivalence cpe ultrafilter} Let $G$ be a countable, discrete, sofic group and let $(\sigma_{k})_{k}$ be a sofic approximation of $G.$ Fix a free ultrafilter $\omega$ on the natural numbers, and let $G\actson X$ be an algebraic action.  Suppose that $G'_{\omega}\actson (Z_{\omega},u_{\omega})$ is ergodic. If $G\actson X$ has complete positive topological entropy in the presence with respect to $(\sigma_{k})_{k},\omega,$ then $G\actson (X,m_{X})$ is strongly sofic with respect to $(\sigma_{k})_{k},\omega.$ In particular, $G\actons (X,m_{X})$ has completely positive measure entropy in the presence with respect to $(\sigma_{k})_{k},\omega.$

\end{cor}
\begin{proof}
The ``in particular" part follows from \cite[Corollary 1.4]{Me13}. So it suffices to show that $G\actson (X,m_{X})$ is strongly sofic with respect to $(\sigma_{k})_{k},\omega.$ By Lemma \ref{L:cpe implies mixing} we know that $G\actson (X,m_{X})$ is mixing, so by \cite[Lemma 5.15]{AustinAdd} it suffices to show that there is a $\mu_{k}\in \Prob(X^{d_{k}})$ which is asymptotically supported on topological microstates and  has $\mu_{k}\to^{lw^{*}}m_{X}$ as $k\to\omega.$ If this is false, then by Corollary \ref{C:simpler under ergodicity} we may find a closed, proper, $G$-invariant subgroup $Y$ of $X$ which absorbs all topological microstates for $G\actson X$  along $\omega.$ Since $Y$ absorbs all topological microstates, it is easy to see that $h_{(\sigma_{k})_{k\to\omega},\topo}(G\actson X/Y:X)=0.$ This contradicts the hypothesis that $G\actson X$ has completely positive topological entropy in the presence with respect to $\omega.$

\end{proof}

\begin{cor}\label{C:weak assumption implies soficity} Let $G$ be a countable, discrete, sofic group and let $(\sigma_{k})_{k}$ be a sofic approximation of $G.$ Fix a free ultrafilter $\omega$ on the natural numbers, and let $G\actson X$ be an algebraic action.  Suppose that $G'_{\omega}\actson (Z_{\omega},u_{\omega})$ and $G\actson (X,m_{X})$ are ergodic. If $h_{(\sigma_{k})_{k\to\omega}}(G\actson (X,m_{X}))\ne \infty,$ then $G\actson (X,m_{X})$ is strongly sofic with respect to $(\sigma_{k})_{k},\omega.$ In particular,
\[h^{lde}_{(\sigma_{k})_{k\to\omega}}(G\actson (X,m_{X}))=h_{(\sigma_{k})_{k\to\omega}}(G\actson (X,m_{X}))=h_{(\sigma_{k})_{k\to\omega},\topo}(G\actson X).\]
\end{cor}

\begin{proof}
The ``in particular" part follows from \cite[Theorem 1.1]{Me12}. Since $G\actson (X,m_{X})$ is ergodic, it suffices by \cite[Lemma 5.15]{AustinAdd},\cite[Corollary 2.14]{Me12} to show there is a sequence $\mu_{k}\in \Prob(Y^{d_{k}})$ with $\mu_{k}\to m_{X}$ locally weak$^{*}$ as $k\to\omega.$ Let $Y$ be the maximal element of $\mathcal{S}_{\omega}(G\actson X),$ by Corollary \ref{L:ergodicity on Loeb space} we have that $Y\in \Sub(X).$ Since $h_{(\sigma_{k})_{k\to\omega}}(G\actson (X,m_{X}))\ne -\infty,$ we may find a topological microstate $\theta\colon Z_{\omega}\to X$ so that $\theta_{*}(u_{\omega})=m_{X}.$ By Corollary \ref{C:this the main theorem yo!} must have that $\theta(z)\in Y$ for almost every $z\in Z_{\omega},$ and this implies that $X=\supp(\theta_{*}(Z_{\omega}))\subseteq Y.$ So $X=Y\in \mathcal{S}_{\omega}(G\actson X),$ so there is a sequence $\mu_{k}\in \Prob(Y^{d_{k}})$ with $\mu_{k}\to m_{X}$ locally weak$^{*}$ as $k\to\omega.$

\end{proof}

We can deduce ultrafilter-free versions of the above two results. The main new ingredient we need is the following Proposition which allows us to deduce ultrafilter-free versions of our results from the ultrafilter versions.

\begin{prop}\label{P: this is not hard}
Let $G$ be a countable, discrete, sofic group with sofic approximation $\sigma_{k}\colon G\to S_{d_{k}}.$ Let $(X,\mu)$ be a Lebesgue space and $G\actson (X,\mu)$ a probability measure preserving action. Suppose that for every free ultrafilter $\omega$ on the natural numbers, there is a sequence $\nu_{k}\in \Prob(X^{d_{k}})$ (depending upon $\omega$) which is asymptotically supported on topological microstates and so that $\nu_{k}\to^{lw^{*}}\mu$ as $k\to\omega.$ Then there is a sequence $\mu_{k}\in \Prob(X^{d_{k}})$ which is asymptotically supported on topological microstates and so that $\mu_{k}\to^{lw^{*}}\mu$ as $k\to\infty.$

\end{prop}

\begin{proof}

By \cite[Proposition 5.16]{AustinAdd}, the existence of a sequence of measures which are asymptotically supported on topological microstates and so that $\nu_{k}\to^{lw^{*}}\mu$ as $k\to\omega$ (or $k\to\infty)$ does not depend upon a choice of a topological model for $G\actson (X,\mu).$ So we may assume that $X$ is a compact, metrizable space and that $G\actson X$ by homeomorphisms.

We make the following claim.

\emph{Claim. For any weak$^{*}$ neighborhood $\mathcal{O}$ of $\mu,$ for any finite $F\subseteq G,$ and any $\delta>0,$ there is an integer $K$ so that for all $k\geq K,$ there is a $\nu\in \Prob(X^{d_{k}})$ with $\nu(\Map(F,\delta,\sigma_{k}))\geq 1-\delta$ and $u_{d_{k}}(\{j:\nu_{j}\in \mathcal{O}\})\geq 1-\delta.$ }

If we assume the claim is true, then the fact that there exists a sequence $\mu_{k}\in \Prob(X^{d_{k}})$ with $\mu_{k}\to^{lw^{*}}$ as $k\to\infty$ is a simple diagonal argument.

We prove the claim by contradiction, so assume that the claim is false. Then we may find a weak$^{*}$-neighborhood $\mathcal{O}$ of $\mu,$ a finite $F\subseteq G$ and a $\delta>0,$ and a increasing sequence $(k_{n})_{n}$ of natural numbers with $k_{n}\to\infty,$ and which satisfy the following property: for every $n\in \N,$ and every $\nu\in \Prob(X^{d_{k_{n}}})$ with $\nu(\Map(\rho,F,\delta,\sigma_{k_{n}}))\geq 1-\delta,$ we have that $u_{d_{k}}(\{j:\nu_{j}\in \mathcal{O}\})<1-\delta.$ Let $\omega$ be any free ultrafilter on the natural numbers which has $\{k_{n}:n\in \N\}\in \omega.$ Choose a sequence $\nu_{k}\in \Prob(X^{d_{k}})$ so that $\nu_{k}\to^{lw^{*}}\mu$ as $k\to\omega.$

We may then choose a $B\in\omega$ so that for all $k\in B$ we have:
\begin{itemize}
\item $\nu_{k}(\Map(\rho,F,\delta,\sigma_{k}))\geq 1-\delta,$
\item $u_{d_{k}}(\{j:\nu_{k,j}\in \mathcal{O}\})\geq 1-\delta.$
\end{itemize}
Then $B\cap \{k_{n}:n\in \N\}\in \Omega,$ and is thus not empty. Hence, we can find a natural number $n$ so that $k_{n}\in B.$ But then the above two items contradict our choice of $B.$ Thus we have a contradiction, and this proves the proposition.

\end{proof}

We now obtain ultrafilter-free versions of Corollaries \ref{C: equivalence cpe ultrafilter} and \ref{C:weak assumption implies soficity}. Recall that if $G$ is countable, discrete, sofic group with sofic approximation $\sigma_{k}\colon G\to S_{d_{k}},$ and $G\actson (X,\mu)$ is a probability measure-preserving action, then $G\actson (X,\mu)$ is sofic with respect to $(\sigma_{k})_{k}$ if for all $A_{1},\cdots,A_{r}$ measurable subsets of  $X,$ there is a sequence $\phi_{k}\colon \{1,\dots,d_{k}\}\to X$ so that:
\[\left|u_{d_{k}}\left(\bigcap_{j=1}^{l}\sigma_{k}(g_{j})\phi_{k}^{-1}(A_{p(j)}))\right)-\mu\left(\bigcap_{j=1}^{l}g_{j}A_{p(j)}\right)\right|\to 0,\]
 for all $l\in \N,g_{1},\cdots,g_{l}\in G,$ and $p\colon\{1,\dots,l\}\to \{1,\dots,r\}$.

We also need the notion of ergodic commutant for a sofic approximation.

\begin{defn}\label{D:ergodic commutant no ultrafilter}
Let $G$ be a countable, discrete, group and $\sigma_{k}\colon G\to S_{d_{k}}$ a sofic approximation. The commutant $\mathcal{G}'$ of $(\sigma_{k})_{k}$ is the subset of $\prod_{k}\Sym(d_{k})$ consisting of sequences $(\tau_{k})_{k}$ which satisfy
\[\lim_{k\to\infty}d_{\Hamm}(\sigma_{k}(g)\tau_{k},\tau_{k}\sigma_{k}(g))=0\]
 for all $g\in G.$ We say that $(\sigma_{k})_{k}$ has ergodic commutant if given any sequence $(A_{k})_{k}\subseteq \{1,\dots,d_{k}\}$ with $\liminf_{k\to\infty}u_{d_{k}}(A_{k})>0,$ and any $\varepsilon>0,$ there is an $r\in \N,$ and $(\tau_{1,k})_{k},\cdots,(\tau_{r,k})_{k}\in \mathcal{G}'$ so that
\[\liminf_{k\to\infty}u_{d_{k}}\left(\bigcup_{j=1}^{r}\tau_{j,k}(A_{k})\right)\geq 1-\varepsilon.\]
\end{defn}

It is straightforward to check that $(\sigma_{k})_{k}$ has ergodic commutant if and only if $G'_{\omega}\actson Z_{\omega}$ is ergodic for every free ultrafilter $\omega$ on the natural numbers.  A nice example of a sofic approximation which has ergodic commutant is the following: if $G$ is residually finite, and $G_{k}\triangleleft G$ is a decreasing sequence, with $[G:G_{k}]<\infty,$ and $\bigcap_{k=1}^{\infty}G_{k}=\{1\},$ then the sofic approximation $\sigma_{k}\colon G\to \Sym(G/G_{k})$ given  by $\sigma_{k}(x)(gG_{k})=xgG_{k}$ has ergodic commutant (see \cite[Theorem 5.7]{KerrLi2}).

\begin{cor}
Let $G$ be a countable, discrete, sofic group with sofic approximation $\sigma_{k}\colon G\to S_{d_{k}}.$ Suppose that $(\sigma_{k})_{k}$ has ergodic commutant. Let $G\actson X$ be an algebraic action. If $G\actson (X,m_{X})$ is sofic with respect to $(\sigma_{k})_{k},$ and $G\actson (X,m_{X})$ is ergodic, then $h_{(\sigma_{k})_{k},\topo}(G\actson X)=h_{(\sigma_{k})_{k}}(G\actson (X,m_{X})).$
\end{cor}

\begin{proof}
It is not hard to show that $G\actson (X,m_{X})$ is sofic with respect to $(\sigma_{k})_{k}$ if and only if $G\actson (X,m_{X})$ is sofic with respect to $(\sigma_{k})_{k},\omega$ for every free ultrafilter $\omega.$ So if $G\actson (X,m_{X})$ is sofic with respect to $(\sigma_{k})_{k}$, then it follows by Proposition \ref{P: this is not hard}, Corollary \ref{C: equivalence cpe ultrafilter} and \cite[Corollary 2.14]{Me12} that $G\actson (X,m_{X})$ is strongly sofic with respect to $(\sigma_{k})_{k}.$ So the Corollary follows from \cite[Theorem 1.1]{Me12}.

\end{proof}

\begin{cor}
Let $G$ be a countable, discrete, sofic group with sofic approximation $\sigma_{k}\colon G\to S_{d_{k}}.$ Suppose that $(\sigma_{k})_{k}$ has ergodic commutant. Let $G\actson X$ be an algebraic action. Suppose that $G\actson X$ has completely positive lower topological entropy in the presence with respect to $(\sigma_{k})_{k}$, then $G\actson (X,m_{X})$ has completely positive lower measure-theoretic entropy in the presence with respect to $(\sigma_{k})_{k}$.

\end{cor}

\begin{proof}

If $G\actson X$ has completely positive lower topological entropy in the presence with respect to $(\sigma_{k})_{k}$, then it has completely positive topological entropy in the presence with respect to $(\sigma_{k})_{k},\omega$ for every free ultrafilter $\omega.$ It now follows from  Proposition \ref{P: this is not hard}, Corollary \ref{C:weak assumption implies soficity} and \cite[Corollary 2.14]{Me12} that $G\actson (X,m_{X})$ is strongly sofic with respect to $(\sigma_{k})_{k}.$ So the Corollary follows from \cite[Theorem 1.1]{Me12}.

\end{proof}

Our next application will be to showing that the topological entropy of an algebraic action can be realized as (the lw$^{*}$) measure entropy of the Haar measure of a $G$-invariant \emph{random} subgroup of $X.$ As usual, if we deal with sofic approximations with ergodic commutant then this random subgroup can be replaced with an actual subgroup.

For the result, we recall the definition of entropy and lw$^{*}$ measure entropy, along with some necessary definitions.

\begin{defn}
Let $X$ be a compact, metrizable space, and $\rho$ a continuous pseudometric on $X.$ For $A\subseteq X,$ and $\varpesilon>0,$ we let $N_{\varpesilon}(A,\rho)=\{b\in X:\rho(b,a)<\varepsilon\}.$
For $\varepsilon>0,$ and $A\subseteq X,$ we let $S_{\varepsilon}(A,\rho)$ be the minimal cardinality of a subset $B$ of $A$ with $N_{\varepsilon}(B,\rho)\supseteq A.$ We let $P_{\varepsilon}(A,\rho)$ be the largest cardinality of a subset $B$ of $A$ so that $\rho(b_{1},b_{2})>\varepsilon$ for all $b_{1}\ne b_{2}$ in $B.$
For $\mu\in \Prob(X),$ and $\varepsilon,\delta>0,$ we let $S_{\varepsilon,\delta}(\mu,\rho)$ be the minimum of $|A|$ over all finite sets $A\subseteq X$ which have $\mu(N_{\varepsilon}(A,\rho))\geq 1-\delta.$ For a free ultrafilter $\omega\in \beta\N\setminus \N,$ a sequence of natural numbers $(d_{k})_{k},$ and $(\mu_{k})_{k}\in \prod_{k}\Prob(X^{d_{k}}),$ we let
\[h_{\varepsilon,\delta}((\mu_{k})_{k\to\omega})=\lim_{k\to\omega}\frac{1}{d_{k}}\log S_{\varepsilon,\delta}(\mu_{k},\rho_{2}),\]
\[h((\mu_{k})_{k\to\omega})=\sup_{\varepsilon,\delta>0}h_{\varepsilon,\delta}((\mu_{k})_{k\to\omega}).\]
\end{defn}

Recall that if $X$ is a compact, metrizable space, $n\in \N,$ and $\mu\in \Prob(X^{n}),$ then for $1\leq j\leq n$ we use $\mu_{j}$ for the $j^{th}$ marginal of $\mu:$ i.e. $\mu_{j}=(\mathcal{E}_{j})_{*}(\mu)$ where $\mathcal{E}_{j}\colon X^{n}\to X$ is given by $\mathcal{E}_{j}(x)=x(j).$

\begin{defn}\label{D:coarser lw* limit}
Let $G$ be a countable, discrete, sofic group with sofic approximation $\sigma_{k}\colon G\to S_{d_{k}}.$ Let $X$ be a compact, metrizable space with $G\actson X$ by homeomorphisms. Fix a free ultrafilter $\omega\in\beta\N\setminus\N.$  Let $\mu \in \Prob_{G}(\Prob(X)),$ and $(\mu_{k})_{k}\in \prod_{k}\Prob(X^{d_{k}}).$  We say that $\mu_{k}$ locally weak$^{*}$ converges to $\mu$ as $k\to\omega,$ and write $\mu_{k}\to^{lw^{*}}_{k\to\omega}\mu$ if
\[\lim_{k\to\omega}\frac{1}{d_{k}}\sum_{j=1}^{d_{k}}\delta_{\mu_{k,j}}=\mu.\]
We define the local weak$^{*}$ entropy of $\mu$ by
\[h_{(\sigma_{k})_{k\to\omega}}(\mu)=\sup_{(\mu_{k})_{k}}h((\mu_{k})_{k\to\omega}),\]
where the supremum is over all sequences $(\mu_{k})_{k}\in \prod_{k}\Prob(X^{d_{k}}),$ which  are asymptotically supported  on topological microstates for $G\actson X$ with respect to $(\sigma_{k})_{k},\omega,$ and which have $\mu_{k}\to^{lw^{*}}_{k\to\omega}\mu$ with respect to $(\sigma_{k})_{k}.$

\end{defn}

The reader may now notice that if $G,(\sigma_{k})_{k},X$ are as above, and $(\mu_{k})_{k}\in \prod_{k}\Prob(X^{d_{k}})$ is asymptotically supported on topological microstates, then we have two notions of a locally weak$^{*}$ limit of $(\mu_{k})_{k}.$ One is the element $\widetilde{\mu}=\mathcal{E}((\mu_{k})_{k})\in \mathcal{L}_{\omega}(G\actson X),$ and the other is the element $\mu\in\Prob(\Prob(X))$ given by $\mu=\lim_{k\to\omega}\frac{1}{d_{k}}\sum_{j=1}^{d_{k}}\delta_{\mu_{k,j}}.$ These are related by the fact that $\mu=\widetilde{\mu}_{*}(u_{\omega}),$ which follows from Proposition \ref{P:computing integrals in ultraproducts}. In the case that $G\actson X$ is algebraic, we have a natural analogue of $\mathcal{S}_{\omega}(G\actson X)$ as well, namely we could consider the set of $Y\in \Prob_{G}(\Sub(X))$ so that $\mathcal{M}_{*}(Y)$ is a local weak$^{*}$ limit of measures asymptotically supported on topological microstates.

Of course, the (metrizable, separable) space $\Prob(\Prob(X))$ may appear much more friendly to the reader than the (nonmetrizable, nonseparable) space $\Meas(Z_{\omega},\Prob(X)).$ However, there are certain advantages to working with $\mathcal{L}_{\omega}(G\actson X),$ as opposed to a subspace of $\Prob(\Prob(X)).$ First is that $\widetilde{\mu}$ is manifestly a more ``refined" object, in the sense that $\mu$ factors through the pushforward map $\Meas(Z_{\omega},\Prob(X))\to \Prob(\Prob(X)).$ For the case of algebraic actions, there is a clearer reason for why $\Meas(Z_{\omega},\Prob(X))$ is better for our purposes that $\Prob(\Prob(X)).$ Namely, all of the algebraic structure gets obliterated in the passage from a sequence of measures supported on topological microstates to an element of $\Prob(\Prob(X)).$ For example, suppose that $G\actson X$ is algebraic, and that we are given two sequences $(\mu_{k})_{k},(\nu_{k})_{k}\in \prod_{k}\Prob(X^{d_{k}})$ and let $\mu=\lim_{k\to\omega}\frac{1}{d_{k}}\sum_{j=1}^{d_{k}}\delta_{\mu_{k,j}},\nu=\lim_{k\to\omega}\frac{1}{d_{k}}\sum_{j=1}^{d_{k}}\delta_{\nu_{k,j}}.$ Then there is \emph{no} way to define $\mu*\nu$ in a consistent way: namely there may be sequences $(\mu_{k}'),(\nu_{k}')\in \prod_{k}\Prob(X^{d_{k}})$ so that$\mu=\lim_{k\to\omega}\frac{1}{d_{k}}\sum_{j=1}^{d_{k}}\delta_{\mu'_{k,j}},$ $\nu=\frac{1}{d_{k}}\sum_{j=1}^{d_{k}}\delta_{\nu'_{k,j}},$ \emph{but} so that $\lim_{k\to\omega}\frac{1}{d_{k}}\sum_{j=1}^{d_{k}}\delta_{\mu_{k,j}*\nu_{k,j}}\ne \lim_{k\to\omega}\frac{1}{d_{k}}\sum_{j=1}^{d_{k}}\delta_{\mu'_{k,j}*\nu'_{k,j}}.$

For example, fix two closed subgroups $Y_{1},Y_{2}\in \Sub(X)$ with $Y_{1}\subseteq Y_{2}.$ Now consider two sequences of subsets $A_{k},B_{k}\subseteq \{1,\dots,d_{k}\},$ and suppose that:
\[\lim_{k\to\omega}u_{d_{k}}(A_{k})=\lim_{k\to\omega}u_{d_{k}}(B_{k})=\frac{1}{2},\lim_{k\to\omega}u_{d_{k}}(A_{k}\cap B_{k})=\frac{1}{4}.\]
Now suppose that $\mu_{k}^{(1)}\in \Prob(X^{A_{k}}),\mu_{k}^{(2)}\in \Prob(X^{A_{k}^{c}}),$ $\nu_{k}^{(1)}\in X^{B_{k}},\nu_{k}^{(2)}\in X^{B_{k}^{c}}$ satisfy that $\mu_{k}^{(j)}\to^{lw^{*}}m_{Y_{j}}$ and $\nu_{k}^{(j)}\to^{lw^{*}}m_{Y_{j}}$ for $j=1,2.$ Now set
\[\mu_{k}=(\mu_{k}^{(1)})^{\otimes A_{k}}\otimes (\mu_{k}^{(2)})^{\otimes A_{k}^{c}},\]
\[\nu_{k}=(\nu_{k}^{(1)})^{\otimes B_{k}}\otimes (\nu_{k}^{(2)})^{\otimes B_{k}^{c}}.\]
Then, if we consider local weak$^{*}$ convergence inside of $\Prob(\Prob(X)),$ then $\mu_{k},\nu_{k}$ both locally weak$^{*}$ converge to $\frac{1}{2}m_{Y_{1}}+\frac{1}{2}m_{Y_{2}}.$ However, the sequences $\mu_{k}*\mu_{k}$ locally weak$^{*}$ converges to $\frac{1}{2}m_{Y_{1}}+\frac{1}{2}m_{Y_{2}},$ whereas $\mu_{k}*\nu_{k}$ locally weak$^{*}$  converges to $\frac{3}{4}m_{Y_{1}}+\frac{1}{4}m_{Y_{2}}.$ We can even take into account the action $G\actson X,$ by forcing $A_{k},B_{k}$ to be almost $G$-invariant, and by having $G\actson (Y_{1},m_{Y_{1}}),G\actson (Y_{2},m_{Y_{2}})$ be actions which are strongly sofic with respect to any sofic approximation of $G.$ In this manner, we can force $\mu_{k},\nu_{k}$ as above to be asymptotically supported on topological microstates for $G\actson X.$

	What happens in this example is that $\mathcal{E}((\mu_{k})_{k}),\mathcal{E}((\nu_{k})_{k})\in \mathcal{L}_{\omega}(G\actson X)$ remembers the asymptotic structure of the  sets $A_{k},B_{k},$ whereas the operation of taking local weak$^{*}$ limits inside $\Prob(\Prob(X))$ completely ignores what values of $j\in \{1,\dots,d_{k}\}$ have $\mu_{k,j}\approx m_{Y_{1}}$ (or $\mu_{k,j}\approx m_{Y_{2}}$) and just remembers ``how many" of them have this property.  For example,
\[\mathcal{E}((\mu_{k})_{k})=m_{Y_{1}}1_{(A_{k})_{k\to\omega}}+m_{Y_{2}}1_{(A_{k}^{c})_{k\to\omega}},\]
\[\mathcal{E}((\nu_{k})_{k})=m_{Y_{1}}1_{(B_{k})_{k\to\omega}}+m_{Y_{2}}1_{(B_{k}^{c})_{k\to\omega}}.\]
Thus $\mathcal{E}((\mu_{k})_{k}),\mathcal{E}((\nu_{k})_{k})$ do \emph{not} represent the same element of $\Meas(Z_{\omega},\Prob(X)).$ This is one of the major reasons why the space $\mathcal{L}_{\omega}(G\actson X)$ is better suited for our purposes.

Roughly speaking, the failure of convolution to be well-defined in the Ab\'{e}rt-Weiss approach is related to the following classical fact. If $X,Y$ are two real valued random variables, then we cannot say what the distribution of $X+Y$ is only knowing the distribution of $X$ and the distribution of $Y.$ We have to in fact know the \emph{joint distribution} of $(X,Y)$ (e.g. if $X,Y$ were independent we could say the distribution is the convolution). In our setup, the Ab\'{e}rt-Weiss generalized local weak$^{*}$ limit is the \emph{distribution} of our generalized local weak$^{*}$ limit (regarded as function on the Loeb measure space) and, just as with classical sums, we cannot recover the distribution of the convolution of our generalized local weak$^{*}$ limits knowing only the distribution of each local weak$^{*}$ limit individually.

We see from the above discussion that the convolution operation, which is the most important operation one performs on measures on a compact group, is not well-defined in the space of local weak$^{*}$ limits inside $\Prob(\Prob(X)).$ For similar reasons, we cannot make sense $\mu\preceq Y$ inside $\Prob(\Prob(X)),\Prob(\Sub(X)),$ as well as any of the other order operations we need for our proof.  Thus these ``nicer" spaces $\Prob(\Prob(X))$, $\Prob(\Sub(X))$ have a significant disadvantage: they entirely forget the algebraic structure of the group $X.$ This makes it rather difficult to see how one could exploit the algebraic nature of our action to prove useful theorems about these spaces. The answer, we shall see, is to work with the nonmetrizable, but more structured, spaces $\Meas(Z_{\omega},\Prob(X)),\Meas(Z_{\omega},\Sub(X))$,  and then pass from those spaces to the spaces $\Prob(\Prob(X)),\Prob(\Sub(X)).$

We now proceed to  show that the topological entropy of an algebraic action can be realized as the lw$^{*}$-entropy of the Haar measure of a $G$-invariant \emph{random} subgroup of $X$.
%\begin{defn}
%Let $G$ be a countable, discrete, sofic group with sofic approximation $\sigma_{k}\colon G\to S_{d_{k}}.$ Let $X$ be a compact, metrizable space with $G\actson X$ by homeomorphisms. Fix a free ultrafilter $\omega\in\beta\N\setminus\N.$  For a finite $F\subseteq G,\kappa>0,$ we let $\Map(\rho,F,\kappa,\sigma_{k})$ be the set of all $\phi\in X^{d_{k}}$ so that $\rho_{2}(g\phi,\phi\circ \sigma_{k}(g))<\delta$ for all $g\in F.$ We then set, for $\varepsilon>0,$
%\[h_{(\sigma_{k})_{k\to\omega},\topo}(\rho,\varpesilon,F,\kappa)=\lim_{k\to\omega}\frac{1}{d_{k}}\log S_{\varepsilon}(\Map(\rho,F,\kappa,\sigma_{k}),\rho_{2})\]
%\[h_{(\sigma_{k})_{k\to\omega},\topo}(\rho,\varpesilon)=\inf_{\substack{ \textnormal{ finite } F\subseteq G,\\ \delta>0}}h_{(\sigma_{k})_{k\to\omega}}(\rho,\varpesilon,F,\kappa),\]
%\[h_{(\sigma_{k})_{k\to\omega},\topo}(G\actson X)=\sup_{\varepsilon>0}h_{(\sigma_{k})_{k\to\omega},\topo}(\rho,\varpesilon).\]
%
%\end{defn}

\begin{lem}\label{L:basc convolution shiz}
Let $X$ be a compact group, and $\rho$ a translation invariant pseudometric on $X.$ Then for all $\varepsilon,\delta>0$ and for all measures $\mu,\nu\in \Prob(X),$ we have
\[S_{\varepsilon,\delta}(\mu*\nu,\rho)\geq S_{\varepsilon,\delta}(\nu,\rho).\]
\end{lem}

\begin{proof}
Let $A\subseteq X,$ have $\mu*\nu(N_{\varepsilon}(A,\rho))\geq 1-\delta,$ and $|A|=S_{\varepsilon,\delta}(\mu*\nu,\rho).$ We then have that
\[1-\delta\leq \mu*\nu(N_{\varepsilon}(A,\rho))=\int_{X}\nu(x^{-1}N_{\varepsilon}(A,\rho))\,d\mu(x)=\int_{X}\nu(N_{\varepsilon}(x^{-1}A,\rho))\,d\nu(x),\]
where in the last step we use translation-invariance of $\rho.$ So we can find an $x\in X$ so that
\[\nu(N_{\varepsilon}(x^{-1}A,\rho))\geq 1-\delta.\]
Hence $S_{\varepsilon,\delta}(\nu,\rho)\leq |x^{-1}A|=|A|=S_{\varepsilon,\delta}(\mu*\nu,\rho).$

\end{proof}

\begin{lem}\label{L:easy sup argument topo}
Let $G$ be a countable, discrete, sofic group with a sofic approximation $\sigma_{k}\colon G\to \Sym(d_{k}).$ Let $X$ be a compact, metrizable space with $G\actson X$ by homeomorphisms. Then
\[\sup_{(\mu_{k})_{k}}h((\mu_{k})_{k\to\omega})=h_{(\sigma_{k})_{k\to\omega},\topo}^{lw^{*}}(G\actson X),\]
where the supremum is over all sequences of measures $(\mu_{k})_{k}$ which are asymptotically supported as $k\to\omega$ on topological microstates for $G\actson X$ with respect to $(\sigma_{k})_{k}.$

\end{lem}

\begin{proof}
Fix a compatible metric $\rho$ on $X.$
We first show\[\sup_{(\mu_{k})_{k}}h((\mu_{k})_{k\to\omega})\leq h_{(\sigma_{k})_{k\to\omega},\topo}^{lw^{*}}(G\actson X).\]
Fix a  sequence $(\mu_{k})_{k}\in \prod_{k}\Prob(X^{d_{k}})$ which is asymptotically supported on topological microstates as $k\to\omega.$ Let $\varepsilon,\delta>0$ be given, and fix a finite $F\subseteq G,\kappa>0.$ Then $\lim_{k\to\omega}\mu_{k}(\Map(\rho,F,\kappa,\sigma_{k}))=1,$ so $B=\{k\in \N:\mu_{k}(\Map(\rho,F,\kappa,\sigma_{k}))\geq 1-\delta\}\in \omega.$ Since $S_{\varepsilon,\delta}(\mu_{k},\rho_{2})$ is the minimum of $S_{\varepsilon}(B,\rho_{2})$ over all $B\subseteq X^{d_{k}}$ with $\mu_{k}(B)\geq 1-\delta,$ we  have that $S_{\varpesilon,\delta}(\mu_{k},\rho_{2})\leq S_{\varepsilon}(\Map(\rho,F,\kappa,\sigma_{k}),\rho_{2})$ for all $k\in B.$ Since $B\in \omega,$ we have that
\[\lim_{k\to\omega}\frac{1}{d_{k}}\log S_{\varpesilon,\delta}(\mu_{k},\rho_{2})\leq h_{(\sigma_{k})_{k}}(\rho,\varespilon,F,\kappa).\]
Taking the infimum over all $F,\kappa,$ and then taking the supremum over $\varpesilon,\delta$ proves that
\[h((\mu_{k})_{k\to\omega})\leq h_{(\sigma_{k})_{k\to\omega},\topo}^{lw^{*}}(G\actson X).\]

Because supremums commute with supremums, to prove the reverse inequality it suffices to show that for a fixed $\varepsilon>0$ we have
\begin{equation}\label{E:gimme the sup}
\sup_{(\mu_{k})_{k}}\sup_{\delta>0}h_{\varepsilon/2,\delta}((\mu_{k})_{k\to\omega})\geq h_{(\sigma_{k})_{k\to\omega}}(\rho,2\varepsilon),
\end{equation}
where the supremum is over all sequences $(\mu_{k})_{k}\in \prod_{k}\Prob(X^{d_{k}})$ which are asymptotically supported on topological microstates as $k\to\omega.$
So fix $\varespilon>0.$ Write $G=\bigcup_{n=1}^{\infty}F_{n},$ where $(F_{n})_{n}$ is an increasing sequence of finite subsets of $G,$ and fix a decreasing sequence $(\kappa_{n})_{n}$ of positive numbers with $\kappa_{n}\to 0.$ For $n\in \N,$ let
\[\widetilde{B}_{n}=\left\{k\in \N:\frac{1}{d_{k}}\log P_{\varespilon}(\Map(\rho,F_{n},\kappa_{n},\sigma_{k}),\rho_{2})\geq h_{(\sigma_{k})_{k}}(\rho,2\varepsilon,F_{n},\delta)-2^{-n}\right\}.\]
Because $P_{\varespilon}(\Map(\rho,F_{n},\kappa_{n},\sigma_{k}),\rho_{2})\geq S_{2\varepsilon}(\Map(\rho,F_{n},\kappa_{n},\sigma_{k}),\rho_{2})),$ we have that $\widetilde{B}_{n}\in \omega.$ Moreover, since $\Map(\rho_{n},F_{n},\kappa_{n},\sigma_{k})\subseteq \Map(\rho,F_{m},\kappa_{m},\sigma_{k})$ for all $n\leq m,$ we have that $\widetilde{B}_{n}$ are a decreasing family of sets.
Set $B_{n}=\widetilde{B}_{n}\cap \{1,\cdots,n\}^{c},$ and $B_{0}=\N\setminus B_{1}.$
For each $n\in \N,k\in B_{n},$ let $S_{n,k}\subseteq \Map(\rho,F_{n},\kappa_{n},\sigma_{k})$ be $\varepsilon$-separated and have $|S_{n,k}|=P_{\varespilon}(\Map(\rho,F_{n},\kappa_{n},\sigma_{k}),\rho_{2}).$
For $k\in \N,$ let $n(k)\in \N\cup\{0\}$ be such that $k\in B_{n(k)}\setminus B_{n(k)+1}.$ Now set $\mu_{k}=u_{S_{n(k),k}},$ it is straightforward to argue that $\mu_{k}$ is asymptotically supported on topological microstates as $k\to\omega.$

Fix a $\delta\in(0,1)n\in \N, k\in B_{n}.$ Suppose that $\Omega\subseteq X^{d_{k}}$ has $\mu_{k}(N_{\varepsilon/2}(\Omega,\rho_{2}))\geq 1-\delta.$ Since $S_{n(k),k}$ is $\varepsilon$-separated, if $x\in \Omega,$ then $|N_{\varepsilon/2}(x,\rho_{2})\cap S_{n(k),k}|\leq 1.$ Thus:
\[(1-\delta)|S_{n(k),k}|\leq \sum_{x\in \Omega}|N_{\varepsilon/2}(a,\rho_{2})\cap S|\leq |\Omega|.\]
So we have shown that $S_{\varepsilon/2,\delta}(\mu_{k})\geq (1-\delta)|S_{n(k),k}|=(1-\delta)P_{\varepsilon}(\Map(\rho,F_{n(k)},\kappa_{n(k)},\sigma_{k}),\rho_{2})$. Since $k\in B_{n},$ we have $n(k)\geq n.$ So $\Map(\rho,F_{n(k)},\kappa_{n(k)},\sigma_{k})\subseteq \Map(\rho,F_{n},\kappa_{n},\sigma_{k}),$ and thus
\begin{align*}
S_{\varepsilon/2,\delta}(\mu_{k})\geq (1-\delta)|S_{n(k),k}|&\geq (1-\delta)P_{\varepsilon}(\Map(\rho,F_{n(k)},\kappa_{n(k)},\sigma_{k}),\rho_{2})\\
&\geq  (1-\delta)P_{\varepsilon}(\Map(\rho,F_{n},\kappa_{n},\sigma_{k}),\rho_{2})\\
&\geq(1-\delta)\exp(d_{k}[ h_{(\sigma_{k})_{k\to\omega}}(\rho,2\varepsilon,F_{n},\kappa_{n})-2^{-n}]).
\end{align*}
Since $B_{n}\in\omega,$ we have that
\[h_{\varepsilon,\delta}((\mu_{k})_{k\to\omega})=\lim_{k\to\omega}\frac{1}{d_{k}}\log S_{\varepsilon/2,\delta}(\mu_{k},\rho_{2})\geq h_{(\sigma_{k})_{k\to\omega}}(\rho,2\varepsilon,F_{n},\kappa_{n})-2^{-n}.\]
Letting $n\to\infty,$ we see that
\[\sup_{\delta>0}h_{\varepsilon/2,\delta}((\mu_{k})_{k\to\omega})\geq h_{(\sigma_{k})_{k\to\omega}}(\rho,2\varepsilon),\]
which clearly implies (\ref{E:gimme the sup}).
\end{proof}

Let $X$ be a compact group, and $Y\in \Prob(\Sub(X)).$ We use $m_{Y}\in \Prob(\Prob(X))$ for the pushforward of $Y$ under the map $\Sub(X)\to \Prob(X)$ given by $K\mapsto m_{K}.$

\begin{cor}
Let $G$ be  a countable, discrete, sofic group with sofic approximation $\sigma_{k}\colon G\to \Sym(d_{k}),$ and let $G\actson X$ be an algebraic action. Fix a free ultrafilter $\omega\in\beta\N\setminus \N,$ and let $\widetilde{Y}$ be the maximal element of $\mathcal{S}_{\omega}(G\actons X).$ Let $Y=\widetilde{Y}_{*}(u_{\omega})\in \Prob_{G}(\Sub(X)).$ Then:
\[h_{(\sigma_{k})_{k\to\omega}}^{lw^{*}}(m_{Y})=h_{(\sigma_{k})_{k\to\omega},\topo}(G\actson X).\]
In particular if $G_{\omega}'\actson (Z_{\omega},u_{\omega})$ is ergodic, then there is a subgroup $Y\in \Sub(X)$ so that
\[h_{(\sigma_{k})_{k\to\omega}}^{lw^{*}}(G\actson (Y,m_{Y}))=h_{(\sigma_{k})_{k\to\omega},\topo}(G\actson X).\]

\end{cor}

\begin{proof}
Let $\widetilde{Y}$ be the maximal element of $\mathcal{S}_{\omega}(G\actson X),$ and let $Y=\widetilde{Y}_{*}(u_{\omega}).$ Write $\widetilde{Y}=(\widetilde{Y}_{k})_{k\to\omega}.$
By Lemma \ref{L:easy sup argument topo}, it suffices to show that $h_{(\sigma_{k})_{k\to\omega}}^{lw^{*}}(m_{Y})\geq h_{(\sigma_{k})_{k}}((\mu_{k})_{k\to \omega})$ for any sequence of measures $(\mu_{k})_{k}$ which are asymptotically supported on topological microstates.
Let $\mu=\mathcal{E}((\mu_{k})_{k\to\omega})\in \mathcal{L}_{\omega}(G\actson X).$ Since $Y$ is the maximal element of $\mathcal{S}_{\omega}(G\actson X),$ we have that $\mu\preceq \widetilde{Y}$ by Proposition \ref{P:what it means to absorb}, Corollary \ref{C:this the main theorem yo!}. Clearly this implies that $\mu*m_{\widetilde{Y}}=m_{\widetilde{Y}}.$ Thus $\mu_{k}*m_{\widetilde{Y}_{k}}$ also locally weak$^{*}$ converges to $m_{Y}.$  It thus only suffices to show that for every $\varepsilon,\delta>0$ we have
\[S_{\varepsilon,\delta}(\mu_{k}*m_{\widetilde{Y}_{k}})\geq S_{\varepsilon,\delta}(\mu_{k}),\]
which is Lemma \ref{L:basc convolution shiz}.

\end{proof}

We close with a few general comments about the maximal element of $S_{\omega}(X,G).$

\begin{prop}
Let $G$ be a countable, discrete, sofic group and $\sigma_{k}\colon G\to S_{d_{k}}$ a sofic approximation. Fix an algebraic action $G\actson X$ and $\omega$ a free ultrafilter on the natural numbers. Let $Y$ be the maximal element of $\mathcal{S}_{\omega}(G\actson X).$ Then $Y$ is the minimal element of $\Meas(Z_{\omega},\mathcal{F}(X))$ which absorbs all topological microstates.'
\end{prop}

\begin{proof}
Suppose that $F\in \Meas(Z_{\omega},\mathcal{F}(X))$ absorbs all topological microstates. Since $m_{Y}\in \mathcal{L}_{\omega}(G\actson X)$ and $F$ absorbs all topological microstates, we have by Proposition \ref{P:what it means to absorb} (\ref{I:absorbs measures}) that $m_{Y}\preceq F.$ So for almost every $z\in Z_{\omega},$ we have that $Y(z)=\supp(m_{Y(z)})\subseteq F(z),$ and this completes the proof.

\end{proof}

We now show that strong soficity is an invariant under topological factor maps between algebraic actions (even if the factor map is not a homeomorphism), provided that each action is ergodic with respect to the Haar measure. If $A,B$ are sets and $f\colon A\to B,$ then for every $n\in \N$ we let $f^{n}\colon A^{n}\to B^{n}$ be given by $f^{n}(\phi)(j)=f(\phi(j))$ for $1\leq j\leq n,\phi\in A^{n}.$

\begin{cor}
Let $G$ be a countable, discrete, sofic group with sofic approximation $\sigma_{k}\colon G\to S_{d_{k}}$ a sofic approximation, and $\omega$ a free ultrafilter on the numbers. Let $G\actson X,$ $G\actson Y$ be two algebraic actions. Suppose that there is a continuous, $G$-equivariant map $f\colon X\to Y$ with $\overline{\ip{f(X)}}=Y$ (we do not assume that $f$ is a homomorphism).  If there is a sequence $\mu_{k}\in \Prob(X^{d_{k}})$ which is asymptotically supported on topological microstates and so that $\mu_{k}\to^{lw^{*}}m_{X}$ as $k\to\infty,$ then there is a sequence $\nu_{k}\in \Prob(Y^{d_{k}})$ which is asymptotically supported on topological microstates and so that $\nu_{k}\to^{lw^{*}}m_{Y}$ as $k\to\infty.$ In particular, if both $G\actson (X,m_{X}),G\actson (Y,m_{Y})$ are ergodic, and $G\actson (X,m_{X})$ is strongly sofic with respect to $(\sigma_{k})_{k},$ then $G\actson (Y,m_{Y})$ is strongly sofic with respect to $(\sigma_{k})_{k}.$
\end{cor}

\begin{proof}

Fix a free ultrafilter $\omega$ on the natural numbers. It suffices to show that if there is a sequence $\mu_{k}\in \Prob(X^{d_{k}})$ which is asymptotically supported on topological microstates and has  $\mu_{k}\to^{lw^{*}}m_{X}$ as $k\to\omega,$ then there is a sequence $\nu_{k}\in \Prob(Y^{d_{k}})$ which is asymptotically supported on topological microstates and has   $\nu_{k}\to^{lw^{*}}m_{Y}.$ So suppose that there is a sequence $\mu_{k}\in \Prob(X^{d_{k}})$ which is asymptotically supported on topological microstates and has $\mu_{k}\to^{lw^{*}}m_{X}$ as $k\to\omega.$ Let $S$ be the maximal element of $\mathcal{S}_{\omega}(G\actson Y).$ Let $f\colon X\to Y$ be a factor map. Then $f^{d_{k}}_{*}(\mu_{k})\to^{lw^{*}}f_{*}(m_{X}).$ So $f_{*}(m_{X})\in \mathcal{L}_{\omega}(G\actson X),$ and thus by Proposition \ref{P:what it means to absorb} (\ref{I:absorbs measures}) we have that $f_{*}(m_{X})\preceq S.$ Thus $\supp(f_{*}(m_{X}))\subseteq S(z)$ for almost every $z\in Z_{\omega}.$ But since $f$ is continuous, we have that $\supp(f_{*}(m_{X}))=f(\supp(m_{X}))=f(X).$ Since $S(z)$ is a closed subgroup of $Y,$ it follows that for almost every  $z\in Z_{\omega}$ we have that $\overline{\ip{f(X)}}\supseteq S(z).$ Thus $Y\subseteq S(z)$ for almost every $z\in Z_{\omega}$ and thus $S=Y$ almost everywhere. Hence $Y\in \mathcal{S}_{\omega}(G\actson X),$ and this implies that there is a sequence $\nu_{k}\in \Prob(X^{d_{k}})$ with $\nu_{k}\to^{lw^{*}}m_{Y}$ as $k\to\omega.$

\end{proof}

If $\sigma_{k}\colon G\to S_{d_{k}}$ has ergodic commutant, we can say even more and show that \emph{soficity} of the action with respect to $(\sigma_{k})_{k}$ is a topological conjugacy invariant.

\begin{cor}
Let $G$ be a countable, discrete, sofic group with sofic approximation $\sigma_{k}\colon G\to S_{d_{k}}$ a sofic approximation, and assume that $(\sigma_{k})_{k}$ has ergodic commutant.  Let $G\actson X,G\actson Y$ be two algebraic actions. Suppose that $G\actson (Y,m_{Y})$ is ergodic, and that there is a $G$-equivariant, continuous $f\colon X\to Y$ so that $\overline{\ip{f(X)}}=Y$ (we do not assume that $f$ is a homomorphism). If $G\actson (X,m_{X})$ is sofic, then $G\actson (Y,m_{Y})$ is sofic.

\end{cor}

\begin{proof}

Suppose that $G\actson (X,m_{X})$ has ergodic commutant. Since $G\actson (X,m_{X})$ is sofic, it follows by Theorem \ref{T:le limits appendix} in the appendix that  we may find a sequence $\mu_{k}\in \Prob(X^{d_{k}})$ with $\mu_{k}\to^{lw^{*}}m_{X}.$ By the preceding corollary, there is a sequence $\nu_{k}\in \Prob(Y^{d_{k}})$ with $\nu_{k}\to^{lw^{*}}m_{Y}.$ Since $G\actson (Y,m_{Y})$ is ergodic, it follows that $\mu_{k}\to^{le}m_{Y},$ so $G\actson (Y,m_{Y})$ is sofic.

\end{proof}

We also have a product theorem for the maximal element of $\mathcal{S}_{\omega}.$ If $X_{j},j=1,2$ are compact groups, $(d_{k})_{k}$ is a sequence of natural numbers, $\omega$ is a free ultrafilter on the natural numbers, and $Y_{j}\in \Meas(Z_{\omega},X_{j}),j=1,2$ we define $Y_{1}\times Y_{2}\in \Meas(Z_{\omega},X_{1}\times X_{2})$ by $Y_{1}\times Y_{2}(z)=Y_{1}(z)\times Y_{2}(z).$

\begin{cor}
Let $G$ be a countable, discrete, sofic group with sofic approximation $\sigma_{k}\colon G\to S_{d_{k}}$ a sofic approximation, and fix a free ultrafilter on the natural numbers. Suppose $G\actson X_{1},G\actson X_{2}$ are two algebraic actions.  For $j=1,2$ let $Y_{j}$ be the maximal element of $\mathcal{S}_{\omega}(G\actson X_{j}).$ Then the maximal element of $\mathcal{S}_{\omega}(G\actson X_{1}\times X_{2})$ is $Y_{1}\times Y_{2}.$

\end{cor}

\begin{proof}

Let $Y$ be the maximal element of $\mathcal{S}_{\omega}(G\actson X).$ Since $m_{Y_{j}}\in \mathcal{L}_{\omega}(G\actson X_{j})$ for $j=1,2,$ it is easy to see by taking products that $m_{Y_{1}\times Y_{2}}\in \mathcal{L}_{\omega}(G\actson X),$ i.e. $Y_{1}\times Y_{2}\in \mathcal{S}_{\omega}(G\actson X_{1}\times X_{2}).$ Thus $Y_{1}\times Y_{2}\leq Y.$ For the reverse inequality, suppose that $\theta\colon Z_{\omega}\to X_{1}\times X_{2}$ is a topological microstate. Let $\pi_{j}\colon X_{1}\times X_{2}\to X_{j},j=1,2$ be the projection onto the $j^{th}$ factor. Then $\pi_{j}\circ \theta\colon Z_{\omega}\to X_{j},j=1,2$ are topological microstates, and thus $\pi_{j}(\theta(z))\in Y_{j}(z)$ for almost every $z\in Z_{\omega}$ by Corollary \ref{C:this the main theorem yo!}. So we have that $\theta(z)\in Y_{1}(z)\times Y_{2}(z)$ for almost every $z\in Z_{\omega}.$ Thus $Y_{1}\times Y_{2}\leq Y.$

\end{proof}

\section{Proof of the main reduction}\label{S:main proof}

In this section, we prove Theorem \ref{T:main theorem restated}. We use Hilbert space techniques, for reasons outlined in Section \ref{S:background setup}. Given a Hilbert space $\mathcal{H},$ we let $\Ball(B(\mathcal{H}))=\{T\in B(\mathcal{H}):\|T\|\leq 1\}.$ We also let $\Proj(\mathcal{H})$ be the set of orthogonal projections on $\mathcal{H}.$  If $\mathcal{H}$ is separable, then $\Ball(B(\mathcal{H}))$ is a Polish space with a metric given by
\[\rho(T,S)=\sum_{n=1}^{\infty}2^{-n}\|T(\xi_{n})-S(\xi_{n})\|,\]
for any sequence $(\xi_{n})_{n}$ in $\mathcal{H}$ which has dense linear span. In this manner, we may regard $\Ball(B(\mathcal{H}))$ as a complete metric space, which will be important in order to apply Proposition \ref{P:just the basics Loeb}. We may also regard $\Proj(\mathcal{H})$ as a complete metric space, using the same metric. We caution the reader that $B(\mathcal{H})$ is \emph{not} metrizable in the strong operator topology. This will cause no issue for us, because we will primarily work with $\Ball(B(\mathcal{H})).$

We will actually have little use for the metric $\rho$ above. What will be important for us is that it turns $\Ball(\mathcal{H})$ into a complete metric space, and that if $(X,d)$ is a metric space then a map $f\colon X\to B(\mathcal{H})$ is uniformly continuous if and only if for every $\xi\in \mathcal{H},$ the map $x\mapsto f(x)\xi$ is uniformly continuous (for the metric induced by the norm on $\mathcal{H}$). These facts are easy exercises left to the reader.  Throughout this section, we give $\Ball(B(\mathcal{H}))$ the strong operator topology unless otherwise mentioned.

\begin{defn}
Let $X$ be a compact group, we define the \emph{left regular representation} $\lambda\colon X\to \mathcal{U}(L^{2}(X))$ by $(\lambda(x)\xi)(y)=\xi(x^{-1}y)$ for $x,y\in X,\xi\in L^{2}(X).$ We similarly define $\lambda\colon \Prob(X)\to B(L^{2}(X))$ by
\[\ip{\lambda(\mu)\xi,\eta}=\int_{X}\ip{\lambda(x)\xi,\eta}\,d\mu(x),\mbox{ for all $\mu\in \Prob(X),\xi,\eta\in L^{2}(X)$.}\]
\end{defn}
Observe that $\|\lambda(\mu)\|\leq 1$ for all $\mu \in \Prob(X).$ So we may regard $\lambda\colon \Prob(X)\to \Ball(B(L^{2}(X))).$ As is well known, the map $\lambda\colon \Prob(X)\to \Ball(B(L^{2}(X)))$ is continuous if we give $\Prob(X)$ the weak$^{*}$-topology, and $\Ball(L^{2}(X)))$ the strong operator topology. The map $\lambda$ is also injective. Since $\Prob(X)$ is compact, there is a uniformly continuous map $g\colon \lambda(\Prob(X))\to \Prob(X)$ so that $g(\lambda(\mu))=\mu$ for every $\mu\in \Prob(X).$ It follows by Proposition \ref{P:just the basics Loeb} that we have an injective map
\[\lambda_{*}\colon \Meas(Z_{\omega},\Prob(X))\to \Meas(Z_{\omega},\Ball(B(L^{2}(X))))\]
with closed image, and that $\lambda_{*}$ is a homeomorphism onto its image. We collect this and other properties in the following proposition.

\begin{prop}\label{P:fourier fo sho}
Let $G$ be a countable, discrete, sofic group with sofic approximation $\sigma_{k}\colon G\to S_{d_{k}}.$ Let $G\actson X$ be an algebraic action, and $\omega$ a free ultrafilter on the natural numbers.
\begin{enumerate}[(i)]
\item \label{I:homeo again} The map $\lambda_{*}\colon \Meas(Z_{\omega},\Prob(X))\to \Meas(Z_{\omega},\Ball(B(L^{2}(X))))$ is injective, has closed image, and is a homeomorphism onto its image.
\item \label{I:get these projections yo} $\lambda_{*}(\mathcal{M}_{*}(\Meas(Z_{\omega},\Sub(X))))=\lambda_{*}( \Meas(Z_{\omega},\Prob(X)))\cap \Meas(Z_{\omega},\Proj(L^{2}(X))).$
\item \label{I:order-reversing stuff} The map $\lambda_{*}\circ \mathcal{M}_{*}\colon \Meas(Z_{\omega},\Sub(X))\to \Meas(Z_{\omega},\Proj(L^{2}(X)))$ is order-reversing.
\item \label{I:more sot closed alala} The space $\lambda_{*}(\mathcal{L}_{\omega}(G\actson X))$ is closed under pointwise products, pointwise convex combinations, pointwise adjoints, and is a topologically closed subset of $\Meas(Z_{\omega},\Ball(B(L^{2}(X)))$ if we give $\Ball(B(L^{2}(X)))$ the strong operator topology.
\item \label{I:more project alkdsjala} We have that $\lambda_{*}(\mathcal{M}_{*}(\mathcal{S}_{\omega}(G\actson X))=\lambda_{*}(\mathcal{L}_{\omega}(G\actson X))\cap \Meas(Z_{\omega},\Proj(L^{2}(X))).$
\end{enumerate}

\end{prop}

\begin{proof}

(\ref{I:homeo again}): Automatic from Proposition \ref{P:just the basics Loeb} and the fact that $\Prob(X)$ is compact and $\lambda$ is injective.

(\ref{I:get these projections yo}):
Throughout this part we use that $\lambda,\lambda^{*}$ preserve products and the $*$-operation.
Since $m_{Y}^{*}=m_{Y}$ and $m_{Y}*m_{Y}=m_{Y},$ for all $Y\in \Meas(Z_{\omega},\Sub(X))$ `it is clear that
\[\lambda_{*}(\mathcal{M}_{*}(\Meas(Z_{\omega},\Sub(X))))\subseteq\lambda_{*}( \Meas(Z_{\omega},\Prob(X)))\cap \Meas(Z_{\omega},\Proj(L^{2}(X))).\]
Conversely, suppose that $P\in \lambda_{*}(\Meas(Z_{\omega},\Prob(X))\cap \Meas(Z_{\omega},\Proj(L^{2}(X))),$ and write $P=\lambda_{*}(\mu)$ for some $\mu\in \Meas(Z_{\omega},\Prob(X)).$  By injectivity of $\lambda,$ and the fact that
$P$ is a projection, we have that $\mu(z)*\mu(z)=\mu(z),$ and $\mu(z)=\mu^{*}(z)$ for almost all $z\in Z_{\omega}.$  This implies, by \cite[Theorem 1]{Wendel}, that $\mu(z)\in \mathcal{M}(\Sub(X))$ for almost every $z\in Z_{\omega}.$

(\ref{I:order-reversing stuff}): This reduces to the claim that $\lambda(m_{Y'})\leq \lambda(m_{Y})$ if $Y,Y'\in \Sub(X)$ and $Y'\supseteq Y.$ Since $\lambda(m_{Y})$ can be identified with the projection onto $Y$ invariant functions in $L^{2}(X)$ for $Y\in \Sub(X),$ this claim is trivial.

(\ref{I:more sot closed alala}): Clear from part (\ref{I:homeo again}), Theorem \ref{T:topology is helpful}, and the comments after Proposition \ref{L:topological poset}.

(\ref{I:more project alkdsjala}): Obvious from part (\ref{I:get these projections yo}).

\end{proof}

The above characterization of elements of $\mathcal{M}_{*}(\mathcal{S}_{\omega}(G\actson X))$  in terms of measurably varying families of projections on a Hilbert space will be useful, because it turns out there are concrete ways to recover the meet of two projections in a Hilbert space, as well as recover the projection onto the fixed points of an operator.

\begin{lem}\label{L:A wild Hilbert space appears}
Let $\mathcal{H}$ be a Hilbert space.
\begin{enumerate}[(a)]
\item \label{I:to the nth degree} Let $T\in B(\mathcal{H})$ with $\|T\|\leq 1.$ Let $P$ be the orthogonal projection onto the fixed points of $T.$ Then
\[P=SOT-\lim_{n\to\infty}\left[\left(\frac{1}{2}T+\frac{1}{2}\id\right)^{*}\left(\frac{1}{2}T+\frac{1}{2}\id\right)\right]^{n}.\]
\item \label{I:computing interesections}
 Let $P,Q\in \Proj(\mathcal{H}).$ Then
\[P\wedge Q=SOT-\lim_{n\to\infty}(PQP)^{n}.\]
\end{enumerate}
\end{lem}

\begin{proof}
(\ref{I:to the nth degree}):
 Let $Q$ be the projection onto the fixed points of $\left(\frac{1}{2}T+\frac{1}{2}\right)^{*}\left(\frac{1}{2}T+\frac{1}{2}\right).$ Since $\left\|\frac{1}{2}T+\frac{1}{2}\id\right\|\leq 1,$by the Spectral Theorem we have that
\[Q=1_{\{1\}}\left(\left(\frac{1}{2}T+\frac{1}{2}\id\right)^{*}\left(\frac{1}{2}T+\frac{1}{2}\id\right)\right)=SOT-\lim_{n\to\infty}\left[\left(\frac{1}{2}T+\frac{1}{2}\id\right)^{*}\left(\frac{1}{2}T+\frac{1}{2}\id\right)\right]^{n}.\]

So it suffices to show that $P=Q.$ Clearly, we have that $P\leq Q.$ So it enough to show that if $\xi\in \mathcal{H}$ and $Q(\xi)=\xi,$ then $P(\xi)=\xi.$ So suppose that  $\xi\in \mathcal{H}$ and $Q(\xi)=\xi.$ By the Parallelogram Law, we have
\begin{align*}
\frac{1}{4}\left\|(T-1)\xi\right\|^{2}&=\frac{1}{2}\|T\xi\|^{2}+\frac{1}{2}\|\xi\|^{2}-\left\|\left(\frac{1}{2}T+\frac{1}{2}\right)\xi\right\|^{2}\\
&=\frac{1}{2}\|T\xi\|^{2}+\frac{1}{2}\|\xi\|^{2}-\left\ip{\left(\frac{1}{2}T+\frac{1}{2}\right)^{*}\left(\frac{1}{2}T+\frac{1}{2}\right)\xi,\xi\right}\\
&=\frac{1}{2}\|T\xi\|^{2}+\frac{1}{2}\|\xi\|^{2}-\|\xi\|^{2},
\end{align*}
the last line following because $Q(\xi)=\xi.$ Since $\|T\|\leq 1,$ the above shows that $\|(T-\id)\xi\|\leq 0,$ and thus $T(\xi)=\xi, $ i.e. $P(\xi)=\xi.$ So $P=Q.$

(\ref{I:computing interesections}): Let $E$ be the projection onto the fixed points of $PQP,$ as in part (\ref{I:to the nth degree}), it suffices to show that $E=P\wedge Q.$ Clearly, $P\wedge Q\leq E,$ so suppose that $\xi\in\mathcal{H}$ and that $E(\xi)=\xi.$ Then $\xi=(PQP)(\xi),$ so
\[\|\xi\|=\|(PQP)\xi\|\leq \|P\xi\|,\]
and since $P$ is an orthogonal projection, this implies that $P(\xi)=\xi.$ Since $\xi=(PQP)(\xi),$ this implies that $\xi=(PQ)(\xi).$ As above, we have that $\|\xi\|\leq \|Q\xi\|,$ and this implies that $Q\xi=\xi.$ Hence $(P\wedge Q)(\xi)=\xi.$
\end{proof}

We can now prove parts (\ref{I: supports and shiz}) and (\ref{I:join of subgroups}) of Theorem \ref{T:main theorem restated}.

\begin{proof}[Proof of part (\ref{I: supports and shiz}) and (\ref{I:join of subgroups}) of Theorem \ref{T:main theorem restated}]

(\ref{I: supports and shiz}):
We start with the following preliminary observation.

\emph{Observation: For every $\mu \in \Prob(X),$ the projection onto the fixed points of $\lambda(\mu)$ is $\lambda(m_{\overline{\ip{\supp(\mu)}}}).$}

To prove the observation, set $P=\lambda(m_{\overline{\ip{\supp(\mu)}}}).$ Since $\lambda(\mu)P=\lambda(\mu*m_{\overline{\ip{\supp(\mu)}}})=\lambda(m_{\overline{\ip{\supp(\mu)}}}),$ it is clear that $P$  dominates the projection onto the fixed points of $\lambda(\mu).$ Conversely, if $\xi\in L^{2}(X)$ and $\lambda(\mu)\xi=\xi,$ then
\[\|\xi\|^{2}=\rea(\ip{\lambda(\mu)\xi,\xi})=\int_{X}\rea(\ip{\lambda(x)\xi,\xi})\,d\mu(x).\]
By the Cauchy-Schwartz inequality, we know that $\rea(\ip{\lambda(x)\xi,\xi})\leq\|\xi\|^{2.}$ Thus the above equation shows  that $\rea(\ip{\lambda(x)\xi,\xi})=\|\xi\|^{2}$ almost everywhere. By expanding $\|\lambda(x)\xi-\xi\|_{2}^{2},$ we see that $\lambda(x)\xi=\xi$ for $\mu$-almost every $x\in X.$ Since $\{x\in X:\lambda(x)\xi=\xi\}$ is closed in $X,$ we see that $\lambda(x)\xi=\xi$ for every $x\in \supp(\mu).$ But $\{x\in X:\lambda(x)\xi=\xi\}$ is also a closed subgroup, and thus $\lambda(y)\xi=\xi$ for every $y\in \overline{\ip{\supp(\mu)}}.$ Hence $P\xi=\xi.$

Since $\lambda\colon \Prob(X)\to \Ball(B(L^{2}(X)))$ is a homeomorphism onto its image, Lemma \ref{L:A wild Hilbert space appears} (\ref{I:to the nth degree}) and the observation imply that
\[m_{\overline{\ip{\supp(\mu)}}}=\lim_{n\to\infty}\left[\left(\frac{1}{2}\mu+\frac{1}{2}\delta_{e}\right)^{*}*\left(\frac{1}{2}\mu+\frac{1}{2}\delta_{e}\right)\right]^{n},\]
for all $\mu\in \Prob(X).$

Now fix a $\mu\in \mathcal{P}_{\omega}(X,G).$ Define $Y$ as in Theorem \ref{T:main theorem restated} (\ref{I: supports and shiz}). Then for all $z\in Z_{\omega},$ we have
\[m_{Y(z)}=\lim_{n\to\infty}\left[\left(\frac{1}{2}\mu(z)+\frac{1}{2}\delta_{e}\right)^{*}*\left(\frac{1}{2}\mu(z)+\frac{1}{2}\delta_{e}\right)\right]^{*n}.\]
Note the above limiting formula also implies that $Y(z)$ is a measurable function on the Loeb measure space.
It follows by Egoroff's theorem that
\[m_{Y}=\lim_{n\to\infty}\left[\left(\frac{1}{2}\mu+\frac{1}{2}\delta_{e}\right)^{*}*\left(\frac{1}{2}\mu+\frac{1}{2}\delta_{e}\right)\right]^{*n}\]
 in the measure topology. Hence by Theorem \ref{T:topology is helpful} we have that $m_{Y}\in \mathcal{P}_{\omega}(X,G),$ i.e. $Y\in \mathcal{S}_{\omega}(X,G).$

(\ref{I:join of subgroups}):
It is easy to see that for all $K_{1},K_{2}\in \Sub(X),$ we have that $\lambda(m_{K_{1}})\wedge \lambda(m_{K_{2}})=\lambda(m_{\overline{\ip{K_{1},K_{2}}}}).$
By similar arguments as in  part (\ref{I: supports and shiz}), we know that
\[m_{Y_{1}\vee Y_{2}}=\lim_{n\to\infty}\left(m_{Y_{1}}*m_{Y_{2}}*m_{Y_{1}}\right)^{*n}\]
in the measure topology.
Since $m_{Y_{1}},m_{Y_{2}}\in \mathcal{L}_{\omega}(G\actson X)$ the above shows that $m_{Y_{1}\vee Y_{2}}\in \mathcal{L}_{\omega}(G\actson X).$ Thus by definition we have that $Y_{1}\vee Y_{2}\in \mathcal{S}_{\omega}(G\actson X).$

\end{proof}

The proof of part (\ref{I:lattices and stuff}) of Theorem \ref{T:main theorem restated} requires a bit more effort.

\begin{defn}
Let $(d_{k})_{k}$ be a sequence of natural numbers and $\omega$ a free ultrafilter on $\N.$ Fix a separable Hilbert space $\mathcal{H}.$ Let $L^{\infty}(Z_{\omega},B(\mathcal{H}))$ be the set of all maps $T\colon Z_{\omega}\to B(\mathcal{H})$ so that:
\begin{itemize}
\item There is an $R>0$ so that $\{z\in Z_{\omega}:\|T(z)\|\leq R\}$ is a null set,
\item if $R$ is as above, and $Z_{0}=\{z\in Z_{\omega}:\|T(z)\|\leq R\},$ then $T\big|_{Z_{0}}\in \Meas(Z_{0},R\Ball(B(\mathcal{H}))).$
\end{itemize}
\end{defn}
One small remark on the above definition is necessary. As discussed at the beginning of this section, $R\Ball(B(\mathcal{H}))$ is a Polish space for all $R>0.$ Thus it makes sense to ask that $T\big|_{Z_{0}}\in \Meas(Z_{0},R\Ball(\mathcal{H})),$ it simply means that for every Borel $E\subseteq R\Ball(B(\mathcal{H})),$ the set $T^{-1}(E)\cap Z_{0}$ is measurable. Remember that one has to be slightly careful, since $B(\mathcal{H})$ is not metrizable nor separable in the strong operator topology, and is thus certainly not a Polish space.

It is clear that $L^{\infty}(Z_{\omega},B(\mathcal{H}))$ is a $*$-algebra under the operations
\[(T+S)(z)=T(z)+S(z),\]
\[T^{*}(z)=(T(z))^{*},\]
\[(TS)(z)=T(z)S(z),\]
defined for $T,S$ in $Z.$ Furthermore, the norm $\|T\|_{\infty}=\textnormal{esssup}_{z}\|T(z)\|$ is a Banach algebra norm on $L^{\infty}(Z_{\omega},B(\mathcal{H})).$ We identify $L^{\infty}(Z_{\omega})\subseteq L^{\infty}(Z_{\omega},B(\mathcal{H}))$ by the identification $f\mapsto (z\mapsto f(z)\id).$

If $\xi\colon Z_{\omega}\to \mathcal{H}$ is measurable, then $\|\xi(z)\|=\sup_{\eta\in D}|\ip{\xi(z),\eta}|,$ where $D$ is a countable dense subset of the unit ball of $\mathcal{H}$ (this exists by separability of $\mathcal{H}$), and so $z\mapsto \|\xi(z)\|$ is measurable. We let $L^{2}(Z_{\omega},\mathcal{H})$ be the set of all measurable functions $\xi\colon Z\to \mathcal{H}$ so that $\int \|\xi(z)\|^{2}\,du_{\omega}(z)<\infty.$ We identify two elements of $L^{2}(Z_{\omega},\mathcal{H})$ if they are equal almost everywhere. If $\xi,\eta\in L^{2}(Z_{\omega},u_{\omega})$ then
\[\ip{\xi(z),\eta(z)}=\sum_{j\in J}\ip{\xi(z),e_{j}}\overline{\ip{\eta(z),e_{j}}}\]
where $(e_{j})_{j\in J}$ is any orthonormal basis of $\mathcal{H}$. We must have that $J$ is countable, since $\mathcal{H}$ is separable. Thus $z\mapsto \ip{\xi(z),\eta(z)}$ is measurable. Hence we have an inner product on $L^{2}(Z_{\omega},\mathcal{H})$ given by
\[\ip{\xi,\eta}=\int_{Z_{\omega}}\ip{\xi(z),\eta(z)}\,du_{\omega}(z).\]
It is readily verified that $L^{2}(Z_{\omega},\mathcal{H})$ is a Hilbert space under this inner product.

Define $\iota\colon L^{\infty}(Z_{\omega},B(\mathcal{H}))\to B(L^{2}(Z_{\omega},\mathcal{H}))$ by
$(\iota(T)\xi)(z)=T(z)\xi(z).$
It is easy to see that $\|\iota(T)\|\leq \|T\|_{\infty}.$ Moreover, by \cite[Theorem 5.27 (b)]{ConwayOT} we have that
\begin{equation}\label{E:isometry measurable asaas}
\|\iota(T)\|=\|T\|_{\infty}.
\end{equation}
For $f\in L^{2}(Z_{\omega},u_{\omega})$ and $\xi\in \mathcal{H},$ we let $f\otimes \xi\in L^{2}(Z_{\omega},\mathcal{H})$ be given by $(f\otimes \xi)(z)=f(z)\xi.$
\begin{lem}\label{L:density and stuff}
Let $(d_{k})_{k}$ be a sequence of natural numbers and $\omega$ a free ultrafilter on the natural number, and $\mathcal{H}$ be  a separable Hilbert space. Then $\iota(L^{\infty}(Z_{\omega}))$ commutes with $\iota(L^{\infty}(Z_{\omega},B(\mathcal{H}))).$ Further,
\[\{\iota(f)(1\otimes \xi):\xi\in\mathcal{H},f\in L^{\infty}(Z_{\omega})\}\]
has dense linear span in $L^{2}(Z_{\omega},\mathcal{H}).$

\end{lem}

\begin{proof}
The fact that $\iota(L^{\infty}(Z_{\omega}))$ commutes with $\iota(L^{\infty}(Z_{\omega},B(\mathcal{H})))$ is trivial. Observe that for every $f\in L^{\infty}(Z_{\omega}),$ we have that $\iota(f)(1\otimes \xi)=f\otimes \xi.$ Hence
\[\overline{\Span\{\iota(f)(1\otimes \xi):\xi\in\mathcal{H},f\in L^{\infty}(Z_{\omega})\}}\supseteq \overline{\Span\{f\otimes \xi:f\in L^{2}(Z_{\omega}),\xi\in \mathcal{H}\}}.\]
Since $\mathcal{H}$ is separable, this makes it clear that
\[\overline{\Span\{\iota(f)(1\otimes \xi):\xi\in\mathcal{H},f\in L^{\infty}(Z_{\omega})\}}=L^{2}(Z_{\omega},\mathcal{H}).\]

\end{proof}

\begin{prop}\label{P:basic facts about SOT}Let $(d_{k})_{k}$ be a sequence of natural numbers and $\omega$ a free ultrafilter on the natural number, and $\mathcal{H}$ be  a separable Hilbert space. Then $\iota(L^{\infty}(Z_{\omega}))$ commutes with $\iota(L^{\infty}(Z_{\omega},B(\mathcal{H}))).$
Consider the inclusion map $\iota\colon L^{\infty}(Z_{\omega},B(\mathcal{H}))\to B(L^{2}(Z_{\omega},\mathcal{H})).$
\begin{enumerate}[(i)]
\item \label{I:What's a vNa?} We have that $\iota(L^{\infty}(Z_{\omega},B(\mathcal{H})))$ is a strong  operator topology closed subalgebra of $B(L^{2}(Z_{\omega},\mathcal{H}))$ which is closed under adjoints.
\item \label{I:where for are thou homeo?} For every $R>0,$ the map $\iota\big|_{\Meas(Z_{\omega},R\Ball(B(\mathcal{H})))}$ is continuous, has closed image, and is a homeomorphisms onto its image if we give $B(L^{2}(Z_{\omega},\mathcal{H}))$ the strong operator topology.
\end{enumerate}
\end{prop}

\begin{proof}

(\ref{I:What's a vNa?}): This is true by \cite[Theorem 52.8 (a)]{ConwayOT}

(\ref{I:where for are thou homeo?}): Without loss of generality, $R=1.$
By (\ref{E:isometry measurable asaas}),
\[\iota(\Meas(Z_{\omega},\Ball(B(\mathcal{H}))))=\iota(L^{\infty}(Z_{\omega},\Ball(B(\mathcal{H}))))\cap \Ball(B(L^{2}(Z_{\omega},B(\mathcal{H})))\]
and this proves that $\iota(\Meas(Z_{\omega},\Ball(B(\mathcal{H}))))$ is closed by part (\ref{I:What's a vNa?}).

Let us first show that $\iota$ is continuous. Since $\Meas(Z_{\omega},\Ball(B(\mathcal{H}))$ is metrizable, it suffices to show that if $T_{n}$ is a sequence in $\Meas(Z_{\omega},\Ball(B(\mathcal{H})))$ with $T_{n}\to T,$  then $\iota(T_{n})\to \iota(T)$ in the strong-operator topology. So fix a sequence $T_{n}$ in $\Meas(Z_{\omega},\Ball(B(\mathcal{H})))$ with $T_{n}\to T.$ Let
\[M=\left\{S\in B(L^{2}(Z_{\omega},\mathcal{H})):S\iota(T)=\iota(T)S \mbox{ for all $T\in \Meas(Z_{\omega},B(\mathcal{H}))$}\right\},\]
and
 $\mathcal{K}=\{\xi\in \mathcal{H}:\|\iota(T_{n})\xi-\iota(T)\xi\|_{2}\to 0\}.$

It suffices to show that $\mathcal{K}=\mathcal{H}.$ It is clear that $\mathcal{K}$ is a linear subspace of $\mathcal{H},$ and since $\|\iota(T_{n})\|\leq 1,$ it follows that $\mathcal{K}$ is norm closed. Lastly, it is straightforward to show that $\mathcal{K}$ is $M$-invariant. Hence by Lemma \ref{L:density and stuff}, it suffices to show that if $\xi\in \mathcal{H},$ then $1\otimes \xi\in \mathcal{K}.$ To prove this, let $\varepsilon>0$ be given. Let $\mathcal{O}_{\varespilon}=\left\{A\in B(\mathcal{H}): \|A\xi\|_{2}<\varepsilon\right\},$ and set $V_{\varespilon}=\left\{B\in \Meas(Z_{\omega},\Ball(B(\mathcal{H}))):u_{\omega}\left(\{z:B(z)-T(z)\in \mathcal{O}_{\varepsilon}\}\right)>1-\varepsilon\right\}.$ Since $T_{n}\to T,$ for all large $n$ we have that $T_{n}\in V_{\varepsilon}.$ Thus, for all large $n,$ we have that
\[\|\iota(T_{n})(1\otimes \xi)-\iota(T)(1\otimes \xi)\|_{2}^{2}\leq \varepsilon^{2}(1-\varepsilon)+4\varepsilon^{2}\|\xi\|^{2}.\]
Thus
\[\limsup_{n\to\infty}\|\iota(T_{n})(1\otimes \xi)-\iota(T)(1\otimes \xi)\|_{2}\leq \varpesilon+2\varpesilon\|\xi\|,\]
and letting $\varepsilon\to 0$ proves that $\|\iota(T_{n})(1\otimes \xi)-\iota(T)(1\otimes \xi)\|_{2}\to 0.$ This proves that $\iota$ is continuous.

To prove that $\iota$ is a homeomorphism onto its image, fix $T\in \Meas(Z_{\omega},\Ball(B(\mathcal{H}))),$ let $U$ be a neighborhood of $T.$ It suffices to show that there is a strong operator topology neighborhood $V$ of $\iota(T)$ so that $\iota^{-1}(V)\subseteq U.$ We may choose an $\varepsilon>0$ and $\xi_{1},\dots,\xi_{n}\in \mathcal{H}$ so that
\[U\supseteq \bigcap_{j=1}^{n}\left\{S\in \Meas(Z_{\omega},\Ball(B(\mathcal{H})): u_{\omega}\left(\left\{z:\|S(z)\xi_{j}-T(z)\xi_{j}\|_{2}<\varepsilon\right\}\right)>1-\varepsilon\right\}.\]
Let
\[V=\bigcap_{j=1}^{n}\{B\in B(L^{2}(Z_{\omega},\mathcal{H})): \|B(1\otimes \xi_{j})-\iota(T)(1\otimes \xi_{j})\|_{2}<\varepsilon^{2}\}.\]
It is a straightforward consequence of Chebyshev's inequality that if $S\in\Meas(Z_{\omega},\Ball(B(\mathcal{H})))$ and $\iota(S)\in V,$ then $S\in U.$

\end{proof}

We can now prove part (\ref{I:lattices and stuff}) of Theorem \ref{T:main theorem restated}.

\begin{proof}[Proof of Theorem \ref{T:main theorem restated} (\ref{I:lattices and stuff})]

Set  $F=\iota(\lambda_{*}(\mathcal{M}(\mathcal{S}_{\omega}(X,G))))$. By Proposition \ref{P:fourier fo sho} (\ref{I:get these projections yo}), we know that $F\subseteq \Proj(L^{2}(Z_{\omega}\times X)).$ By Proposition \ref{P:fourier fo sho} (\ref{I:order-reversing stuff}), it suffices to show that $F$ is complete meet-lattice. By  Proposition \ref{P:fourier fo sho} (\ref{I:order-reversing stuff}) and Theorem \ref{T:main theorem restated} (\ref{I:join of subgroups}), it follows that $F$ is a meet-lattice and that the meet in $F$ agrees with the meet in $\Proj(L^{2}(Z_{\omega}\times X))$. To see that it is complete, let $(P_{\alpha})_{\alpha\in I}$ be a family of elements of $F.$ For every finite $A\subseteq I,$ set $P_{A}=\bigwedge_{\alpha\in A}P_{\alpha}.$ Then $P_{A}\in F$ since $F$ is a meet-lattice. We make $(P_{A})_{A}$ a net by ordering the finite subsets of $I$ by inclusion.  It is then easy to see that
\[\bigwedge_{\alpha\in I}P_{\alpha}=SOT-\lim_{A}P_{A},\]
where the meet on the left hand side of the above equation is taken inside $\Proj(L^{2}(Z_{\omega}\times X)).$
By Proposition \ref{P:fourier fo sho} and Proposition \ref{P:basic facts about SOT} (\ref{I:where for are thou homeo?}) it follows that $F$ is strong operator topology closed. We thus have that $\bigwedge_{\alpha\in I}P_{\alpha}\in F,$ and so we have shown that $F$ is a complete meet-lattice.

\end{proof}

\section{Computations in the case of non-strongly sofic algebraic actions}\label{S:Not SS actions}

In this case, we mention a few examples of cases where the maximal element of $\mathcal{S}_{\omega}(G\actson X)$ is not $X,$ as well as how to compute what the maximal element is in some of these cases. Some notions from unitary representation theory will be helpful.

\begin{defn}
Let $G$ be a countable, discrete group and $\pi\colon G\to \mathcal{U}(\mathcal{H}),$ $\rho\colon G\to \mathcal{U}(\mathcal{K})$ two Hilbert space representations. We say that \emph{$\pi$ is weakly contained in $\rho$}, and write $\pi\preceq \rho,$ if
\[\left\|\sum_{g\in G}a_{g}\pi(g)\right\|\leq \left\|\sum_{g\in G}a_{g}\rho(g)\right\|\mbox{ for all $a=\sum_{g\in G}a_{g}g\in \C(G)$,}\]

We say that $\pi$ has \emph{spectral gap} if there is a finite $F\subseteq G$ and a $C>0$ so that
\[\|\xi\|\leq C\sum_{g\in F}\|\pi(g)\xi-\xi\|\mbox{ for all $\xi\in \mathcal{H}.$}\]
It can be shown that this is the same as saying that the trivial representation $1\colon G\to S^{1}$ given by $1(g)=1$ is not weakly contained in $\pi$ (see \cite[Theorem 4.4]{BHV} and \cite[Proposition F.1.7]{BHV}).

Given two unitary representation $\pi_{j}\colon G\to \mathcal{U}(\mathcal{H}_{j}),j=1,2$ we define $\pi_{1}\otimes \pi_{2}\colon G\to \mathcal{U}(\mathcal{H}_{1}\otimes \mathcal{H}_{2})$ by $(\pi_{1}\otimes \pi_{2})(g)(\xi_{1}\otimes \xi_{2})=\pi_{1}(g)\xi_{1}\otimes \pi_{2}(g)\xi_{2}$ for $\xi_{j}\in \mathcal{H}_{j},j=1,2.$

\end{defn}

A particular example of interest is the left regular representation $\lambda\colon G\to \mathcal{U}(\ell^{2}(G))$ by $(\lambda(g)\xi)(x)=\xi(g^{-1}x)$ for all $g,x\in G,\xi\in \ell^{2}(G).$ In this case $\lambda$ has spectral gap if and only if $G$ is nonamenable.

\begin{defn}
Let $(X,\mu)$ be a Lebesgue probability space, $G$ a countable discrete group and $G\actson (X,\mu)$ a probability measure-preserving action. We let $L^{2}_{0}(X,\mu)=\left\{f\in L^{2}(X):\int f\,d\mu=0\right\},$ and define the \emph{Koopman representation} $\rho\colon G\to \mathcal{U}(L^{2}_{0}(X))$ by $(\rho(g)f)(x)=f(g^{-1}x).$ We say that $G\actson (X,\mu)$ has \emph{spectral gap} if its Koopman representation has spectral gap. We say that $G\actson (X,\mu)$ has \emph{stable spectral gap} if $G\actson (X\times X,\mu\otimes \mu)$ has spectral gap. It is trivial that an action with stable spectral gap is weakly mixing.
\end{defn}

Each of our examples in this section will be for actions of free groups. In this case, we use the recent results of Bordenave-Collins (see \cite{CollinsBordenave}) which imply the following result.

\begin{prop}\label{P:Corollary of BC}
Let $\F_{r}=\ip{a_{1},\cdots,a_{r}}$ be a free group of rank $r>1.$ Then there is a sequence $\Omega_{n}\subseteq \Sym(n)^{r}$ with the following properties:
\begin{enumerate}[(a)]
\item $\frac{1}{n!^{r}}|\Omega_{n}|\to_{n\to\infty}1$, \label{I:large size permutation}
\item for every sequence $(\sigma_{n})_{n}\in \prod_{n}\Omega_{n},$ if we let $\pi_{n}\colon \F_{r}\to \Sym(n)$ be the homomorphism $\pi_{n}(a_{i})=\sigma_{n,i},$ then $(\pi_{n})_{n}$ is a sofic approximation, \label{I:random sofic approximation}
\item for every free ultrafilter $\omega$, the Koopman representation of $\F_{r}\actson (Z_{\omega},u_{\omega})$ is weakly contained in the left regular representation. In particular $\F_{r}\actson (Z_{\omega},u_{\omega})$ is weak mixing has stable spectral gap. \label{I:stable spectral gap random}
\item for every free ultrafilter $\omega,$ we have that $\F_{r,\omega}'$ is trivial. \label{I:random trivial centralizer}
\end{enumerate}
\end{prop}

\begin{proof}
It is well-known (see \cite{Nica}) that we may find a sequence $\Omega_{n,0}\subseteq\Sym(n)^{r}$ with $\frac{1}{n!^{r}}|\Omega_{n,0}|\to 1$ so that for every sequence $(\sigma_{n})_{n}\in \prod_{n}\Omega_{n,0}$ one has that $(\pi_{n})_{n}$ as defined in (\ref{I:random sofic approximation}) is a sofic approximation.
View $\Sym(n)\subseteq B(\ell^{2}(n))$ associating a permutation $\tau$ with the operator $(\tau f)(j)=f(\tau^{-1}(j)).$ Note that the operators in $\Sym(n)$ leave $\C1,$ invariant, and so we may use them to define unitary operators $\tau\big|_{\ell^{2}_{n}(0)}$ on $\ell^{2}_{0}(n)=\left\{f\in \ell^{2}(n):\sum_{j=1}^{n}f(j)=0\right\},$ as $\ell^{2}_{0}(n)$ is the orthogonal complement of $\C1$ in $\ell^{2}(n).$ Given $\tau_{1},\cdots,\tau_{k}\in \Sym(n)$ and $a_{1},\cdots,a_{k}\in \C$ we may thus view $\sum_{j=1}^{k}a_{k}\tau_{k}\in B(\ell^{2}(n)),$ with similar comments for $B(\ell^{2}_{0}(n)).$ We endow $B(\ell^{2}(n)),B(\ell^{2}_{0}(n))$ with the operator norm.

For $\tau\in \Sym(n)$ we define $\tau\otimes \tau\in \Sym(\{1,\cdots,n\}\times \{1,\cdots,n\})$ by $(\tau\otimes \tau)(j,k)=(\tau(j),\tau(k))$ for $1\leq j,k\leq n.$ As above, these operators leave $\C1$ invariant, they also leave $\C J$ invariant where $J\colon\{1,\cdots,n\}^{2}\to \{0,1\}$ is given by $J(j,k)=\frac{n-1}{n}\delta_{j\ne k}.$ So we may restrict them to $\mathcal{H},$ the orthogonal complement of $\C 1+\C J$ in $\ell^{2}_{0}(\{1,\cdots,n\}^{2}).$ For concreteness, we note that
\[\mathcal{H}=\left\{f:\sum_{j=1}^{n}f(j,j)=0,\sum_{j\ne k}f(j,k)=0.\right\}.\]

Lastly, we define the left regular representation $\lambda\colon \F_{r}\to \mathcal{U}(\ell^{2}(\F_{r}))$ by
\[(\lambda(g)\xi)(x)=\xi(g^{-1}x)\mbox{ for $g\in \F_{r},\xi\in \ell^{2}(\F_{r})$}\]

By \cite[Theorem 3]{CollinsBordenave} and \cite[Theorem 5]{CollinsBordenave} we may find a sequence $\Omega_{n,1}\subseteq \Sym(n)^{r}$ with $\frac{1}{n!^{r}}|\Omega_{n,1}|\to 1$ and so that every $(\sigma_{n})_{n}\in \prod_{n}\Omega_{n,1}$ satisfies the following properties:
\begin{itemize}
    \item for every $a=\sum_{g\in \F_{r}}a_{g}g\in \C(\F_{r})$ we have that $\lim_{n\to\infty}\left\|\sum_{g\in \F_{r}}a_{g}\pi_{n}(g)\big|_{\ell^{2}_{0}(n)}\right\|=\left\|\sum_{g\in G}a_{g}\lambda(g)\right\|$\\
    \item for every $a=\sum_{g\in \F_{r}}a_{g}g\in \C(\F_{r})$ we have that
    \[\left\|\sum_{g\in \F_{r}}a_{g}(\pi_{n}(g)\otimes \pi_{n}(g))\big|_{\mathcal{H}}\right\|=\left\|\sum_{g\in \F_{r}}a_{g}\lambda(g)\right\|.\]
\end{itemize}
Set $\Omega_{n}=\Omega_{n,0}\cap \Omega_{n,1}.$
Properties (\ref{I:large size permutation}), (\ref{I:random sofic approximation}) are true by construction. We focus on proving the other properties. Fix a choice of $(\sigma_{n})_{n}\in \prod_{n}\Omega_{n}.$

Proof of (\ref{I:stable spectral gap random}):
Let $\xi\in L^{\infty}(Z_{\omega},u_{\omega}),$ with $\int \xi\,du_{\omega}=0.$ By Proposition \ref{P:loeb basics yo} applied to $\{z\in \C:|z|\leq \|\xi\|_{\infty}\},$ we may write $\xi=(\xi_{n})_{n\to\omega}$ where $\|\xi_{n}\|_{\infty}\leq \|\xi\|_{\infty}.$ It is easy to see that $\xi=(\xi_{n}-\left(\int \xi_{n}\,du_{n}\right)1)_{n\to\omega}$ so we may assume that $\int \xi_{n},du_{\omega}(z)=0$ (at the cost of now assuming that  $\|\xi_{n}\|_{\infty}\leq 2\|\xi\|_{\infty}$).

Then for every $a\in \C(\F_{r}),$ we have by Proposition \ref{P:computing integrals in ultraproducts}  that:
\begin{align*}
    \left\|\sum_{g\in F_{r}}a_{g}\pi_{\omega}(g)\xi\right\|_{2}=\lim_{n\to\omega}\left\|\sum_{g\in\F_{r}}a_{g}\pi_{n}(g)\xi_{n}\right\|_{\ell^{2}(u_{n})}&\leq \lim_{n\to\omega}\left\|\sum_{g\in \F_{r}}a_{g}\pi_{n}(g)\big|_{\ell^{2}_{0}(n)}\right\|\|\xi_{n}\|_{\ell^{2}(u_{n})}\\
    &=\left\|\sum_{g\in \F_{r}}a_{g}\lambda(g)\right\|\lim_{n\to\omega}\|\xi_{n}\|_{\ell^{2}(u_{n})},\\
    &=\left\|\sum_{g\in \F_{r}}a_{g}\lambda(g)\right\|\|\xi\|_{2}.
\end{align*}
By density of $L^{\infty}(Z_{\omega},u_{\omega})$ inside $L^{2}(Z_{\omega},u_{\omega}):$
\[\left\|\sum_{g\in \F_{r}}a_{g}\pi_{\omega}(g)\big|_{L^{2}_{0}(Z_{\omega})}\right\|_{B(L^{2}_{0}(Z_{\omega}))}\leq \left\|\sum_{g\in \F_{r}}a_{g}\lambda(g)\right\|_{B(\ell^{2}(\F_{r})},\]
for every $a=\sum_{g\in \F_{r}}a_{g}\lambda(g)\in \C(\F_{r}).$ By definition this means that the Koopman representation of $\F_{r}\actson (Z_{\omega},u_{\omega})$ is weakly contained in the left regular representation.

To see why this implies stable spectral gap, let $\rho$ be the Koopman representation of $\F_{r}\actson (Z_{\omega},u_{\omega}).$ Then $\rho\preceq \lambda,$ so $\rho\otimes \rho\preceq \lambda\otimes \lambda\cong \lambda^{\infty},$  by Fell's absorption principle. The Koopman representation of $\F_{r}$ on $(Z_{\omega}\times Z_{\omega},u_{\omega}\otimes u_{\omega})$ is isomorphic to $(\rho\oplus \rho)\oplus (\rho\otimes \rho),$ and is thus weakly contained in $\lambda^{\oplus \infty}.$ By nonamenability of $\F_{r},$ the trivial representation of $\F_{r}$ is not weakly contained in the left regular representation of $\F_{r}.$ Since the Koopman representation of $\F_{r}\actson (Z_{\omega}\times Z_{\omega},u_{\omega}\times u_{\omega}\otimes u_{\omega})$ is weakly contained in the left regular representation of $\F_{r},$ it thus follows that it does not weakly contained the trivial representation either. This means that the Koopman representation of $\F_{r}\actson (Z_{\omega}\times Z_{\omega},u_{\omega}\otimes u_{\omega})$ has spectral gap, and so $\F_{r}\actson (Z_{\omega},u_{\omega})$ has stable spectral gap. It is clear that stable spectral  gap implies weak mixing.

(\ref{I:random trivial centralizer}):
Let $\mathcal{K}$ be $M_{n}(\C)$ endowed with the trace inner product:
\[\ip{A,B}=\frac{1}{n}\Tr(B^{*}A),\]
and let $\|A\|_{2}$ be the resulting Hilbert space norm. We will use $\tr=\frac{1}{N}\Tr.$ Observe that we have a representation $\Ad$ of $\mathcal{U}(n)$ on $\mathcal{K}$ by $\Ad(U)V=UVU^{*}.$ This representation leaves $\C\id$ invariant, it also leaves $\C\mathcal{J}$ invariant, where $\mathcal{J}_{jk}=\frac{1}{\sqrt{n-1}}\delta_{j\ne k}.$ So we may restrict $\Ad$ to obtain a unitary representation on $\mathcal{K}_{0}=\mathcal{H}\ominus (\C\id +\C J).$
 Consider the representation of $\Sym(n)$ on $\mathcal{K}_{0}$ obtained by $\Ad\big|_{S_{n}}.$ Define $W\colon \mathcal{K}\to \ell^{2}(\{1,\cdots,n\}^{2})$ by
\[W(A)(j,k)=\sqrt{n}A_{j,k}.\]
It is direct to show that $W$ is a unitary, and is equivariant. Further $W(\C\id)=\C\id, W(\mathcal{J})=J,$ so $W$ induces a $\Sym(n)$-equivariant unitary $\mathcal{K}_{0}\to\mathcal{H}.$ Thus by our choice of $\Omega_{n,1}$ it follows that
\[\lim_{n\to\omega}\left\|\sum_{g\in \F_{r}}a_{g}\Ad(\pi_{n}(g))\big|_{\mathcal{K}_{0}}\right\|=\left\|\sum_{g\in \F_{r}}a_{g}\lambda(g)\right\|\]

Now suppose that $\tau=(\tau_{n})_{n\to\omega}\in \F_{r,\omega}'.$ By direct computation
\[\|\Ad(\pi_{n}(g))\tau_{n}-\tau_{n}\|_{2}^{2}=d_{\Hamm}(\pi_{n}(g)\tau_{n},\tau_{n}\pi_{n}(g))\to_{n\to\omega}0,\]
so it follows that $\tau_{n}$ is asymptotically fixed under $\Ad(\sigma_{n}(\F_{r})).$ Set $T_{n}=P_{\mathcal{K}_{0}}(\tau_{n}),$ where $P_{\mathcal{K}_{0}}$ denotes orthogonal projection. Since $P_{\mathcal{K}_{0}}$ commutes with $\Ad(\Sym(n)),$ it follows that
\[\|\Ad(\pi_{n}(g))T_{n}-T_{n}\|_{2}\to 0.\]
But then:
\begin{align*}
    0=\lim_{n\to\omega}\frac{1}{r}\sum_{j=1}^{r}\|\Ad(\pi_{n}(a_{j}))(T_{n})-T_{n}\|_{2}^{2}&=\lim_{n\to\omega}2\|T_{n}\|_{2}^{2}-2\Re(\frac{1}{r}\sum_{j=1}^{r}\tr(T_{n}^{*}\Ad(\pi_{n}(a_{j}))(T_{n}))\\
    &=2\lim_{n\to\omega}\|T_{n}\|_{2}^{2}-2\ip{\left(\frac{1}{r}\sum_{j=1}^{r}\Ad(\pi_{n}(a_{j}))T_{n}\right),T_{n}}\\
    &\geq 2\lim_{n\to\omega}\|T_{n}\|_{2}^{2}-\left\|\frac{1}{r}\sum_{j=1}^{r}\Ad(\pi_{n}(a_{j})\right\|\|T_{n}\|_{2}^{2}\\
    &=2\left(1-\left\|\frac{1}{r}\sum_{j=1}^{r}\lambda(a_{j})\right\|\right)\lim_{n\to\omega}\|T_{n}\|_{2}^{2}
\end{align*}
By nonamenability of $\F_{r}$ we have that $\left\|\frac{1}{r}\sum_{j=1}^{r}\lambda(a_{j})\right\|<1,$ so it follows that $\lim_{n\to\omega}\|T_{n}\|_{2}^{2}=0.$ By Pythagoras' theorem,
\[\lim_{n\to\omega} d_{\Hamm}(\tau_{n},1)=\lim_{n\to\omega}\|\tau_{n}-1\|_{2}^{2}=\lim_{n\to\omega}\|P_{\C\mathcal{J}}(\tau_{n})\|_{2}^{2}+\|P_{\mathcal{K}_{0}}(\tau_{n})\|_{2}^{2}=\lim_{n\to\omega}\|P_{\C\mathcal{J}}(\tau_{n})\|_{2}^{2},\]
the last line following from the fact from the definition of $T_{n}.$ But
\[\lim_{n\to\omega}\|P_{\C\mathcal{J}}(\tau_{n})\|_{2}^{2}=\lim_{n\to\omega}|\tr(\tau_{n}^{-1}J)|^{2}=\lim_{n\to\omega}\frac{1}{n\sqrt{n-1}}\sum_{j=1}^{n}\sum_{k\ne j}(\tau_{n})_{j,k}.\]
Since $\tau_{n}$ is a permutation, we have that
\[\sum_{k}(\tau_{n})_{j,k}=1.\]
Thus,
\[\lim_{n\to\omega}\frac{1}{n\sqrt{n-1}}\sum_{j=1}^{n}\sum_{k\ne j}(\tau_{n})_{j,k}\leq \lim_{n\to\omega}\frac{1}{\sqrt{n-1}}=0.\]
So we have shown that $\lim_{n\to\omega} d_{\Hamm}(\tau_{n},1)=1,$ and thus $(\tau_{n})_{n\to\omega}=\id.$

\end{proof}

We will not actually use part (\ref{I:random trivial centralizer}), but it is interesting in light of our Corollary \ref{L:ergodicity on Loeb space} stating that the maximal element of $\mathcal{S}_{\omega}(G\actson X)$ is in $\Sub(X),$ when $G_{\omega}'\actson (Z_{\omega},u_{\omega})$ is ergodic. In all of the following examples, the maximal element of $\mathcal{S}_{\omega}(G\actson X)$ which be much smaller than one expects and in some cases will not be an element of $\Sub(X).$ Of course, by our discussion a necessary condition for the maximal element of $\mathcal{S}_{\omega}(G\actson X)$ to not be in $\Sub(X)$ is that $G_{\omega}'\actson (Z_{\omega},u_{\omega})$ is not ergodic. In our examples, this is certainly true since $G_{\omega}'$ is in fact trivial!

\begin{example}\label{E:finite example}
Consider an algebraic action of the free  group $\F_{r}$ with $r>1$ on a finite group $X.$  By Proposition \ref{P:Corollary of BC} we may find a sofic approximation $\sigma_{n}\colon \F_{r}\to \Sym(n)$ so that for every free ultrafilter $\omega,$ we have that $\F_{r}\actson (Z_{\omega},u_{\omega})$ is weak mixing. If $\omega$ is a free ultrafilter on the natural numbers, then if $\nu=(\theta)_{*}(u_{\omega}),$ we know that $\F_{r}\actson (X,\nu)$ is weak mixing. Since $X$ is finite, this forces $\nu$ to be a point mass. So $\theta(Z_{\omega})\subseteq \Fix_{\F_{r}}(X)$ almost surely. So the maximal subgroup of $\mathcal{S}_{\omega}(\F_{r}\actson X)$ is  $\Fix_{\F_{r}}(X).$ In this example, since $\omega$ was arbitrary, we can even say that $\Fix_{\F_{r}}(X)$ is the largest subgroup of $X$ so that there is a sequence $\mu_{n}\in \Prob(X^{n})$ which is asymptotically supported on topological microstates and with $\mu_{n}\to^{lw^{*}}m_{\Fix_{G}(X)}$ as $n\to\infty$ (and not just as $n$ approaches some ultrafilter).
\end{example}

\begin{example}\label{E:coinduced example}
The previous example was not ergodic. We can fix this as follows. Let $r$ be an integer at least $3,$ and let  $a_{1},\cdots,a_{r}$ be the free generators of $\F_{r}.$ Again fix a sofic approximation $\sigma_{n}\colon \F_{r}\to \Sym(n)$ so that for every free ultrafilter $\omega,$ we have that $\F_{r}\actson (Z_{\omega},u_{\omega})$ has Koopman representation which is weakly contained in the left regular representation. Take an $s\in \{2,\cdots, r\}$ and $\F_{s}\actson X$ where $X$ is a finite group and we regard $\F_{s}=\ip{a_{1},\cdots,a_{s}}\leq \F_{r}.$
Let $\F_{r}\actson \widetilde{X}$ be the coinduced action. This action is defined as follows:
\[\widetilde{X}=\{x\in X^{\F_{r}}:x(gh)=h^{-1}x(g)\mbox{ for all $g\in \F_{r},h\in \F_{s}$}\},\]
and $\F_{r}\actson \widetilde{X}$ by left shifts.
Observe that $\widetilde{X}$ is a compact subgroup of $X^{\F_{r}},$ so this is still an algebraic action.
Since $\F_{s}$ has infinite index in $\F_{r},$ this action is ergodic with respect to $m_{\widetilde{X}}$. Fix a free ultrafilter $\omega,$ and let $\Theta\colon Z_{\omega}\to \widetilde{X}$ be a topological microstate. Let $\Phi\colon \widetilde{X}\to X$ be given by $\Phi(x)=x(1).$ Then $\Phi\circ \Theta$ is $\F_{s}$-equivariant, and thus as in the first example we know that $(\Phi\circ \Theta)_{*}(u_{\omega})\subseteq \Fix_{\F_{s}}(X)$ almost surely. By $\F_{r}$-equivariance, we know that $\Theta(Z_{\omega})\subseteq (\Fix_{\F_{s}}(X))^{\F_{r}}\cap \widetilde{X}$ almost surely.
In this case, computation of the maximal element $Y$ of $\mathcal{S}_{\omega}(\F_{r}\actson \widetilde{X})$ is elusive. However, we can say that if $\Fix_{\F_{s}}(X)=\{1\},$ then  $Y\ne (\Fix_{\F_{s}}(X))^{\F_{r}}\cap \widetilde{X}.$ Indeed, suppose that   $Y=(\Fix_{\F_{s}}(X))^{\F_{r}}\cap \widetilde{X}.$ Then $\F_{r}\actson (Y,m_{Y})$ is isomorphic to the generalized Bernoulli shift $\F_{r}\actson (\Fix_{\F_{s}}(X))^{\F_{r}/\F_{s}}$ and is thus ergodic. By assumption, there is a sequence of measures $\mu_{n}\in \Prob(\widetilde{X}^{n})$ which are asymptotically supported on topological microstates and $\mu_{n}\to_{n\to\omega}^{lw^{*}}m_{Y}.$ By ergodicity, we have that $\mu_{n}$ are asymptotically supported on measure microstates, in particular $\F_{r}\actson (Y,m_{Y})$ is sofic with respect to $(\sigma_{n})_{n},\omega.$ By definition, this means that $\F_{r}\actson (Y,m_{Y})$ is a factor of $\F_{r}\actson (Z_{\omega},u_{\omega}).$ Thus the Koopman representation of $\F_{r}\actson (Y,m_{Y})$ would be weakly contained in the left regular representation of $\F_{r}.$ By restriction, this implies that the Koopman representation of $\F_{s}\actson (Y,m_{Y})$ is weakly contained in the left regular representation of $\F_{s}.$ But this action has $\F_{s}\actson (\Fix_{\F_{s}}(X),m_{\Fix_{\F_{s}}(X)})$ as a factor (via the map $\Phi$).  Since $|\Fix_{\F_{s}}(X)|\geq 2,$ this implies that $\F_{s}\actson (Y,m_{Y})$ is not ergodic. So the Koopman representation of $\F_{s}\actson (Y,m_{Y})$ contains the trivial representation, and we already saw it was weakly contained in the left regular representation. Thus the trivial representation of $\F_{s}$ is weakly contained in the left regular representation of $\F_{s},$ contradicting nonamenability of $\F_{s}.$
\end{example}

\begin{example}\label{E:algebraic austin} We investigate an algebraic version of \cite[Example 3.5]{AustinAdd}. So let $\F_{4}=\ip{a,b,c,d},$ and view $\F_{2}=\{a,b\}.$ Consider the trivial action $\F_{2}\actson (\Z/2\Z).$ As in Example \ref{E:coinduced example},   consider the coinduced action $\F_{4}\actson X.$ Constructed a random sofic approximation  $\sigma_{k}\colon \F_{4}\to \Sym(2k)$ as follows. Let $\sigma_{k,c},\sigma_{k,d}$ be two permutations of $\Sym(2k)$ chosen independently at random with respect to the uniform probability measures on $\Sym(2k).$  Let $U_{k}=\{1,\cdots,k\},V_{k}=\{k+1,\cdots,2k\}.$ Let $\sigma_{k,1}^{a},\sigma_{k,1}^{b}\in \Sym(U_{k}),$ and $\sigma_{k,2}^{a},\sigma_{k,2}^{b}\in \Sym(V_{k})$ be chosen independently at random with respect to the uniform probability measures on $\Sym(U_{k}),\Sym(V_{k})$ and set $\sigma_{k,a}=\sigma_{k,2}^{a}\cup \sigma_{k,2}^{a},\sigma_{k,b}=\sigma_{k,b}^{1}\cup \sigma_{k,b}^{2}.$ Let $\sigma_{k}$ be the unique homomorphism of $\F_{4}$ such that $\sigma_{k}(a)=\sigma_{k,a}$, $\sigma_{k}(b)=\sigma_{k,b},$ $\sigma_{k}(c)=\sigma_{k,c},$ $\sigma_{k}(d)=\sigma_{k,d}.$ Then $(\sigma_{k})_{k}$ is a sofic approximation with high probability. Set $\widetilde{\F}_{2}=\ip{a,b}.$ With high probability this sofic approximation has the property that $\F_{2}\actson (Z_{\omega},u_{\omega})$ has an ergodic decomposition into two pieces given by $U=(U_{k})_{k\to\omega},V=(V_{k})_{k\to\omega},$ and by Proposition \ref{P:Corollary of BC} furthermore has the property that $\widetilde{\F}_{2}\actson (Z_{\omega},u_{\omega})$ is ergodic. So we may fix a choice of $(\sigma_{n})_{n}$ which is a sofic approximation, and which has the property that $\F_{2}\actson (Z_{\omega},u_{\omega})$ has an ergodic decomposition into two pieces $U,V$ and so $\widetilde{\F}_{2}\actson (Z_{\omega},u_{\omega})$ is ergodic.

Note that $U,V$ give topological microstates $\Phi_{U}\colon Z_{\omega}\to \Z/2\Z,\Phi_{V}\colon Z_{\omega}\to \Z/2\Z$  for $\F_{2}\actson \Z/2\Z$  by
\[\Phi_{U}(z)=\begin{cases}
0&, \textnormal {if $z\notin U$}\\
1+2\Z&, \textnormal{if $z\in U$},
\end{cases}\]
\[\Phi_{V}(z)=\begin{cases}
0&, \textnormal{ if $z\notin V$}\\
1+2\Z&,\textnormal{ if $z\in V.$}
\end{cases}.\]
These naturally give rise to topological microstates  $\Theta_{U}\colon Z_{\omega}\to X,\Phi_{V}\colon Z_{\omega}\to X$ by
\[\Theta_{U}(z)(g)=\Phi_{U}(g^{-1}z),\,\,\, \Theta_{V}(z)(g)=\Phi_{V}(g^{-1}z).\]
It is straightforward to check that because $\Theta_{U},\Theta_{V}$ are $\F_{2}$-equivariant, that these maps are indeed in $\widetilde{X}$ (i.e. that $\Theta_{U}(z)(gh)=h^{-1}\Theta_{U}(z)(g)$ for all $g\in \F_{4},h\in \F_{2},z\in Z_{\omega}$ and similarly for $\Theta_{V}$). We remark that what is going on here is the canonical adjunction property of co-induction in the topological setting: $\Meas_{\F_{2}}(Z_{\omega},\Z/2\Z)\cong \Meas_{\F_{4}}(Z_{\omega},X).$ We also have two other topological microstates $0\colon Z_{\omega}\to X,$ $1\colon Z_{\omega}\to X$ given by $0(z)(g)=0,$ $1(z)(g)=1+2\Z$ for all $g\in \F_{4},z\in Z_{\omega}.$ Observe that $\{0,1,\Theta_{U},\Theta_{V}\}$ is a group with respect to pointwise addition.

We claim that every topological microstate is almost surely equal to one of $0,1,\Theta_{U},\Theta_{V}.$ Suppose that $\Theta\colon Z_{\omega}\to X$ is a topological microstate. Let $\Phi\colon X\to \Z/2\Z$ be given by $\Phi(x)=x(1).$ Then $\Phi\circ \Theta$ is almost surely $\F_{2}$-equivariant, so $(\Phi\circ \Theta)^{-1}(\{0\})$ is almost surely $\F_{2}$-invariant. Since $\F_{2}\actson Z_{\omega}$ has an ergodic decomposition into two pieces given by the sets $U,V,$ this means that $(\Phi\circ \Theta)^{-1}(\{0\})$ is almost surely equal to one of: $U,V,Z_{\omega},\varnothing.$ From $\F_{4}$-equivariance of $\Theta,$ we see that $\Theta$ is almost surely equal to $\Theta_{U},\Theta_{V},0,1$ if  $(\Phi\circ \Theta)^{-1}(\{0\})$ is almost surely equal to $U,V,Z_{\omega},\varnothing$ respectively.

Thus every topological microstate is almost everywhere equal to one of $\{0,1,\Theta_{U},\Theta_{V}\}.$ So the maximal element $Y$ of $\mathcal{S}_{\omega}(G\actson X)$ is given by
\[Y(z)=\{0,1+2\Z,\Theta_{U}(z),\Theta_{V}(z)\},\]
where we abuse notation and regard $1+2\Z$ as the element of $X^{\F_{4}}$ whose value at $g$ is $1+2\Z$ for every $g.$ Since $\widetilde{\F}_{2}\actson Z_{\omega}$ is ergodic, it is easy to see that $|\{0,1+2\Z,\Theta_{U}(z),\Theta_{V}(z)\}|$ is almost surely equal to $4.$
 In this case it is clear that $Y(z)\ne X$ for almost every $z\in Z_{\omega}$ since $|Y(z)|=4$ for almost every $z\in Z_{\omega}$ and $\widetilde{X}$ is uncountable.

Moreover, in this example we see that $Y(z)$ genuinely depends upon $z.$ Indeed, by using ergodicity of $\widetilde{\F}_{2}\actson Z_{\omega},$ we can find a $g\in \widetilde{\F}_{2}$ so that
\[0<u_{\omega}(gU\cap U)<u_{\omega}(U).\]
Fix such a $g,$ and consider $z\in gU\cap U,\widetilde{z}\in U\cap g^{-1}U$ with $|\{0,1,\Theta_{U}(z),\Theta_{V}(z)\}|=4=|\{0,1,\Theta_{U}(\widetilde{z}),\Theta_{V}(\widetilde{z})\}|.$  We claim that $Y(z)\ne Y(\widetilde{z}).$ Indeed, $\Theta_{U}(z)\ne 0,1,$ since $|\{0,1,\Theta_{U}(z),\Theta_{V}(z)\}|=4.$ Moreover, $\Theta_{U}(z)\ne \Theta_{V}(\widetilde{z}),$ since they have different values at the identity coordinate. Lastly, $\Theta_{U}(z)(g)=\Phi_{U}(g^{-1}z)=1+2\Z,$ and $\Theta_{U}(\widetilde{z})(g)=\Phi_{U}(g^{-1}\widetilde{z})=0,$ so $\Theta_{U}(z)\ne \Theta_{U}(\widetilde{z}).$ So we have shown that $\Theta_{U}(z)\notin Y(\widetilde{z}),$ and thus we see that $Y(z)\ne \widetilde{Y}(z).$ Since $z\in gU\cap U,\widetilde{z}\in U\cap g^{-1}U$  are positive measure sets, and $|\{0,1,\Theta_{U}(z),\Theta_{V}(z)\}|=4$ for almost every $z,$ we thus see that $Y(z)$ is not almost surely constant.

\end{example}

We remark that each of Examples \ref{E:coinduced example}, \ref{E:algebraic austin} we do not have to use the results of Bordenave-Collins, and can simply use $f$-invariant entropy (as defined in \cite{Bowenfinvariant}) and earlier results of \cite{Fried}.

For instance, in Example \ref{E:coinduced example} the results of \cite{Fried} imply that a random sofic approximation $(\sigma_{k})_{k}$ of $\F_{r}$ acts ergodically on $(Z_{\omega},u_{\omega})$ for any free ultrafilter $\omega,$ and that in fact $\F_{s}$ acts ergodically on $(Z_{\omega},u_{\omega}).$ Since the $f$-invariant entropy $\F_{s}\actson (X,m_{X})$ is negative (obvious from the definition in \cite{Bowenfinvariant}), the argument in \cite[Proposition 6.9]{Me5} shows that we may find a $(\sigma_{k})_{k}$ so that $\F_{s}$ acts ergodically on $(Z_{\omega},u_{\omega})$ for every free ultrafilter $\omega,$ and so that $F_{s}\actson (X,m_{X})$ is not sofic with respect to $(\sigma_{k})_{k}.$ This is sufficient to complete part of the argument given in Example \ref{E:coinduced example}, at least to the point of showing that the maximal element of $\mathcal{S}_{\omega}(\F_{r}\actson \widetilde{X})$ is not all of $\widetilde{X}.$ Similar remarks apply to Example \ref{E:algebraic austin}. In fact, in this case we can already get away with finding some sofic approximation $(\sigma_{k})_{k}$ for which $h_{(\sigma_{k})_{k}}(\F_{2}\actson (\Z/2\Z,u_{\Z/2\Z}))=-\infty,$ and for which $\widetilde{\F}_{2}$ acts ergodically on $(Z_{\omega},u_{\omega})$ for any free ultrafilter $\omega,$ and these two properties can be shown by combining $f$-invariant entropy (following the argument in \cite[Proposition 6.9]{Me5}) with the results of \cite{Fried}.

\section{Discussion on the Use of Ultrafilters}

Since ultrafilters are not concrete and only exist via some abstract argument, it is likely that in order to apply our methods in the future we will need ultrafilter-free version of our results. In this section, we give ultrafilter versions of our main results  (deducing them from their ultrafilter versions). In the last subsection of this section, we outline one special case where we know how to remove the usage of ultrafilters from the proof. This is the case the group being acted on is abelian, \emph{and} where the sofic approximation has ergodic commutant in the sense of Definition \ref{D:ergodic commutant no ultrafilter}.
\subsection{Ultrafilter-Free versions of some of the results}\label{S:ultrafilter free versions}
%
%We first make the following observation.

We state an ultrafilter-free versions of Corollaries \ref{C:sick corollary  bro} \ref{C:simpler under ergodicity}, for which we need the following lemma. We use the following notations. If $(X,d)$ is a metric space and $F\subseteq X,$ we let $N_{r}(F)=\{x\in X:\mbox{ there is a $y\in F$ with $d(y,x)<r$}\},$ and $\overline{N}_{r}(F)=\{x\in X:\mbox{ there is a $y\in F$ with $d(y,x)\leq r$}\}.$

\begin{lem}\label{L:wont you be my neighbor?}
Let $(X,d)$ be a compact metric space, and fix $\mu\in \Prob(X).$ Let $\varepsilon,r,s\in (0,\infty)$ with $r<s.$ Then
\[\{\nu\in \Prob(X):\nu(N_{r}(F))<\mu(N_{s}(F))\mbox{ for all $F\in \mathcal{F}(X)$}\}\]
is a weak$^{*}$ neighborhood of $\mu.$
\end{lem}

\begin{proof}
Let
\[\mathcal{O}=\{\nu\in \Prob(X):\nu(N_{r}(F))<\varepsilon+\mu(N_{s}(F))\mbox{ for all $F\in \mathcal{F}(X)$}\},\]
and suppose, for the sake of contradiction, that there is a sequence $\mu_{n}\in \Prob(X)\cap \mathcal{O}^{c}$ so that $\mu_{n}\to \mu$ weak$^{*}.$ Since $\mu_{n}\notin\mathcal{O},$ we can find a sequence $F_{n}\in \mathcal{F}(X)$ so that $\mu_{n}(N_{r}(F_{n}))\geq \varespilon+\mu(N_{s}(F_{n})).$ Since $\mathcal{F}(X)$ is compact, we may, and will, assume that there is a $F\in\mathcal{F}(X)$ so that $F_{n}\to F$ in the Hausdorff topology. Fix $r_{0},r_{1},s_{1}\in (r,s)$ with $r_{0}<r_{1}<s_{1}.$ Since $\mu_{n}\to \mu$ weak$^{*}$ and $r_{0}<r_{1},$ we have that:
\[\mu(N_{r_{1}}(F))\geq \limsup_{n\to\infty}\mu_{n}(N_{r_{0}}(F)).\]
Since $F_{n}\to F,$ for all large $n$ we have that $N_{r}(F_{n})\subseteq N_{r_{0}}(F).$ Hence,
\[\mu(N_{r_{1}}(F))\geq \limsup_{n\to\infty}\mu_{n}(N_{r}(F_{n}))\geq \varepsilon+ \limsup_{n\to\infty}\mu(N_{s}(F_{n})).\]
Since $F_{n}\to F,$ for all large $n$ we have that $N_{s}(F_{n})\supseteq N_{s_{1}}(F).$ So:
\[\mu(N_{r_{1}}(F))\geq \varepsilon+\mu(N_{s_{1}}(F)),\]
and since $r_{1}<s_{1}$ this is an obvious contradiction.

\end{proof}

\begin{lem}\label{L:more  semicty}
Let $(X,d)$ be a compact metric space. Suppose $\eta_{k}\in \Prob(\mathcal{F}(X)),$ $\eta_{k}\to \eta$ weak$^{*},$ and fix a measure $\mu\in \Prob(X).$ Then, for every $r>0,$ we have that
\[\limsup_{k\to\infty}\int \mu(\overline{N}_{r}(F))\,d\eta_{k}(F)\leq \int \mu(\overline{N}_{r}(F))\,d\eta(F).\]

\end{lem}

\begin{proof}
We claim that  $\Phi\colon \mathcal{F}(X)\to [0,1]$ given by $\Phi(F)=\mu(\overline{N}_{r}(F))$ is upper semicontinuous. Suppose that $F_{n}\in \mathcal{F}(X),$ and $F_{n}\to F.$ Then for all $s>r,$ we have that $N_{s}(F)\supseteq \overline{N}_{r}(F_{n})$ for all large $n.$ Thus
\[\mu(N_{s}(F))\geq \limsup_{n\to\infty}\mu(\overline{N}_{r}(F_{n})).\]
Since $s>r$ was arbitrary, we can let $s\to r$ to see that
\[\mu(\overline{N}_{r}(F))\geq \limsup_{n\to\infty}\mu(\overline{N}_{r}(F_{n})).\]
Thus $F\mapsto \mu(\overline{N}_{r}(F))$ is upper semicontinuous.

Hence we may write $\Phi(F)=\inf_{\Phi_{\alpha}}\Phi_{\alpha}(F),$ where the infimum is over all continuous functions $\Phi_{\alpha}\colon \mathcal{F}(X)\to [0,1]$ which have $\Phi_{\alpha}\geq \Phi.$ Hence, for all $\eta\in \Prob(\mathcal{F}(X))$ we have:
\[\int \mu(\overline{N}_{r}(F))\,d\eta(F)=\int \inf_{\alpha}\Phi_{\alpha}(F)\,d\eta(F)=\inf_{\alpha}\int \Phi_{\alpha}(F)\,d\eta(F),\]
the last part following, for example, by \cite[Proposition 7.12]{Folland}.
This shows that $\eta\mapsto \int \mu(\overline{N}_{r}(F))\,d\eta(F)$ is a upper semicontinuous function on $\Prob(\mathcal{F}(X)),$ which is equivalent to the conclusion of the lemma.

\end{proof}

\begin{cor}\label{C:messy corollary}
Let $G$ be a countable, discrete, sofic group with sofic approximation $\sigma_{k}\colon G\to S_{d_{k}}.$ Let $G\actson X$ be an algebraic action. The following are equivalent:
\begin{enumerate}[(i)]
\item There does not exist a sequence $\mu_{k}\in \Prob(X^{d_{k}})$ which is asymptotically supported on topological microstates and so that $\mu_{k}\to^{lw^{*}}m_{X}$ as $k\to\infty.$ \label{I:classical nonexistence}
\item There is a $\eta\in \Prob_{G}(\Sub(X))$ with $\eta\ne\delta_{X},$ and a sequence $(Y_{k})_{k}\in \Sub(X)^{d_{k}}$ which satisfies the following property. Given any sequence $(\phi_{k})_{k}$ of topological microstates, there is an increasing sequence of natural numbers $(k_{l})_{l}$ so that $\lim_{l\to\infty}(Y_{k_{l}})_{*}(u_{d_{k_{l}}})=\eta,$ and
\[\lim_{l\to\infty}\frac{1}{d_{k_{l}}}\sum_{j=1}^{d_{k_{l}}}\rho(\phi_{k_{l}}(j),Y_{k_{l}}(j))=0.\]\label{I:messy nonexistence subsequences}
\end{enumerate}

\end{cor}

\begin{proof}
(\ref{I:classical nonexistence}) implies (\ref{I:messy nonexistence subsequences}): Suppose that (\ref{I:classical nonexistence}) holds. By Proposition \ref{P: this is not hard}, we may find a free ultrafilter $\omega$ on the natural numbers so that for every sequence $\mu_{k}\in \Prob(X^{d_{k}})$ which is supported on topological microstates as $k\to\omega,$ we have that $\mu_{k}$ does not locally weak$^{*}$ converge to $\mu$ as $k\to\omega.$ By Corollary \ref{C:this the main theorem yo!}, we may find a $Y\in \Meas(Z_{\omega},\Sub(X))$ which is $G$-equivariant, so that $\eta=Y_{*}(u_{\omega})\ne \delta_{X},$ and which absorbs all topological microstates with respect to $\omega.$ By Proposition \ref{P:basics compact Loeb} we may assume that $Y=(Y_{k})_{k\to\omega}.$ Now suppose that $(\phi_{k})_{k}$ is a sequence of topological microstates for $G\actson X.$ Let $\mathcal{O}_{n}$ be a decreasing sequence of weak$^{*}$-neighborhoods of $\eta$ in $\Prob(\Sub(X)).$ Since $Y$ absorbs all topological microstates, we have by Proposition \ref{P:what it means to absorb} that
\[\lim_{k\to\omega}\frac{1}{d_{k}}\sum_{j=1}^{d_{k}}\rho(\phi_{k}(j),Y_{k}(j))=0.\]
By Proposition \ref{P:computing integrals in ultraproducts} we know that$\eta=Y_{*}(u_{\omega})=\lim_{k\to\omega}(Y_{k})_{*}(u_{d_{k}})$. Combining this fact with the above equation, we may find $B_{n}\in \omega$ so that for every $k\in B_{n}$ we have:
\begin{itemize}
\item $(Y_{k})_{*}(u_{d_{k}})\in \mathcal{O}_{n},$ and
\item $\frac{1}{d_{k}}\sum_{j=1}^{d_{k}}\rho(\phi_{k}(j),Y_{k}(j))<2^{-n}.$
\end{itemize}
Since $B_{n}\cap F^{c}\in \omega$ for every finite $F\subseteq \N,$ it follows that we may find an increasing sequence $k_{l}$ of natural numbers with $k_{l}\in B_{l}$ for every $l\in \N.$ By construction, we then have that $\lim_{l\to\infty}(Y_{k_{l}})_{*}(u_{d_{k_{l}}})=\eta$ and
\[\lim_{l\to\infty}\frac{1}{d_{k_{l}}}\sum_{j=1}^{d_{k_{l}}}\rho(\phi_{k_{l}}(j),Y_{k_{l}}(j))=0.\]

(\ref{I:messy nonexistence subsequences}) implies (\ref{I:classical nonexistence}): Assume that(\ref{I:messy nonexistence subsequences}) holds, but that there is a sequence $\mu_{k}\in \Prob(X^{d_{k}})$ which is asymptotically supported on topological microstates and which has $\mu_{k}\to^{lw^{*}}m_{X}$ as $k\to\infty.$ Suppose that $\eta,(Y_{k})_{k}$ satisfy the conclusion of (\ref{I:messy nonexistence subsequences}). Let $\rho$ be a translation-invariant metric on $X.$

Since $\eta\ne \delta_{X},$ we may choose an $r>0$ so that
\[\int m_{X}(\overline{N_{r}(Z)})\,d\eta(Z)<1.\]
Set $c=\int m_{X}(\overline{N}_{r}(Z))\,d\eta(Z),$ and choose $\varespilon\in \left(0,\frac{1-c}{3}\right),s\in (0,r).$
By Lemma \ref{L:wont you be my neighbor?} it follows that for all sufficiently large $k$ we have that
\begin{align}\label{E:integral ajldskjlaf}
\int u_{d_{k}}(\{j:\phi(j)\in \overline{N}_{s}(Y_{k}(j))\})\,d\mu_{k}(\phi)=\frac{1}{d_{k}}\sum_{j=1}^{d_{k}}\mu_{k,j}(\overline{N}_{s}(Y_{k,j}))&< \varepsilon+\frac{1}{d_{k}}\sum_{j=1}^{d_{k}}m_{X}(\overline{N}_{r}(Y_{k,j}))\\ \nonumber
&=\varepsilon+\int m_{X}(\overline{N}_{r}(Z))\,d(Y_{k})_{*}(u_{d_{k}})(Z).
\end{align}
Since $\mu_{k}$ is asymptotically supported on topological microstates, we may find a sequence of subsets $A_{k}\subseteq X^{d_{k}}$ with $\mu_{k}(A_{k})\to 1,$ and so that for all $g\in G,$ $\delta>0$
\[\lim_{k\to\infty}\inf_{\phi\in A_{k}}u_{d_{k}}(\{j:\rho(\phi(\sigma_{k}(g)(j)),g\phi(j))<\delta\})=1.\]
Since $\mu_{k}(A_{k})\to 1,$ equation (\ref{E:integral ajldskjlaf}) shows that for all large $k:$
\begin{equation*}
\frac{1}{\mu_{k}(A_{k})}\int_{A_{k}} u_{d_{k}}(\{j:\phi(j)\in \overline{N}_{s}(Y_{k}(j))\})\,d\mu_{k}(\phi)<2\varepsilon+\int m_{X}(\overline{N}_{r}(Z))\,d(Y_{k})_{*}(u_{d_{k}})(Z).
\end{equation*}
Hence, we may find a natural number $K$ so that for all $k\geq K,$ there is a $\phi_{k}\in A_{k}$ with
\begin{equation}\label{E:cantlw*}
u_{d_{k}}(\{j:\phi_{k}(j)\in \overline{N}_{s}(Y_{k}(j))\})<2\varepsilon+\int m_{X}(\overline{N}_{r}(Z))\,d(Y_{k})_{*}(u_{d_{k}})(Z).
\end{equation} Set $\phi_{k}=e$ for every $k\in \{1,\dots,K-1\}.$ Then $(\phi_{k})_{k}$ is a sequence of topological microstates, and so we may let $(k_{l})_{l}$ be an increasing sequence of natural numbers as in the conclusion to (\ref{I:messy nonexistence subsequences}). We then have that
\[\limsup_{l\to\infty}u_{d_{k_{l}}}(\{j:\phi_{k_{l}}(j)\in \overline{N}_{s}(Y_{k_{l}}(j))\})\leq 2\varepsilon+\limsup_{l\to\infty}\int m_{X}(\overline{N}_{r}(Z))\,d(Y_{k_{l}})_{*}(u_{d_{k_{l}}})(Z)\leq 2\varepsilon+c,\]
by Lemma \ref{L:more  semicty}. It then follows that for all large $l$ we have that $u_{d_{k_{l}}}(\{j:\phi_{k_{l}}(j)\in \overline{N}_{s}(Y_{k_{l}}(j))\})\leq 3\varespilon+c.$ Since $3\varepsilon+c<1,$ this contradicts the assumption that
\[\lim_{l\to\infty}\frac{1}{d_{k_{l}}}\sum_{j=1}^{d_{k_{l}}}\rho(\phi_{k_{l}}(j),Y_{k_{l}}(j))=0.\]

\end{proof}

In the case that the group being acted on is abelian, Corollary \ref{C:messy corollary} can be rephrased more positively as follows. Recall that if $X$ is abelian, then $\widehat{X}$ is the set of all continuous homomorphism $X\to \R/\Z.$

\begin{cor}\label{C:messy corollary 2 for some reason}
Let $G$ be a countable, discrete group with sofic approximation $\sigma_{k}\colon G\to S_{d_{k}}.$ Fix an algebraic action $G\actson X$ with $X$ abelian. Then, the following are equivalent:
\begin{enumerate}[(i)]
\item There exists a sequence $(\mu_{k})_{k}\in \prod_{k}\Prob(X^{d_{k}})$ so that $\mu_{k}\to^{lw^{*}}m_{X}.$ \label{I:lw* existence abelian}
\item For every $\alpha\in \widehat{X},$ $\alpha \ne 0,$ and for every sequence $(k_{l})_{l}$ of increasing integers, and every sequence $A_{k_{l}}\subseteq \{1,\dots,d_{k_{l}}\}$ so that $\lim_{l\to\infty}u_{d_{k_{l}}}(A_{k_{l}})>0,$ there is a sequence $(\phi_{k})_{k}$ of topological microstates for $G\actson X$ such that
\[\liminf_{l\to\infty}\frac{1}{|A_{k_{l}}|}\sum_{j\in A_{k_{l}}}|\ip{\phi_{k_{l}}(j),\alpha}|>0.\]
\label{I:messy existence abelian}
\end{enumerate}
\end{cor}

\begin{proof}

(\ref{I:lw* existence abelian}) implies (\ref{I:messy existence abelian}):

Let $\alpha,k_{l},A_{k_{l}}$ be as in the hypothesis to (\ref{I:messy existence abelian}). Suppose that $(\mu_{k})_{k}\in \prod_{k}\Prob(X^{d_{k}})$ is asymptotically supported on topological microstates, and $\mu_{k}\to^{lw^{*}}m_{X}.$ We then have that
\[\int_{X^{d_{k_{l}}}}\frac{1}{|A_{k_{l}}|}\sum_{j\in A_{k_{l}}}\exp(2\pi i \ip{\phi(j),\alpha})=\frac{1}{|A_{k_{l}}|}\sum_{j\in A_{k_{l}}}\widehat{\mu}_{k,j}(\alpha).\]
Hence,
\[\int_{X^{d_{k}}}\frac{1}{|A_{k_{l}}|}\sum_{j\in A_{k_{l}}}|\exp(2\pi i \ip{\phi(j),\alpha})-1|^{2}\,d\mu_{k}(\phi)=2-2\cdot\left(\frac{1}{|A_{k_{l}}|}\sum_{j\in A_{k_{l}}}\rea(\widehat{\mu}_{k,j}(\alpha))\right).\]
Since $\mu_{k}\to^{lw^{*}}m_{X}$ and $\liminf_{l\to\infty}u_{d_{k_{l}}}(A_{k_{l}})>0,$ we have that
\[\liminf_{l\to\infty}\int_{X^{d_{k}}}\frac{1}{|A_{k_{l}}|}\sum_{j\in A_{k_{l}}}|\exp(2\pi i \ip{\phi(j),\alpha})-1|^{2}\,d\mu_{k}(\phi)=2.\]
Thus we may find a sequence $(\phi_{k_{l}})_{l}$ of topological microstates for $G\actson X$ so that
\[\frac{1}{|A_{k_{l}}|}\sum_{j\in A_{k_{l}}}|\exp(2\pi i \ip{\phi_{k_{l}}(j),\alpha})-1|^{2}=2.\]
Set $\phi_{k}=0$ for every $k\in \N\setminus\{k_{l}:l\in \N\}.$ Then $(\phi_{k})_{k}$ are topological microstates for $G\actson X$ and clearly
\[\liminf_{l\to\infty}\frac{1}{|A_{k_{l}}|}\sum_{j\in A_{k_{l}}}|\ip{\phi_{k_{l}}(j),\alpha}|>0.\]

(\ref{I:messy existence abelian}) implies (\ref{I:lw* existence abelian}): Fix a compatible metric $\rho$ on $X.$ Suppose that  (\ref{I:messy existence abelian}) holds, but that (\ref{I:lw* existence abelian}) fails. We can then find $\eta\in \Prob(\Sub(X)),(Y_{k})_{k}\in \prod_{k}\Sub(X)^{d_{k}}$ which satisfies (\ref{I:messy nonexistence subsequences}) of Corollary \ref{C:messy corollary}. Since $\eta\ne \delta_{X},$ we have that
\[\eta\left(\bigcup_{\alpha\in \widehat{X}:\alpha\ne 0}\{Y:Y\subseteq \{\alpha\}^{o}\}\right)>0,\]
thus we can find an $\alpha\in \widehat{X}\setminus\{0\}$ so that $\eta(\{Y:Y\subseteq \{\alpha\}^{o}\})>0.$ Let $E=\{Y:Y\subseteq \{\alpha\}^{o}\}.$ Fix a strictly increasing sequence $(k_{l})_{l}$ of integers with $\lim_{l\to\infty}(Y_{k_{l}})_{*}(u_{d_{k_{l}}})=\eta.$ For every open neighborhood $U$ of $E$ we have that
\[\liminf_{l\to\infty}(Y_{k_{l}})_{*}(u_{d_{k_{l}}})(U)\geq \eta(E).\]
Hence, by a diagonal argument, we may find a sequence $A_{k_{l}}\subseteq \{1,\dots,d_{k_{l}}\}$ with $\liminf_{l\to\infty}u_{d_{k_{l}}}(A_{k_{l}})\geq \eta(E)>0,$ and so that for every neighborhood $U$ of $E,$ there is an $L_{U}\in \N$ with $Y_{k_{l}}(j)\in U$ for all $l\geq L_{U},j\in A_{k_{l}}.$ Let $(\phi_{k})_{k}$ be a sequence of topological microstates for $G\actson X.$ We claim that
\[\lim_{l\to\infty}\frac{1}{|A_{k_{l}}|}\sum_{j\in A_{k_{l}}}|\ip{\phi_{k_{l}}(j),\alpha}|=0.\]

If the claim is false, then we can find a further subsequence $(k_{l_{p}})$ and an $\varepsilon>0$ so that
\[\lim_{p\to\infty}\frac{1}{|A_{k_{l_{p}}}|}\sum_{j\in A_{k_{l_{p}}}}|\ip{\phi_{k_{l_{p}}}(j),\alpha}|=\varepsilon.\]
Applying Corollary \ref{C:messy corollary}, and passing to a further subsequence we may, and will, assume that \[\lim_{p\to\infty}\frac{1}{d_{k_{l_{p}}}}\sum_{j=1}^{d_{k_{l_{p}}}}\rho(\phi_{k_{l_{p}}}(j),Y_{k_{l_{p}}}(j))=0.\] Choose a $\delta>0$ so that $x\in X$ and $\rho(x,\{\alpha\}^{o})<\delta$ implies that $|\ip{x,\alpha}|<\varespilon/2.$ Let $U=\{K\in \Sub(X):K\subseteq N_{\delta/2}(\{\alpha\}^{o})\},$ then $U$ is an open neighborhood of $E.$ So for all large $p$ and all $j\in A_{k_{l_{p}}}$ we have that $Y_{k_{l_{p}}}(j)\subseteq N_{\delta/2}(\{\alpha\}^{o}).$ Since $\liminf_{l\to\infty}u_{d_{k_{l}}}(A_{k_{l}})>0,$ and $\lim_{p\to\infty}u_{d_{k}}(\{j:\rho(\phi_{k_{l_{p}}}(j),Y_{k_{l_{p}}}(j))<\delta/2\})=1,$ we have:
\[\varepsilon=\lim_{p\to\infty}\frac{1}{|A_{k_{l_{p}}}|}\sum_{j\in A_{k_{l_{p}}}}|\ip{\phi_{k_{l_{p}}}(j),\alpha}|=\lim_{p\to\infty}\frac{1}{|A_{k_{l_{p}}}|}\sum_{j\in A_{k_{l_{p}}}:\rho(\phi_{k_{l_{p}}}(j),Y_{k_{l_{p}}}(j))<\delta/2}|\ip{\phi_{k_{l_{p}}}(j),\alpha}|\leq \varepsilon/2,\]
a contradiction.

\end{proof}

The statement of Corollary \ref{C:messy corollary} can be drastically simplified when the sofic approximation has an ergodic commutant.

\begin{cor}\label{C: ultraftiler free ergodic commutant}
Let $G$ be a countable, discrete, sofic group with sofic approximation $\sigma_{k}\colon G\to S_{d_{k}}.$Suppose that $(\sigma_{k})_{k}$ has ergodic commutant. Fix an algebraic action $G\actson X,$ and let $\rho$ be a compatible metric on $X.$ Then the following are equivalent:
\begin{enumerate}[(i)]
\item There does not exist a sequence $\mu_{k}\in \Prob(X^{d_{k}})$ with $\mu_{k}\to^{lw^{*}}m_{X}$ as $k\to\infty.$ \label{I:classical nonexistence ergodic commutant}
\item There is a proper, closed, $G$-invariant subgroup $Y$ of $X$ with the following property. Given any sequence $(\phi_{k})_{k}\in \prod_{k}X^{d_{k}}$ of topological microstates, there is a strictly increasing sequence $(k_{l})_{l}$ of natural numbers which satisfy
\[\lim_{n\to\infty}\rho_{2}(\phi_{k_{l}},Y^{d_{k_{l}}})=0.\]
\label{I:messy nonexistence subsequences ergodic commutant}
\end{enumerate}
\end{cor}

\begin{proof}

(\ref{I:classical nonexistence ergodic commutant}) implies (\ref{I:messy nonexistence subsequences ergodic commutant}): This can be argued exactly as in Corollary \ref{C:messy corollary}, using Corollary \ref{C:simpler under ergodicity}.

 (\ref{I:messy nonexistence subsequences ergodic commutant}) implies (\ref{I:classical nonexistence ergodic commutant}): Assume that (\ref{I:classical nonexistence ergodic commutant}) holds, but that there is a sequence $\mu_{k}\in \Prob(X^{d_{k}})$ which is asymptotically supported on topological microstates and which satisfies $\mu_{k}\to^{lw^{*}}m_{X}.$ Let $Y$ be as in the hypothesis of (\ref{I:messy nonexistence subsequences ergodic commutant}).   Choose $\delta>0$ so that $\overline{N_{\delta}(Y)}\ne X,$ and set $c=m_{X}(\overline{N_{\delta}(Y)})<1.$ Observe that for any $k\in \N,$
\[\int u_{d_{k}}(\{j:\phi_{k}(j)\in \overline{N_{\delta}(Y)}\})\,d\mu_{k}(\phi)=\frac{1}{d_{k}}\sum_{j=1}^{d_{k}}\mu_{k}(\overline{N_{\delta}(Y)}).\]
Let $F_{n}$ be an increasing sequence of finite subsets of $G$ with $G=\bigcup_{n=1}^{\infty}F_{n}.$ Since $\mu_{k}\to^{lw^{*}}m_{X}$ and is asymptotically supported on topological microstates, we may choose an increasing sequence of integers $K_{1},K_{2},\cdots$ so that
\[\int u_{d_{k}}(\{j:\phi(j)\in \overline{N_{\delta}(Y)}\})\,d\mu_{k}(\phi)\leq (1-2^{-m})c.\]
\[\mu_{k}\left(\bigcap_{g\in F_{m}}\{\phi:\rho_{2}(g\phi_{k},\phi_{k}\circ \sigma_{k}(g))\leq 2^{-m}\}\right)\geq 1-2^{-m} \mbox{ for all $m\in \N,k\geq K_{m}.$}\]
We thus have for every $m\in \N,$ and $k\geq K_{m}$ that
\[\mu_{k}(\{\phi:u_{d_{k}}(\{j:\phi(j)\in\overline{N_{\delta}(Y)}\}\})\geq \frac{1+c}{2}\})\leq \frac{2c}{1+c}(1-2^{-m}).\]
Since $c<1,$ it follows that $\frac{2c}{1+c}(1-2^{-m})<1-2^{-m}$ for all $m\geq 1.$

Now fix a $k\geq K_{1},$ and choose $m\geq 1$ so that $K_{m}\leq k<K_{m+1}.$ By the above, we may find a $\phi_{k}\in X^{d_{k}}$ so that
\begin{itemize}
\item $u_{d_{k}}(\{j:\phi_{k}(j)\in N_{\delta}(Y)\})\leq \frac{1+c}{2}$ and
\item  $\rho_{2}(g\phi_{k},\phi_{k}\circ \sigma_{k}(g))\leq 2^{-m}$ for all $g\in F_{m}.$
\end{itemize}
For a natural number $k<k_{1},$ we set $\phi_{k}=1.$ Then $(\phi_{k})_{k}$ is a sequence of topological microstates. Let $(k_{l})_{l}$ be an increasing sequence of natural numbers as in the conclusion to (\ref{I:messy nonexistence subsequences ergodic commutant}). Then for all sufficiently large $l,$ we have that
\[\rho_{2}(\phi_{k_{l}}(j),Y^{d_{k_{l}}})^{2}\geq \delta^{2}u_{d_{k_{l}}}(\{j:\phi_{k_{l}}(j)\notin N_{\delta}(Y)\})\geq \delta^{2}\frac{1-c}{2}.\]
Since $\frac{1-c}{2}>0,$ this contradicts the assumption that $\rho_{2}(\phi_{k_{l}}(j),Y^{d_{k_{l}}})\to 0$ as $l\to\infty.$

\end{proof}

\subsection{The case when the compact group is abelian and $(\sigma_{k})_{k}$ has ergodic commutant}\label{S:ergodic commutant sketch}.

In this section, we sketch an ultrafilter-free proof of Corollary \ref{C: ultraftiler free ergodic commutant} when $X$ is abelian. The proof that (\ref{I:lw* existence abelian}) implies (\ref{I:messy nonexistence subsequences ergodic commutant}) in Corollary \ref{C: ultraftiler free ergodic commutant}  did not use ultrafilters. So we focus on proving that (\ref{I:messy nonexistence subsequences ergodic commutant}) implies (\ref{I:lw* existence abelian}) in Corollary \ref{C: ultraftiler free ergodic commutant}. We do this by proving the contrapositive. So suppose that (\ref{I:messy nonexistence subsequences ergodic commutant}) is false. Then for every proper, closed subgroup $Y\leq X,$ there is a sequence $(\phi_{k})_{k}$ of topological microstates so that $\liminf_{k\to\infty}u_{d_{k}}(\phi_{k},Y^{d_{k}})>0.$ For each $\alpha\in \widehat{X}\setminus\{0\},$ we apply this to $Y=\{\alpha\}^{o}$ to deduce the following:

\emph{Fact: For every $\alpha\in\widehat{X}\setminus\{0\},$ there is a sequence $(\phi_{k})_{k}$ of topological microstates for $G\actson X$ and a constant $c\in (0,1]$ so that}
\[\liminf_{k\to\infty}u_{d_{k}}(\{j:|\ip{\phi_{k}(j),\alpha}|\geq c\})>0.\]

We use this to prove the following:

\emph{Claim: For every $\alpha\in \widehat{X}\setminus\{0\},$ and every $\varepsilon>0,$ there is a sequence $(\mu_{k})_{k}$ of measures supported on topological microstates so that for every $\delta>0$ one has $\liminf_{k\to\infty}u_{d_{k}}(\{j:|\widehat{\mu}_{k,j}(\alpha)|<\delta\})>1-\varepsilon.$ }

To see how the fact implies the claim, fix $\varespilon>0,$ and $\alpha\in \widehat{X}\setminus\{0\}$ be given. Let $(\phi_{k})_{k},c$ be as in the fact. Let $A_{k}=\{j:|\ip{\phi_{k}(j),\alpha}|\geq c\}.$ We may choose an $r\in \N,$ and $(\tau_{1,k})_{k},\cdots,(\tau_{r,k})_{k}\in \mathcal{G}'$ so that $\liminf_{k\to\infty}u_{d_{k}}\left(\bigcup_{j=1}^{r}\tau_{j,k}(A_{k})\right)\geq 1-\varepsilon.$ Set $E_{k}=\left(\bigcup_{j=1}^{r}\tau_{j,k}(A_{k})\right).$  Let $\nu_{k}=\frac{1}{r}\sum_{j=1}^{r}\left[\frac{1}{2}\delta_{\phi_{k}\circ \tau_{k}}+\frac{1}{2}\delta_{0}\right].$ It is then not hard to show that
\[\limsup_{k\to\infty}\sup_{j\in E_{k}}|\widehat{\nu}_{k,j}(\alpha)|<1.\]
By continuity of the addition map $X\times X\to X,$ we may chosen a sequence of integers $m_{k}$ with $m_{k}\to \infty$ sufficiently slowly so that $\nu_{k}^{*m_{k}}$ is still asymptotically supported on topological microstates. If we set $\mu_{k}=\nu_{k}^{*m_{k}},$ then one can check $\mu_{k}$ satisfies the conclusion of the claim.

From the claim, and a diagonal argument, for every $\alpha\in \widehat{X}\setminus\{0\}$ we may find a sequence $(\mu_{k}^{(\alpha)})_{k}\in \prod_{k}\Prob(X^{d_{k}})$ so that for every $\delta>0$ we have
\[\lim_{k\to\infty}u_{d_{k}}(\{j:|\widehat{\mu}_{k,j}^{(\alpha)}(\alpha)|\leq \delta\})=1.\]

Now let $(\alpha_{l})_{l=1}{\infty}$ be a fixed enumeration of $\widehat{X}.$ It is then not hard to show that we may choose a sequence $L(k)$ of integers with $L(k)\to\infty$ sufficiently slowly so that
\[\mu_{k}=\mu_{k}^{(\alpha_{l(1)})}*\mu_{k}^{(\alpha_{l(2)})}*\cdots*\mu_{k}^{(\alpha_{L(k)})}\]
is asymptotically supported on topological microstates, and satisfies $\mu_{k}\to^{lw^{*}}m_{X}.$

\appendix

\section{Loeb Measure Space and Preliminaries}\label{A:Loeb Measures}

\begin{prop}\label{P:loeb basics yo}
Let $X$ be a compact, metrizable space, $(d_{k})_{k}$ a sequence of natural numbers and let $\omega$ be a free ultrafilter on the natural numbers.
\begin{enumerate}[(i)]
\item \label{I:borelmaps and shiz} Let $(\phi_{k})_{k}\in\prod_{k}X^{d_{k}}.$ Then $(\phi_{k})_{k\to\omega}$ is Borel.
\item \label{I:internal sets} Fix a sequence $(E_{k})_{k}$ where $E_{k}\subseteq X^{d_{k}},$ and let
\[E=\left\{(\phi_{k})_{k\to\omega}:(\phi_{k})_{k}\in\prod_{k}E_{k}\right\}.\]
Then for any compatible continuous metric $\rho$ on $X,$ we have that $E$ is a closed subset of $\Meas(Z_{\omega},X)$ with respect to the metric $\rho_{m}.$
\item  \label{I:sequences and aljsakljal} Given a $\Phi\in \Meas(Z_{\omega},X),$ there is a sequence $(\phi_{k})_{k}\in\prod_{k}X^{d_{k}}$ so that $\Phi=(\phi_{k})_{k\to\omega}$ almost everywhere.
\end{enumerate}
\end{prop}

\begin{proof}

Throughout we fix a compatible metric $\rho$ on $X,$ and let $M$ be the diameter of $(X,\rho).$

(\ref{I:borelmaps and shiz}): Set $\Phi=(\phi_{k})_{k\to\omega}$. Then, for every $x\in X,$ we have that
\[\Phi^{-1}(B(x,\varepsilon))=\bigcup_{n=1}^{\infty}\prod_{k\to\omega}\phi_{k}^{-1}(B(x,\varepsilon-1/n)).\]
Since $\prod_{k\to\omega}\phi_{k}^{-1}(B(x,\varepsilon-1/n))$ is measurable by definition, we see that $\Phi$ is Borel.

(\ref{I:internal sets}): It suffices to show that $E$ is complete in the metric $\rho_{m}.$ Let $(\Phi_{n})_{n}$ be a Cauchy sequence in $E.$ Without loss of generality, we may assume that $\rho_{m}(\Phi_{n},\Phi_{n+1})<2^{-4n}$ for every $n\in\N.$ Thus for every natural number $n$ we have
\[\mu(\{z:\rho(\Phi_{n}(z),\Phi_{n+1}(z))>2^{-2n}\})<2^{-2n}.\]
Write $\Phi_{n}=(\phi_{n,k})_{k\to\omega}.$ Fix $x_{0}\in X$ and define for all $k\in \N$ a map $\phi_{0,k}\colon\{1,\dots,d_{k}\}\to X$ by $\phi_{0,k}(j)=x_{0}.$
We may choose a decreasing sequence $B_{n}$ of elements of $\omega$ so that
\begin{itemize}
\item $B_{n}\cap\{1,\dots,n\}=\varnothing,$
\item $u_{d_{k}}(\{j:\rho(\phi_{l,k}(j),\rho_{l+1,k}(j))>2^{-l}\})<2^{-l}$ for all $1\leq l\leq n.$
\end{itemize}
Set $B_{0}=\N\setminus B_{1}.$ For $k\in \N,$ let $n(k)\in \N\cup\{0\}$ be defined by $k\in B_{n(k)}\setminus B_{n(k)+1},$ and set $\Phi=(\phi_{n(k)+1,k})_{k\to\omega}.$ Fix an $n\in\N,$ and a $k\in B_{n}.$ Let
\[A_{k}=\bigcap_{l=n}^{n(k)}\{1\leq j\leq d_{k}: \rho(\phi_{l,k}(j),\phi_{l+1,k})(j))\leq 2^{-l}\},\]
then $u_{d_{k}}(A_{n}^{c})\leq 2^{-n+1},$ and for all $j\in A_{n}$ we have $\rho(\phi_{n(k)+1,k}(j),\phi_{n,k}(j))\leq 2^{-n+1}.$ Hence we have that $\mu(\{z:\rho(\Phi(z),\Phi_{n}(z))\leq 2^{-n+1}\})\geq 1-2^{-n+1}.$ Thus,
\[\rho_{m}(\Phi,\Phi_{n})\leq 2^{-n+1}(1+M)\]
and so we have that $\Phi_{n}\to \Phi.$ Since $\Phi\in E,$ we have shown that $E$ is complete.

(\ref{I:sequences and aljsakljal}):
Let $\mathcal{Y}=\{(\phi_{k})_{k\to\omega}:(\phi_{k})_{k}\in \prod_{k}X^{d_{k}}\}.$
By (\ref{I:internal sets}), we know that $\mathcal{Y}$ is closed in $\Meas(Z_{\omega},X),$ so it suffices to show that $\mathcal{Y}$ is dense in $\Meas(Z_{\omega},X).$
To prove this, fix a $\Psi\in \Meas(Z_{\omega},X)$ and an $\varpesilon>0.$ Let $(x_{n})_{n}$ be a dense sequence in $X.$ For each $n\in \N,$ set
\[B_{n}=B(x_{n},\varepsilon)\cap \bigcap_{l=1}^{n}B(x_{l},\varepsilon)^{c}.\]
For each $n\in \N,$ choose a sequence $(E_{n,k})_{k}$ with $E_{n,k}\subseteq\{1,\dots,d_{k}\}$ so that
\[u_{\omega}\left(\Psi^{-1}(B_{n})\Delta \prod_{k\to\omega}E_{n,k}\right)=0.\]
By replacing $E_{n,k}$ with $E_{n,k}\cap \left(\bigcap_{l=1}^{n}E_{l,k}^{c}\right),$ we may (and will) assume that for every $k\in \N,$ the family $(E_{n,k})_{n}$ is disjoint. Fix an $\N$ large enough so that
\[u_{\omega}\left(\bigcup_{n=N+1}^{\infty}\Psi^{-1}(B_{n})\right)< \varepsilon/2.\]
We may then choose a $B\in \omega$ so that for all $k\in B$ we have
\[u_{d_{k}}\left(\bigcup_{n=1}^{N}E_{n,k}\right)>1-\varepsilon/2.\]
For each $k\in B,$ define $\phi_{k}\colon \{1,\dots,d_{k}\}\to X$ by $\phi_{k}\big|_{E_{l,k}}=x_{l}$ for all $l=1,\dots,N$ and $\phi_{k}\big|_{\bigcap_{l=1}^{N}E_{l,k}^{c}}=x_{1}.$ Set $\Phi=(\phi_{k})_{k\to\omega}.$ For every $x\in B_{n}\cap \prod_{k\to\omega}E_{n,k},$ we have that $\Phi(z)\in B_{n}$ and $\Psi(z)=x_{n}$ and so $\rho(\Phi(z),\Psi(z))<\varepsilon.$  Hence our choice of $E_{n,k},B_{n}$ imply that $u_{\omega}(\{z:\rho(\Phi(z),\Psi(z))>\varepsilon\})\leq \varpesilon/2<\varpesilon.$ We thus have that $\rho_{m}(\Phi,\Psi)\leq (1+M)\varepsilon.$ Since $\Psi\in \mathcal{Y},$ we clearly have that $\mathcal{Y}$ is dense.

\end{proof}

Suppose $(d_{k})_{k},\omega$ are as in Proposition \ref{P:loeb basics yo}, given a sequence $(E_{k})_{k}$ with $E_{k}\subseteq X^{d_{k}},$ we let $\prod_{k\to\omega}E_{k}$ be the subset of $\Meas(Z_{\omega},X)$ defined by
\[\prod_{k\to\omega}E_{k}=\left\{(\phi_{k})_{k\to\omega}:(\phi_{k})_{k}\in \prod_{k}E_{k}\right\}.\]
We call set of this form \emph{internal} subsets of $\Meas(Z_{\omega},X).$ We have the following generalization of Proposition \ref{P:loeb basics yo} (\ref{I:internal sets}).

\begin{prop}\label{PA:countable internal sets}
Let $(d_{k})_{k}$ be a sequence of natural numbers, and $\omega$ be a free ultrafilter on the natural numbers. Suppose that $E\subseteq\Meas(Z_{\omega},X)$ is a countable intersection of internal sets. Then $f_{*}(E)$ is a closed subset of $\Meas(Z_{\omega},Y).$

\end{prop}

\begin{proof}
Write $E=\bigcap_{l=1}^{\infty}E_{l}$ where $E_{l}$ is internal. For each $l\in \N,$ write $E_{l}=\prod_{k\to\omega}E_{l,k}$ where $E_{l,k}$ are subsets of $X^{d_{k}}.$ Suppose $y^{(n)}\in f_{*}(E)$ and $y^{(n)}\to y.$ Without loss of generality we may, and will, assume that $u_{\omega}(\{z:\rho(y^{(n)},y)<2^{-n}\})>1-2^{-n}.$  For each $n\in \N,$ write $y^{(n)}=f_{*}(x^{(n)})$ for $x^{(n)}\in E.$ For each $n,l\in \N,$ we may then write $x^{(n)}=(x^{(n)}_{l,k})_{k\to\omega}$ for some $x^{(n)}_{l,k}\in E_{l,k}.$ Write $y=(y_{k})_{k\to\omega},$ with $y_{k}\in Y^{d_{k}}.$

We may now choose a decreasing sequence $(B_{n})_{n}$ of subsets of $\N$ so that:
\begin{itemize}
\item $B_{n}\in\omega$ for all $n\in \N,$
\item $B_{n}\cap\{1,\cdots,n\}=\varnothing,$
\item for all $k\in B_{n},$ and all $1\leq l,s\leq n,$ we have $u_{d_{k}}(\{j:\rho(x^{(s)}_{1,k}(j),x^{(s)}_{l,k})<2^{-n}\})>1-2^{-n},$
\item for all $k\in B_{n},$ we have that $u_{d_{k}}(\{j:\rho(f(x^{(n)}_{1,k}),y_{k})<2^{-n}\})>1-2^{-n}.$
\end{itemize}

For $k\in B_{1},$ let $n(k)$ be such that $k\in B_{n}\cap B_{n+1}^{c}.$ Now set $x^{(\infty)}=(x^{(n(k))}_{1,k})_{k\to\omega}.$ It is straightforward to check that for all $l\in \N,$ we have that $x^{(\infty)}=(x^{(n(k))}_{l,k})_{k\to\omega},$ so $x^{(\infty)}\in \bigcap_{l=1}^{\infty}E_{l}.$ By construction, we have that $f_{*}(x^{(\infty)})=y,$ so $y\in f_{*}(E).$ Thus we see that $f_{*}(E)$ is closed.

\end{proof}

\begin{prop}\label{P:computing integrals in ultraproducts}
Let $(d_{k})_{k}$ be a sequence of natural numbers, and $\omega$ a free ultrafilter on the natural numbers. Suppose $R\in [0,\infty)$ and $f\in L^{\infty}(Z_{\omega},u_{\omega})$ and $(f_{k})_{k}\in \prod_{k}\ell^{\infty}(d_{k})$ has $\|f_{k}\|_{\infty}\leq R$ for all $k$ and $f=(f_{k})_{k\to\omega}$ almost everywhere. Then
\begin{enumerate}[(a)]
\item \label{I:basic integral computation}
\[\int f(z)\,du_{\omega}(z)=\lim_{k\to\omega}\frac{1}{d_{k}}\sum_{j=1}^{d_{k}}f_{k}(j).\]

\item \label{I:basic pushforward fact}
We have that $f_{*}(u_{\omega})=\lim_{k\to\omega}(f_{k})_{*}(u_{\omega})$ where the limit is taken in the weak$^{*}$ topology on probability measures on $[-R,R].$

\end{enumerate}

\end{prop}

\begin{proof}

(\ref{I:basic integral computation}): This is tautological if $f=1_{E},$ and $1_{E}=(1_{E_{k}})_{k\to\omega}$ almost everywhere for a sequence $(E_{k})_{k}$ of subsets of $\{1,\dots,d_{k}\}.$ Now suppose that $f,(f_{k})_{k}$ are as in the statement of the proposition. Given $\varepsilon>0,$ we may find a measurable simple function $g$ so that $\|f-g\|_{\infty}<\varepsilon.$ Since every measurable subset of $Z_{\omega}$ is almost everywhere equal to an internal set, we may assume that $g=\sum_{j=1}^{n}c_{j}1_{E_{j}},$ where $E_{j}$ is an internal subset of $Z_{\omega}.$ We may thus write $E_{j}=(E_{jk})_{k\to\omega}$ for a sequence $(E_{jk})_{k}$ where $E_{jk}\subseteq \{1,\dots,d_{k}\}.$ Let $g_{k}=\sum_{j=1}^{n}c_{j}1_{E_{jk}}.$

From the case of internal sets, we know that
\[\lim_{k\to\omega}\frac{1}{d_{k}}\sum_{j=1}^{d_{k}}g_{k}(j)=\int g(z)\,du_{\omega}(z).\]
Let $F_{k}=\{1\leq j\leq d_{k}:|f_{k}(j)-g_{k}(j)|<\varepsilon\}.$ By definition of the Loeb measure, we have that
\[\lim_{k\to\omega}u_{d_{k}}(F_{k})\leq u_{\omega}(\{z:|f(z)-g(z)|\leq\varepsilon\})=1.\]
From this, and the fact that we have a uniform bound on the $\ell^{\infty}$-norms of $f_{k},g_{k},$ it follows easily that
\[\lim_{k\to\omega}\frac{1}{d_{k}}\sum_{j=1}^{d_{k}}|f_{k}(j)-g_{k}(j)|\leq \varepsilon.\]
Hence we have that
\[\left|\int f(z)\,du_{\omega}(z)-\lim_{k\to\omega}\frac{1}{d_{k}}\sum_{j=1}^{d_{k}}f_{k}(j)\right|\leq \varepsilon+\|f-g\|_{1}<2\varepsilon.\]
Since $\varepsilon$ is arbitrary, this completes  the proof.

(\ref{I:basic pushforward fact}): Fix a continuous function $g\colon D_{R}\to \C,$ where $D_{R}=\{z\in \C:|z|\leq R\}.$ Chasing the definitions shows that $g\circ f=g_{*}(f)=g_{*}((f_{k})_{k\to\omega})=(g\circ f_{k})_{k\to\omega}$. Thus by part (\ref{I:basic integral computation}),
\[\int g\,df_{*}(u_{\omega})=\int g\circ f\,du_{\omega}=\lim_{k\to\omega}\frac{1}{d_{k}}\sum_{j=1}^{d_{k}}g(f_{k}(j))=\lim_{k\to\omega}\int g\,d(f_{k})_{*}(u_{d_{k}}).\]
By definition, this means that $\lim_{k\to\omega}(f_{k})_{*}(u_{d_{k}})=f_{*}(u_{\omega}).$

\end{proof}

\section{Application of the methods to local and empirical convergence}\label{A:empirical limit}

In this section, we consider ``generalized local and empirical limits" as an analogue of our space of ``generalized local weak$^{*}$-limits." Recall that if $G$ is a countable, discrete, sofic group with sofic approximation $\sigma_{k}\colon G\to S_{d_{k}},$ and $X$ is a compact, metrizable space with $G\actson X$ by homeomorphisms, and $\mu\in \Prob_{G}(X),$ then a sequence of measures $(\mu_{k})_{k}\in \prod_{k}X^{d_{k}}$ locally and empirically converges to $\mu$ if it is asymptotically supported on measure-theoretic microstates for $G\actson (X,\mu)$ as $k\to\infty,$ and if $\mu_{k}\to^{lw^{*}}\mu.$ It is obvious how to modify this definition to say that $\mu_{k}$ locally and empirically converges to $\mu$ as $k\to\omega$ for a free ultrafilter $\omega$ on the natural numbers. From here it is obvious how to modify our definition of a space of ``generalized local and empirical limits."

\begin{defn}
Let $G$ be a countable, discrete, sofic group and $\sigma_{k}\colon G\to S_{d_{k}}$ a sofic approximation. Let $X$ be a compact, metrizable space with $G\actson X$ by homeomorphisms, and fix $\mu\in \Prob_{G}(X).$  Fix a free ultrafilter on the natural numbers. We let $\mathcal{LE}_{\omega}(G\actson (X,\mu))$ be the set of $\mathcal{E}((\mu_{k})_{k})$ where $\mu_{k}$ is asymptotically supported on measure-theoretic microstates for $G\actson (X,\mu)$ as $k\to\omega.$

\end{defn}

The following may be argued exactly as in Theorem \ref{T:topology is helpful}.
\begin{lem}\label{L:le limits appendix}
Let $G$ be a countable, discrete, sofic group with sofic approximation $\sigma_{k}\colon G\to S_{d_{k}}.$ Let $X$ be compact, metrizable space with $G\actson X$ by homeomorphisms, and fix $\mu\in \Prob_{G}(X).$ Then $\mathcal{LE}_{\omega}(G\actson X)$ is a closed subset of $\Meas(Z_{\omega},\Prob(X))$ for every free ultrafilter $\omega$ on the natural numbers. In fact, $\mathcal{LE}_{\omega}(G\actson X)$ is a countable intersection of internal sets.

\end{lem}

\begin{thm}\label{T:le limits appendix}
Let $G$ be a countable, discrete, sofic group with sofic approximation $\sigma_{k}\colon G\to S_{d_{k}}.$ Let $X$ be compact, metrizable space with $G\actson X$ by homeomorphisms, and fix $\mu\in \Prob_{G}(X).$  Suppose that $G'_{\omega}\actson (Z_{\omega},u_{\omega})$ is ergodic. If $G\actson (X,\mu)$ is sofic with respect to $(\sigma_{k})_{k},\omega,$ then there is a sequence of measures $\mu_{k}$ so that $\mu_{k}\to^{le}_{k\to\omega}\mu.$

\end{thm}

\begin{proof}
Since $\mathcal{LE}_{\omega}(G\actson (X,\mu))$ is closed inside $\Meas(Z_{\omega},\Prob(X)),$ it suffices to prove the following claim:

\emph{Claim. For every open neighborhood $\mathcal{O}$ of $\mu$ in the weak$^{*}$-topology, there is a $\nu_{\mathcal{O}}\in \mathcal{LE}_{\omega}(G\actson (X,\mu))$ so that $\nu_{\mathcal{O}}(z)\in \mathcal{O}$ for almost every $z\in Z_{\omega}.$}

So fix an open neighborhood $\mathcal{O}$ of $\mu.$ We may then find $f_{1},\cdots,f_{n}\in C(X)$ and an $\varepsilon>0$ so that
\[\mathcal{O}\supseteq \bigcap_{j=1}^{n}\left\{\nu\in \Prob(X):\left|\int f_{j}\,d\mu-\int f_{j}\,d\nu\right|<\varepsilon\right\}.\]
For $g\in C(X),$ define $I_{g}\colon \Prob(X)\to \C$ by $I_{g}(\nu)=\int g\,d\nu.$ Let $f=(f_{1},\cdots,f_{n})\in C(X)^{n},$ and define $I_{f}\colon \Prob(X)\to \C^{n}$ by $I_{f}(\nu)=(I_{f_{j}}(\nu))_{j=1}^{n}.$ Let $M=\max_{1\leq j\leq n}\|f_{j}\|,$ and let $D_{M}=\{z\in \C:|z|\leq M\}.$

By Proposition \ref{PA:countable internal sets} and Lemma \ref{L:le limits appendix}, we know that $(I_{f})_{*}(\mathcal{LE}_{\omega}(G\actson (X,\mu)))$ is a closed subset of $\Meas(Z_{\omega},D_{M}^{ n}).$ Let $\iota\colon \Meas(Z_{\omega},D_{M}^{n})\to L^{2}(Z_{\omega})^{\oplus n}$ be the natural inclusion map given by $\iota(k)=(\pi_{j}\circ k)_{j=1}^{n}$ where $ \pi_{j}\colon D_{M}^{n}\to D_{M}$ is the projection onto the $j^{th}$ coordinate. Set $K=\iota((I_{f})_{*}(\mathcal{LE}_{\omega}(G\actson (X,\mu))).$ It is not hard to show that $\iota$ is a homeomorphism onto its image, and that its image is closed. Since $(I_{f})_{*}(\mathcal{LE}_{\omega}(G\actson (X,\mu)))$ is a closed subset of $\Meas(Z_{\omega},D_{M}^{ n}),$ we see that $K$ is a closed subset of $L^{2}(Z_{\omega})^{\oplus n}.$ Since $(I_{f})_{*}$ is affine, we know that $K$ is convex.

So $K$ is a closed, convex subset of $L^{2}(Z_{\omega})^{\oplus n}.$ Thus there is a unique element $\lambda\in (I_{f})_{*}(\mathcal{LE}_{\omega}(G\actson (X,\mu)))$ of minimal $\|\cdot\|_{2}$-norm. Since $\mathcal{LE}_{\omega}(G\actson (X,\mu))$ is clearly $G'$-invariant, we know that $\lambda$ is $G'$-fixed, by uniqueness. By ergodicity, we have that $\lambda\in (\C 1)^{\oplus n}.$ So we may write $\lambda=(\lambda_{j} 1)_{j=1}^{n}$ for some $\lambda_{j}\in \C,1\leq j\leq n.$ Write $\lambda=(I_{f})_{*}(\nu_{\mathcal{O}})$ for some $\nu_{\mathcal{O}}\in \mathcal{LE}_{\omega}(G\actson (X,\mu)).$ Then
\[\lambda_{j}=\ip{\lambda_{j}1,1}=\ip{(I_{f_{j}})_{*}(\nu_{\mathcal{O}}),1}=\int (I_{f_{j}})_{*}(\nu_{\mathcal{O}})\,du_{\omega}(z).\]
Write $\nu_{\mathcal{O}}=\mathcal{E}((\nu_{k})_{k}),$ where $\nu_{k}$ is asymptotically supported on measure-theoretic microstates as $k\to\omega.$ Then:
\[\int (I_{f_{j}})_{*}(\nu_{\mathcal{O}})\,du_{\omega}(z)=\lim_{k\to\omega}\frac{1}{d_{k}}\sum_{l=1}^{d_{k}}\int f_{j}(\phi(l)),d\nu_{k,l}(\phi)=\lim_{k\to\omega}\int \int f_{j}\,d\phi_{*}(u_{d_{k}})\,d\nu_{k}(\phi).\]
Since $\nu_{k}$ is asymptotically supported on topological microstates, the above formula shows that $\int (I_{f_{j}})_{*}(\nu_{\mathcal{O}})\,du_{\omega}(z)=\int f_{j}\,d\mu.$ Thus $\lambda_{j}=\int f_{j},d\mu.$ Since $\lambda_{j}1=I_{f_{j}}(\nu_{\mathcal{O}}),$ this implies that $\int f_{j}\,d\nu_{\mathcal{O}}(z)=\int f_{j}\,d\mu$ for almost every $z\in Z_{\omega}$ and for all $1\leq j\leq n.$ This clearly implies that $\nu_{\mathcal{O}}(z)\in \mathcal{O}$ for almost every $z\in Z_{\omega},$ and this completes the proof.

\end{proof}

%
%\bibliographystyle{abbrv}
%\bibliography{research}

\end{document}